\documentclass[a4paper,leqno]{amsart}

\usepackage[alphabetic]{amsrefs}

\usepackage{amsmath}

\usepackage{mathdots} 

\usepackage{amssymb}

\usepackage{amsbsy}

\usepackage{enumerate}

\usepackage[alphabetic]{amsrefs}

\usepackage{mathptmx} 

\usepackage{hyperref}

\usepackage{mathrsfs} 

\usepackage{txfonts}  

\usepackage{frcursive}
\usepackage{calligra}
\usepackage[T1]{fontenc}

\usepackage{verbatim}

\usepackage{stmaryrd} 

\usepackage{fancyhdr}  

\usepackage{comment}

\usepackage{algorithm}

\usepackage{algpseudocode}

\usepackage{tikz}

\usetikzlibrary{arrows,positioning}

\usepackage{makecell}

\usetikzlibrary{datavisualization.formats.functions}

\usepackage{pgfplots}

\usepackage{graphicx,amssymb}

\usepackage{ifthen}

\usepackage{epsfig}

\usepackage{multirow,multicol}

\newtheorem{theo}{Theorem}

\newtheorem{prop}[theo]{Proposition}

\newtheorem{lemm}[theo]{Lemma}

\newtheorem{coro}[theo]{Corollary}

\newtheorem{defi}[theo]{Definition}

\newtheorem{rema}[theo]{Remark}

\newtheorem{prob}{Problem}

\newtheorem{exam}{Example}

\newtheorem{theoalpha}{Theorem}

\renewcommand{\le}{\leqslant}
\renewcommand{\leq}{\leqslant}

\renewcommand{\ge}{\geqslant}
\renewcommand{\geq}{\geqslant}

\newcommand{\vep}{\varepsilon}

\newcommand{\betabar}{\ensuremath{\bar{\beta}}}

\newcommand{\mbar}{\ensuremath{\bar{m}}}
\newcommand{\nbar}{\ensuremath{\bar{n}}}

\newcommand{\ubar}{\ensuremath{\bar{u}}}

\newcommand{\bH}{\ensuremath{\mathbb H}}

\newcommand{\bN}{\ensuremath{\mathbb N}}
\newcommand{\bM}{\ensuremath{\mathbb M}}

\newcommand{\bR}{\ensuremath{\mathbb R}}

\newcommand{\bZ}{\ensuremath{\mathbb Z}}

\newcommand{\sK}{\ensuremath{\mathsf K}}
\newcommand{\sL}{\ensuremath{\mathsf L}}

\newcommand{\cE}{\ensuremath{\mathcal E}}

\newcommand{\troot}{\emptyset}

\newcommand{\cdist}[1]{\mathsf{c}_{#1}}

\newcommand{\bin}[1]{\ensuremath{\mathsf{B}}_{#1}}
\newcommand{\dia}{\ensuremath{\mathsf{D}}}

\newcommand{\dX}{\ensuremath{\mathsf{d_X}}}
\newcommand{\dY}{\ensuremath{\mathsf{d_Y}}}

\newcommand{\sd}{\ensuremath{\mathsf{d}}}

\newcommand{\tree}{\ensuremath{\mathsf{T}}}

\newcommand{\lip}{\ensuremath{\mathrm{Lip}}}
\newcommand{\colip}{\ensuremath{\mathrm{coLip}}}
\newcommand{\codist}{\ensuremath{\mathrm{codist}}}

\newcommand{\ban}[1]{\mathcal{#1}}

\newcommand{\met}[1]{\mathsf{#1}}

\newcommand{\gra}[1]{\mathsf{#1}}

\newcommand{\banF}{\ban F}
\newcommand{\banX}{\ban X}
\newcommand{\banY}{\ban Y}
\newcommand{\metX}{\met X}
\newcommand{\metY}{\met Y}
\newcommand{\banXn}{(\ban X,\norm{\cdot})}

\newcommand{\metXd}{(\met X,\dX)}
\newcommand{\metYd}{(\met Y,\dY)}

\newcommand{\ie}{\textit{i.e.,}\ }
\newcommand{\eqd}{\stackrel{\mathrm{def}}{=}}

\newcommand{\Ex}{\mathbb{E}}

\newcommand{\co}{\mathrm{c}_0}

\newcommand{\xbl}{\ensuremath{\left(\sum_{n=1}^\infty\ell_{\infty}^n\right)_2}}

\newcommand{\wtree}[2]{[{#1}]^{{#2}}}

\newcommand{\kin}{\!\in\!}

\newcommand{\cube}[1]{\{-1,1\}^{#1}}

\newcommand{\abs}[1]{\lvert #1\rvert}

\newcommand{\norm}[1]{\| #1\|}
\newcommand{\bnorm}[1]{\Big\| #1\Big\|}

\renewcommand{\bH}{\mathbb{H}}

\begin{document}

\title[Umbel convexity]{Umbel convexity and the geometry of trees}

\author{F.~Baudier}
\address{Texas A\&M University, College Station, TX 77843, USA}
\email{florent@tamu.edu}

\author{C.~Gartland}
\address{Texas A\&M University, College Station, TX 77843, USA}
\curraddr{University of California San Diego, La Jolla, CA 92093, USA}
\email{cgartland@ucsd.edu}
	
\thanks{F. Baudier was partially supported by the National Science Foundation under Grant Number DMS-1800322 and DMS-2055604. C. Gartland was partially supported by the National Science Foundation under Grant Number DMS-2342644}
	


\begin{abstract}
For every $p\in(0,\infty)$, a new metric invariant called umbel $p$-convexity is introduced. The asymptotic notion of umbel convexity captures the geometry of countably branching trees, much in the same way as Markov convexity, the local invariant which inspired it, captures the geometry of bounded degree trees. Umbel convexity is used to provide a ``Poincar\'e-type" metric characterization of the class of Banach spaces that admit an equivalent norm with Rolewicz's property $(\beta)$. We explain how a relaxation of umbel $p$-convexity, called infrasup-umbel $p$-convexity, plays a role in obtaining compression rate bounds for coarse embeddings of countably branching trees. Local analogues of these invariants - fork $p$-convexity and infrasup-fork $p$-convexity - are introduced, and their relationship to Markov $p$-convexity and relaxations of the $p$-fork inequality is discussed. The metric invariants are estimated for a large class of Heisenberg groups, and in particular a parallelogram $p$-convexity inequality is proved for Heisenberg groups over $p$-uniformly convex Banach spaces. Finally, a new characterization of non-negative curvature is given.  
\end{abstract}

\maketitle

\setcounter{tocdepth}{3}
\tableofcontents

\section{Introduction}

After the discovery by Ribe \cite{Ribe76} of a striking rigidity phenomenon regarding \emph{local properties} of Banach spaces, the search for metric characterizations of local properties of Banach spaces has been a main research avenue for what would become known as the Ribe program. The Ribe program has grown into an extensive and tentacular research program with far reaching  ramifications, in particular in theoretical computer science and geometric group theory. We refer the interested reader to \cite{Ball13} and \cite{Naor12} for more information about this program. 

The foundational result of the Ribe program is a 1986 theorem of Bourgain. 
\begin{theo}\cite{Bourgain86}
A Banach space $\banY$ is super-reflexive if and only if $\sup_{k \in \bN} \cdist{\banY} (\bin{k})=\infty$.
\end{theo}

In Bourgain's metric characterization of super-reflexivity, $\{\bin{k}\}_{k\ge 1}$ is the sequence of binary trees, and the parameter $\cdist{\metY}(\metX)$ denotes the $\metY$-distortion of $\metX$ for two metric spaces $\metYd$ and $\metXd$, \ie the least constant $D$ such that there exist $s>0$ and a map $f\colon \metX \to \metY$ satisfying for all $x,y\in \metX$ 
\begin{equation*}
s\cdot \dX(x,y)\le \dY(f(x),f(y))\le sD\cdot \dX(x,y).
\end{equation*} 
An important renorming result of Enflo \cite{Enflo72} states that super-reflexivity can be characterized in terms of uniformly smooth or uniformly convex renormings. Moreover thanks to Asplund's averaging technique \cite{Asplund67}, we can equivalently consider in Bourgain's metric characterization the class of Banach spaces that admit an equivalent norm that is uniformly convex and uniformly smooth. From this perspective, an asymptotic analogue of Bourgain's metric characterization was obtained by Baudier, Kalton and Lancien in 2009.

\begin{theo}\label{thm:BKL}\cite{BKL10}
Let $\banY$ be a reflexive Banach space. Then,

$\banY$ admits an equivalent norm that is asymptotically uniformly convex and asymptotically uniformly smooth if and only if $\sup_{k \in \bN} \cdist{\banY} (\tree_k^{\omega})=\infty.$
\end{theo}

The tree $\tree_k^{\omega}$ in Theorem \ref{thm:BKL} is the countably branching version of the binary tree $\bin{k}$. The discovery of Theorem \ref{thm:BKL} launched the quest for metric characterizations of \emph{asymptotic properties} of Banach spaces. It is worth pointing out that Ribe's rigidity theorem \cite{Ribe76} provides a theoretical motivation to metrically characterize local properties of Banach spaces, but no such rigidity result is known in the asymptotic setting. Nevertheless, the asymptotic declination of Ribe program has seen some steady progress in the past decade (see for instance \cite{LimaLova12}, \cite{DKLR14}, \cite{DKR16}, \cite{BaudierZhang16}, \cite{Baudier-et-all17}, \cite{CauseyDilworth17}, \cite{DKLR17}, \cite{BLMS_IJM}, \cite{Zhang22}) with some interesting applications to coarse geometry such as in \cite{BLS_JAMS} and \cite{BLMS_JIMJ}.

Pisier's influential quantitative refinement \cite{Pisier75} of Enflo's renorming states that a super-reflexive Banach space $\banX$ admits an equivalent norm whose modulus of uniform convexity is of power type $p$ for some $p\ge 2$, or equivalently as shown in \cite{BCL94}, satisfies the following inequality for all $x,y\in \banX$ and some constant $K\ge 1$.
\begin{equation}\label{eq:p-convex}
\frac{\norm{x+y}^p+\norm{x-y}^p}{2} \ge \norm{x}^p+\frac{1}{K^p}\norm{y}^p.
\end{equation}
A Banach space whose norm satisfies \eqref{eq:p-convex} is said to be \emph{$p$-uniformly convex}. The following quantification of Bourgain's metric characterization was obtained by Mendel and Naor \cite{MendelNaor13} building upon previous work of Lee, Naor, and Peres \cite{LNP06,LNP09}.

\begin{theo}\label{thm:MN-LNP}\cite{MendelNaor13, LNP09}
A Banach space $\banX$ admits an equivalent norm that is $p$-uniformly convex if and only if $\banX$ is Markov $p$-convex.  
\end{theo}

On the metric side of the equivalence in Theorem \ref{thm:MN-LNP} is a complex inequality that captures the geometry of trees with bounded degree. According to \cite{LNP09} and given $p>0$, a metric space $\metXd$ is \emph{Markov $p$-convex} if there exists a constant $\Pi>0$ such that for every Markov chain $\{W_t\}_{t \in \bZ}$ on a state space $\Omega$ and every $f\colon \Omega \to \metX$, 
\begin{equation}\label{eq:p-Markov}
\sum_{s=0}^\infty \sum_{t \in \bZ} \frac{\Ex \big[\dX\big(f(W_t),f(\tilde{W}_t(t-2^{s})) \big)^p \big]}{2^{sp}}\le \Pi^p \sum_{t \in \bZ}\Ex \big[\dX \big( f(W_t),f(W_{t-1}) \big)^p \big],
\end{equation}
where given an integer $\tau$, $\{\tilde{W}_t(\tau)\}_{t \in \bZ}$ is the stochastic process which equals $W_t$ for time $t\le \tau$ and evolves independently, with respect to the same transition probabilities, for time $t>\tau$. The smallest constant $\Pi$ such that \eqref{eq:p-Markov} holds will be denoted by $\Pi_p^M(\metX)$.

Markov $p$-convexity is easily seen to be a bi-Lipschitz invariant, and quantitatively $\Pi_p^M(\metX)\le \cdist{\metY}(\metX)\Pi_p^M(\metY)$. The discovery of the Markov convexity inequality was partially inspired by the non-embeddability argument in Bourgain's characterization, and it thus naturally provides restrictions on the faithful embeddability of binary trees. Considering the regular random walk on the binary tree $\bin{2^k}$, it is easy to check that $\Pi^M_p(\bin{2^k}) \ge 2^{1-2/p}k^{1/p}$ and hence any bi-Lipschitz embedding of $\bin{k}$ into a Markov $p$-convex metric space incurs distortion at least $\Omega \big( (\log k)^{1/p} \big)$. This lower bound extends to the purely metric setting the lower bound obtained for $p$-uniformly convex spaces in \cite{Bourgain86}. Note also that Markov $p$-convexity is stable under taking $\ell_p$-sums of metric spaces and is preserved under Lipschitz quotient mappings \cite[Prop. 4.1]{MendelNaor13}.

Recall that a map $f \colon \metXd \to \metYd $ is a coarse embedding if there are non-decreasing maps $\rho,\omega\colon [0,\infty)\to [0,\infty)$ and $\lim_{t\to \infty} \rho(t)=\infty$ and for all $x,y\in \metX$,
\begin{equation*}
\rho(\dX(x,y))\le \dY(f(x) , f(y))\le \omega(\dX(x,y)).
\end{equation*}
The function $\rho$ (resp. $\omega$) is usually called the compression (resp. expansion) control function. We talk about equi-coarse embedding of a sequence of metric spaces if there is a sequence of coarse embeddings that are controlled uniformly by given compression and expansion functions. For graphs the expansion control function can always be assumed to be linear and the compression rate is the best compression control function that can be achieved.
In his investigation of the compression rate of coarse embeddings of groups, Tessera established the following restriction on the compression rate for equi-coarse embeddings of binary trees.

\begin{theo}\label{thm:Tessera}\cite{Tessera08}
The compression rate of any equi-coarse embedding of $\{\bin{k}\}_{k\ge 1}$ into a $p$-uniformly convex Banach space satisfies  $$\int_{1}^\infty \Big(\frac{\rho(t)}{t}\Big)^p\frac{dt}{t}<\infty.$$
\end{theo}

The proof of Theorem \ref{thm:Tessera} is another variation of Bourgain's non-embeddability argument and relies on the fact that for any $p$-uniformly convex Banach space $\banX$ there exists a constant $C>0$ such that for all $k \ge 1$ and $f\colon \bin{2^k}\to \banX$, the following refinement of an inequality implicit in \cite{Bourgain86} holds
\begin{equation}\label{eq:q-tess}
\sum_{s=0}^{k-1} \min_{2^s< \ell \le 2^k-2^s} \Ex_{\vep \in \cube{\ell}}\Ex_{\delta\in \cube{2^s}}\Ex_{\delta' \in \cube{2^s} }\frac{\norm{f(\vep,\delta)-f(\vep,\delta')}_\banX^p}{2^{sp}}\le C^p \lip(f)^p.
\end{equation} 
Here, $\lip(f)$ is the Lipschitz constant of $f$ and $\cube{h}$ is the set of vertices of $\bin{2^k}$ whose height is exactly $h$, or in other words, the vertex set of the binary tree is $\bin{2^k}:=\cup_{h=0}^{2^k} \cube{h}$ and the edge set consists of pairs of the form $\{\vep,(\vep,\delta)\}$ where $\vep \in \cube{h}$ for some $0 \leq h < 2^k$ and $\delta \in \cube{}$.
In the Banach space setting and thanks to Theorem \ref{thm:MN-LNP}, Tessera's inequality \eqref{eq:q-tess} is implied by Markov $p$-convexity. Even though not readily apparent, inequality \eqref{eq:q-tess} also follows from Markov $p$-convexity in the purely metric setting, and in turn the compression rate for the binary trees is also valid when the embedding takes values into a Markov $p$-convex metric space. We suspect this observation is known to experts and it is best seen when considering a deterministic inequality implied by Markov convexity. This fact will be properly justified in Section \ref{sec:relaxed-Markov} (cf. Remark \ref{rem:Tessera}) where we study certain relaxations of the Markov convexity inequality.

In this article we introduce new metric invariants, which are inspired by Markov convexity and inequality \eqref{eq:q-tess}, and that are crucial in resolving some problems regarding the asymptotic geometry of Banach spaces. These new inequalities share many features with their local cousins and capture the geometry of countably branching trees. The difficulty in obtaining non-trivial inequalities for countably branching trees lies in the fact that it is not clear how to make sense of the various averages over vertices when there are infinitely many of them. The strongest asymptotic metric invariant that we introduce in this article is the notion of \emph{umbel convexity}. In the definition below, $\wtree{\bN}{\le h}$ (resp. $\wtree{\bN}{h}$) denotes the set of all subsets \footnote{We will slightly abuse notation and write an element $\nbar \in \wtree{\bN}{\le h}$ as $\nbar=(n_1,n_2,\dots, n_\ell)$ where $n_1<n_2<\dots<n_\ell$ and write concisely $f(\nbar,\bar{\delta})$ instead of the more formal expression $f((n_1,\dots,n_\ell,\delta_1,\dots,\delta_{\ell'}))$ whenever the last expression makes sense.} of size at most $h$ (resp. exactly $h$) which is commonly used to code the vertex set of $\tree^\omega_{h}$, the countably branching tree of height $h$. Recall that two vertices $\mbar=(m_1,m_2,\dots, m_i)$ and $\nbar=(n_1,n_2,\dots, n_j)$ in $\tree^\omega_{h}$ belong to an edge if and only if $j=i+1$ and $m_1=n_1$, $m_2=n_2, \dots, m_i=n_i$. 
\begin{defi} Let $p\in(0,\infty)$. A metric space $\metXd$ is \emph{umbel $p$-convex} if there exists a constant $\Pi>0$ such that for all $k\ge 1$ and all $f\colon \wtree{\bN}{\le 2^k}\to \met X$,
\begin{align}\label{eq:umbel-p-convex}
\nonumber \sum_{s=1}^{k-1}\frac{1}{2^{k-1-s}}\sum_{t=1}^{2^{k-1-s}}\inf_{\nbar\in\wtree{\bN}{t2^{s+1}-2^s}}\inf_{\stackrel{\bar{\delta}\in\wtree{\bN}{2^s}\colon}{(\nbar,\bar{\delta})\in \wtree{\bN}{\le 2^k}}}\liminf_{j\to\infty}\inf_{\stackrel{\bar{\eta}\in\wtree{\bN}{2^s-1}\colon}{(\nbar,j,\bar{\eta})\in \wtree{\bN}{\le 2^k}}}\frac{\dX \big( f(\nbar,\bar{\delta}),f(\nbar,j,\bar{\eta}) \big)^p}{2^{sp}}\\
\le \Pi^p\frac{1}{2^{k}}\sum_{\ell=1}^{2^{k}}\sup_{\nbar\in \wtree{\bN}{\ell}}\dX \big( f(n_1,\dots,n_{\ell-1}),f(n_1,\dots,n_{\ell}) \big)^p.
\end{align}
The smallest constant $\Pi$ such that \eqref{eq:umbel-p-convex} holds for all $k \ge 1$ and all maps $f\colon \wtree{\bN}{\le 2^k}\to \met X$ will be denoted by $\Pi^u_p(\metX)$ and called the umbel $p$-convexity constant of $\metX$.
\end{defi}

Umbel convexity behaves in many respects as Markov convexity does, albeit with some significant and at times unavoidable discrepancies. This will be explained at length throughout the sections. Our approach to obtain the umbel convexity inequality is reminiscent of how Markov convexity can be derived from a certain $4$-point inequality. The terminology ``umbel convexity'' reflects the fact that the umbel $p$-convexity inequality follows from a certain inequality for the complete bipartite graph $\sK_{1,\omega}$ which we view pictorially as an umbel.  

\begin{figure}[H]
\label{fig:umbel}
\caption{$\sK_{1,\omega}$ in the umbel position}
\vskip 0.2cm
\hskip 0cm\includegraphics[scale=.25]{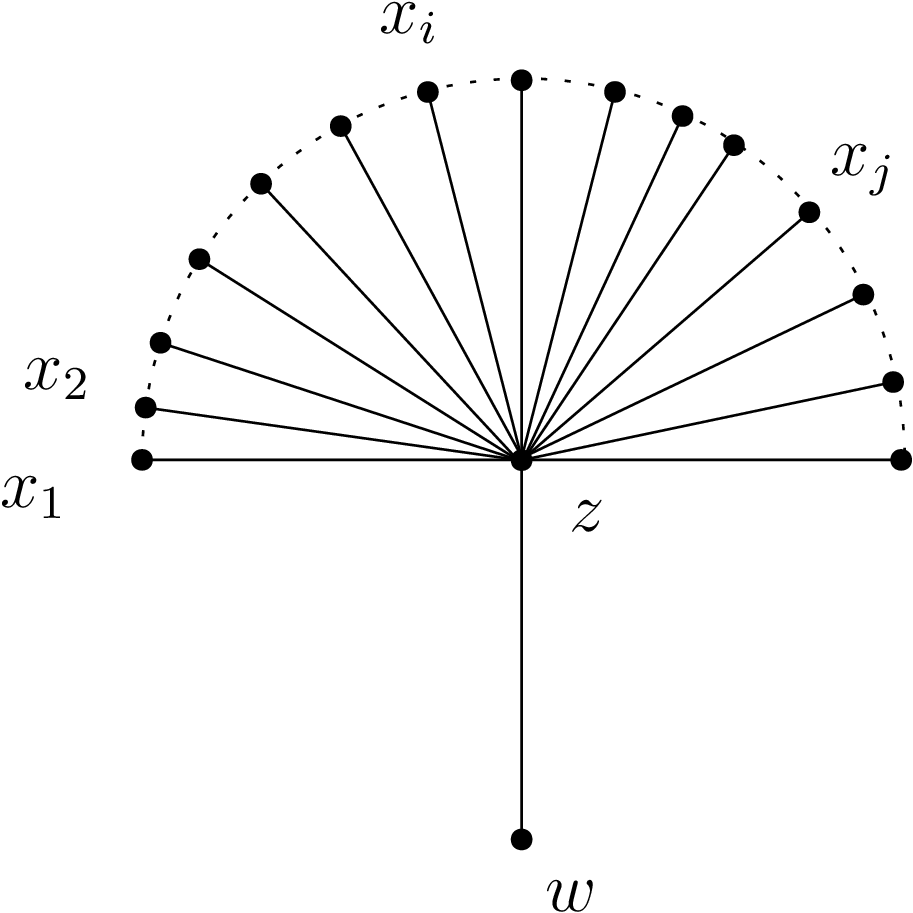}
\end{figure}

Umbel convexity plays a central role in the problem of characterizing metrically the class of Banach spaces admitting an equivalent norm with property $(\beta)$. The definition below, due to Kutzarova \cite{Kutzarova91}, is equivalent to Rolewicz's original definition \cite{Rolewicz87}. A Banach space $\banXn$ has Rolewicz's property $(\beta)$ if for all $t>0$ there exists $\bar{\beta}(t)>0$ such that for all $z\in B_\banX$ and $\{x_n\}_{n\in \bN}\subseteq B_\banX$ with $\inf_{i\neq j}\norm{x_i-x_j}\ge t$, there exists $i_0\in\bN$ so that
\begin{equation*}\label{eq:beta}
\bnorm{\frac{z-x_{i_0}}{2}}\le 1-\bar{\beta}(t).
\end{equation*} 
Moreover, $\banX$ is said to have property $(\beta)$ with power type $p>0$ and constant $c>0$ (shortened to property $(\beta_p)$) if $\bar{\beta}(t)\ge \frac{t^p}{c}$.

Property $(\beta)$ is an asymptotic generalization of uniform convexity, but it is much more than that. We denote by $\textrm{cof}(\banX)$ the set of all the finite co-dimensional subspaces of $\banX$. Recall that a Banach space $\ban X$ is \emph{asymptotically uniformly convex} if $\bar{\delta}_\banX(t)>0$ for all $t>0$, and \emph{asymptotically uniformly smooth} if $\lim_{t\to 0}\bar{\rho}_\banX(t)/t =0$, where 
\begin{equation}
\bar{\delta}_\banX(t) \eqd \inf_{\stackrel{x\in \banX}{\norm{x}=1}}\sup_{Y\in\textrm{cof}(\ban X)}\inf_{\stackrel{y\in Y}{\norm{y}=1}} \norm{x+ty}-1,
\end{equation}
and
\begin{equation}
\bar{\rho}_\banX(t) \eqd \sup_{\stackrel{x\in \banX}{\norm{x}=1}}\inf_{Y\in\textrm{cof}(\ban X)}\sup_{\stackrel{y\in Y}{\norm{y}=1}} \norm{x+ty}-1.
\end{equation}
  The following theorem follows from several important renorming results (in particular from \cite{Kutzarova90} and \cite{KOS99}) and we refer to \cite{DKLR17} for a thorough discussion.

\begin{theo}
\label{thm:renorming}
The following classes of Banach spaces coincide:
\begin{enumerate}[(i)]
\item The class $\langle (\beta) \rangle$ of Banach spaces admitting an equivalent norm with Rolewicz's property $(\beta)$.
\item The class $\langle (\beta_p) \rangle$ of Banach spaces admitting an equivalent norm with property $(\beta_p)$ for some $p\in(1,\infty)$.
\item The class of reflexive Banach spaces admitting an equivalent norm that is asymptotically uniformly convex and asymptotically uniformly smooth.
\end{enumerate}
\end{theo}

We want to emphasize a subtle point here. Theorem \ref{thm:BKL}, in combination with Theorem \ref{thm:renorming} $(iii)$, provides a metric characterization of the class $\langle (\beta) \rangle$ \emph{within} the class of reflexive Banach spaces. However, Perreau \cite{Perreau20} recently showed that $\sup_{k\ge 1}\cdist{\ban J}(\tree_k^\omega)=\infty$ where $\ban J$ is James (non-reflexive) space \cite{James51}. Therefore, the condition $\sup_{k\ge 1}\cdist{\banX}(\tree_k^\omega)=\infty$, does not necessarily force $\banX$ to be reflexive, and consequently it does not characterize the class $\langle (\beta) \rangle$. This reflexivity issue, which does not arise in the local setting, is resolved with the help of umbel convexity. 

Banach spaces with property $(\beta_p)$ are the prototypical spaces that are umbel $p$-convex (see Corollary \ref{cor:betap->umbelp} in Section \ref{sec:umbel}). The following theorem is a metric characterization of the class $\langle (\beta) \rangle$ in terms of the existence of a certain Poincar\'e-type inequality.

\begin{theoalpha}
\label{thmA}
Let $\banX$ be a Banach space. Then,
$\banX$ admits an equivalent norm with property $(\beta)$ if and only if $\banX$ is umbel $p$-convex for some $p\in(1,\infty)$.
\end{theoalpha}

While writing this article, we learned from Sheng Zhang \cite{Zhang22} that he had discovered independently a metric  characterization of the class $\langle (\beta) \rangle$ in terms of a submetric test-space in the sense of Ostrovskii \cite{Ostrovskii14b}. A similar submetric test-space characterization can be extracted with some care from the work of Dilworth, Kutzarova, and Randrianarivony in \cite{DKR16} and is also a direct consequence of our work (see Corollary \ref{cor:submetric} in Section \ref{sec:umbel}).

The delicate question of renorming a Banach space that is umbel $p$-convex will be discussed in Section \ref{sec:conclusion}. Let us just mention here that there exists an example of a Banach space constructed by Kalton in \cite{Kalton13} that is umbel $p$-convex and does not admit an equivalent norm with property $(\beta_p)$, but for every $\vep>0$ admits an equivalent norm with property $(\beta_{p+\vep})$.

The question of estimating from above compression rates for equi-coarse embeddings of the countably branching trees has remained open for a while, even for simple target spaces such as $\xbl$ for which the geometry of binary trees does not provide any obstruction. The techniques in \cite{BKL10} and \cite{BaudierZhang16} provide quantitative information about the faithful embeddability of the countably branching trees that are inherently of a bi-Lipschitz nature, and do not provide any estimates on compression rates of coarse embeddings. Umbel convexity can be used to resolve this problem. In fact, a significant relaxation of the umbel convexity inequality is sufficient for this purpose. 

\begin{defi} 
\label{def:umbelcotypep}
Let $p\in(0,\infty)$. A metric space $\metXd$ is said to be \emph{infrasup-umbel $p$-convex} if there exists a constant $C>0$ such that for all $k\ge 1$ and all $f\colon \tree^\omega_{2^k}\to \met X$,
\begin{align*}
\sum_{s=1}^{k-1}\inf_{\nbar\in\wtree{\bN}{\le 2^{k}-2^s}}\inf_{i \neq j \in \bN} \inf_{\stackrel{\mbar,\mbar' \in \wtree{\bN}{2^s-1} \colon }{(\nbar,i,\mbar),(\nbar,j,\mbar')\in \wtree{\bN}{\le 2^k}}}\frac{\sd(f(\nbar,i,\mbar),f(\nbar,j,\mbar'))^p}{2^{sp}} \le C^p \lip(f)^p.
\end{align*}
\end{defi}

It is plain that umbel $p$-convexity implies infrasup-umbel $p$-convexity. Compression rate estimates can be obtained using the notion of infrasup-umbel convexity in the same way Tessera derived Theorem \ref{thm:Tessera} from inequality \eqref{eq:q-tess}. 

\begin{theoalpha}
\label{thm:B}
Let $p\in(0,\infty)$. The compression rate of any equi-coarse embedding of $\{\tree^\omega_k\}_{k\ge 1}$ into a metric space that is infrasup-umbel $p$-convex satisfies 
\[\int_{1}^\infty \Big(\frac{\rho(t)}{t}\Big)^p\frac{dt}{t}<\infty.\]
\end{theoalpha}

The true gist of Theorem \ref{thm:B} lies in the large class of spaces it covers, since the infrasup-umbel $p$-convexity inequality follows from a significant relaxation of the inequality that was needed to prove umbel $p$-convexity. Moreover, the geometry of countably branching trees provides embeddability obstructions for spaces such as infinite-dimensional hyperbolic spaces (see Corollary \ref{cor:hyperbolic}) which cannot be achieved by merely resorting to the geometry of locally finite trees. More examples supporting these claims can be found in Section \ref{sec:examples}. A particularly interesting class of examples are Heisenberg groups over certain infinite-dimensional Banach spaces. We refer to Section \ref{sec:examples} for the definition of the Heisenberg group $(\bH(\omega_\banX),\sd_{cc})$, where $\sd_{cc}$ denotes the Carnot-Carath\'eodory metric, and the important fact that $(\bH(\omega_\banX),\sd_{cc})$ does not embed bi-Lipschitzly into a Banach space with property $(\beta_p)$.

\begin{theoalpha}
\label{thm:C}
For every non-null, antisymmetric, and bounded bilinear form $\omega_\banX$ on $\banX$ and every $p\ge 2$, the infinite-dimensional Heisenberg group $(\bH(\omega_\banX),\sd_{cc})$ is infrasup-umbel $p$-convex whenever $\banX$ has property $(\beta_p)$.\end{theoalpha}

It is natural to ask if a stronger conclusion can be achieved in Theorem \ref{thm:C}, namely if infrasup-umbel $p$-convexity can be upgraded to umbel $p$-convexity. We do not know if this stronger conclusion holds, and we discuss the issue further following Problem~\ref{prob:Heisenbergumbel}.

Theorem \ref{thm:C} is in stark contrast with the situation in the local theory, as it was shown by S. Li in \cite{Li16} that the Heisenberg group $(\bH(\omega_{\bR^2}),\sd_{cc})$ is not Markov $p$-convex for $p < 4$, where $\omega_{\bR^2}$ is the scalar cross product on $\bR^2$. The reason that we can achieve better convexity properties in the asymptotic setting is, loosely speaking, due to the fact that the twisting factor $\omega_\banX(x_i,x_j)$ in the last coordinate always tends to 0 along a subsequence (the importance of this fact is apparent in the proof of Theorem~\ref{thm:Heisenberg}, from which Theorem~\ref{thm:C} follows). Therefore, as far as infrasup-umbel convexity is concerned, the Heisenberg group $\bH(\omega_\banX)$ behaves the same as the abelian group $\banX \oplus \bR$ (where the second factor is equipped with a snowflaked metric $\sqrt{|\cdot|}$), and thus one would expect it to be infrasup-umbel $p$-convex whenever $\banX$ has property $(\beta_p)$. Of course, for fixed vectors $x,y$, the twisting factor $\omega_\banX(x,y)$ in the last coordinate cannot be ignored, resulting in a more complex local geometry, as evidenced by the aforementioned result of Li. In fact, as an application of his methods, it was also shown in \cite{Li16} that the ball of radius $n$ in the integer lattice of $\bH(\omega_{\bR^2})$ has $\ell_2$-distortion at least a constant multiple of\footnote{The sharp bound $\Omega((\log n)^{\frac{1}{2}})$ was proved by Lafforgue-Naor in \cite{LafforgueNaor14}.}  $(\log n)^{\frac{1}{4}-o(1)}$. Our Theorem~\ref{thm:C} shows that an analogous argument with infrasup-umbel convexity in place of Markov convexity cannot be used to derive a nontrivial lower bound for the distortion of the integer lattice of $\bH(\omega_{\ell_2})$ (where $\omega_{\ell_2}$ is the form on $\ell_2 \oplus \ell_2$ given by $\omega_{\ell_2}((x,y),(x',y')) := \frac{1}{2}\langle x,y' \rangle - \frac{1}{2}\langle x',y \rangle$) into a Banach space with property $(\beta_2)$ (such as $\ell_2$). As far as we can tell, it is plausible that the integer lattice of $\bH(\omega_{\ell_2})$ does admit a bi-Lipschitz embedding into some Banach space with property $(\beta_2)$. On the other hand, Theorem~\ref{thm:C} gives sharp distortion bounds of countably branching trees into $\bH(\omega_{\ell_2})$, while in \cite{Li16} it is shown that Markov convexity does not give sharp distortion bounds of the binary trees into $\bH(\omega_{\ell_2})$. Later, we will introduce a local analogue of infrasup-umbel $p$-convexity, called infrasup-fork $p$-convexity. If we had a local analogue of Theorem~\ref{thm:C} stating that $\bH(\omega_{\ell_2})$ is infrasup-fork $2$-convex, then this would recover the sharp distortion bounds of the binary trees into $\bH(\omega_{\ell_2})$. However, we do not know if this is true (see Problem~\ref{pb:heisenberg-is2fork} and the discussion surrounding it).

Infrasup-umbel convexity can also be used to provide alternate and unified proofs of generalizations of a number of results that can be found in \cite{LimaLova12}, \cite{DKLR14}, \cite{DKR16}, and \cite{BaudierZhang16}. These applications can mostly be found in Section \ref{sec:distortion} and \ref{sec:quotients} where a quantitative analysis of embeddings of countably branching trees and the stability of umbel convexity and infrasup umbel convexity under nonlinear quotients are carried out.

As already alluded to, the Markov $p$-convexity inequality is elegantly shown in \cite{MendelNaor13} to follow from a certain iteration of the following inequality:
\begin{equation}\label{eq:qfork}
\frac{2^{-p}\dX(w,x)^p}{2}+\frac{2^{-p}\dX(w,y)^p}{2}+\frac{\dX(x,y)^p}{(4K)^p}\le \frac{1}{2}\dX(z,w)^p + \frac{1}{4}\dX(z,x)^p + \frac{1}{4}\dX(z,y)^p.
\end{equation}

A metric space $\metXd$ is said to satisfy the \emph{$p$-fork inequality with constant $K>0$} if \eqref{eq:qfork} holds for all $w,x,y,z\in \met X$.

In Section \ref{sec:Heisenberg-parallelogram}, we prove a parallelogram $2p$-convexity inequality for Heisenberg groups over $p$-uniformly convex Banach spaces. This useful inequality - first investigated for finite-dimensional Carnot groups by the second author (\cite[Lemma 4.17]{GartlandCarnot}) - is shown to imply the $2p$-fork inequality \eqref{eq:qfork} and $2p$-short diagonals inequality \eqref{eq:short-q}.

In light of our work on metric invariants related to countably branching trees, we study in Section \ref{sec:relaxed-Markov} certain relaxations of the $p$-fork inequality \eqref{eq:qfork} and the related full-blown deterministic metric invariants that can be derived from those. As previously mentioned, we introduce the metric invariant infrasup-fork $p$-convexity - a natural local analogue to infrasup-umbel $p$-convexity - that is sufficient to derive the conclusion of Theorem \ref{thm:Tessera}. The advantage to work with this invariant, which is a significant relaxation of Tessera's inequality \eqref{eq:q-tess}, is that it covers a large class of examples.

Finally, in Section \ref{sec:non-negative}, we borrow an idea from Lebedeva and Petrunin \cite{LebedevaPetrunin10} to show that the $2$-fork inequality with constant $K=1$ implies non-positive curvature. Interestingly, it was shown by Austin and Naor in \cite{AustinNaor} that non-negative curvature implies the $2$-fork inequality with constant $K=1$. The following characterization of non-negative curvature follows by combining these two observations.

\begin{theoalpha}
\label{thm:D}
Let $\metXd$ be a geodesic metric space. Then $\met X$ has non-negative curvature if and only if $\met X$ satisfies the $2$-fork inequality with constant $K=1$.
\end{theoalpha}

For the convenience of the reader, we also include in Appendix~\ref{app:table} a table summarizing the main inequalities introduced or recalled in the paper. \\

\textbf{Acknowledgments.}
We would like to thank Alexandros Eskenazis, Manor Mendel, and Assaf Naor for sharing with us their forthcoming work \cite{EMN}, and Manor Mendel for his generous and enlightening feedback on a first draft of this work that in particular led to Proposition \ref{prop:saturation}.

\section{Property $(\beta)$ with power type $p$ implies umbel $p$-convexity}
\label{sec:umbel}

The main goal of this section is to provide a proof of Theorem \ref{thmA}. This will be done via several steps interesting in their own right. First we prove some preparatory lemmas that will be used to derive a homogeneous inequality that is valid in any Banach space with property $(\beta_p)$. The first lemma is essentially technical.

\begin{lemm} 
\label{lem:1}
Let $\banXn$ be a Banach space. For all $\delta,\vep>0$, $v,w \in \banX$, and $V,W \geq \vep$ with $\|v\| \leq V$, $\|w\| \leq W$, and $\frac{1}{2}V + \frac{1}{2}W \le 1$, if $\bnorm{ \frac{v}{2V} +\frac{w}{2W}} \le 1-\delta$, then $\norm{\frac{1}{2}v+\frac{1}{2}w}\le 1-\vep\delta$.
\end{lemm}

\begin{proof}
Rescaling if needed, we may assume that $V+W=2$ and without loss of generality that $W \ge 1$. By assumption we have $\bnorm{\frac{v}{2V}+\frac{w}{2W}} \le 1-\delta$. Multiplying each side by $VW$ yields $\bnorm{\frac{W}{2}v+\frac{V}{2}w} \le (1-\delta) VW$. Then we have
\begin{align*}
\bnorm{\frac{v}{2}+\frac{w}{2}} & =  \bnorm{\Big(1-\frac{1}{W}\Big)w+\frac{1}{W}\Big(\frac{W}{2}v+\frac{V}{2}w\Big)} \\
							&  \le  \left(1-\frac{1}{W}\right)W + \frac{1}{W}\bnorm{\frac{W}{2}v+\frac{V}{2}w}\\
							& \le W-1+(1-\delta)V \\
							& = (V + W)-1-V\delta \\
							& \leq 1-\vep\delta.
\end{align*}
\end{proof}

The second lemma is a simple, but crucial, refinement of property $(\beta_p)$.

\begin{lemm}
\label{lem:betterbeta}
If $\banXn$ has property $(\beta_p)$ with $p>0$ and constant $c>0$ then for all $x\in B_\banX$ and $\{z_n\}_{n\in \bN}\subseteq B_\banX$, 
\begin{equation}\label{eq:betterbeta}
\inf_{n\in \bN} \bnorm{\frac{x-z_n}{2}}\le 1-\frac{1}{c}\inf_{i\in\bN}\liminf_{j\to \infty}\norm{z_i-z_j}^p.
\end{equation}
\end{lemm}

\begin{proof}
Assume, as we may, that $\inf_{i\in\bN}\liminf_{j\in \bN}\norm{z_i-z_j}=t>0$ and hence for all $i\in \bN$ we have $\liminf_{j\to\infty}\norm{z_i-z_j}\ge t$. Let $\vep>0$ be arbitrary. A diagonal extraction argument gives a subsequence $\{z_{n_j}\}_{j\ge 1}$ such that for all $i,j\in\bN$ it holds $\norm{z_i-z_{n_j}}\ge (1-\vep)t$. Therefore, there exists an infinite subset $\bM\eqd\{n_1, n_2,\dots\}$ of $\bN$ such that $\inf_{i\neq j\in\bM}\norm{z_i-z_j}\ge (1-\vep)t$. Since by assumption $\betabar(t)\ge \frac{t^p}{c}$, it follows from the definition of the $(\beta)$-modulus that 
there exists $m\in \bM$ such that 
\begin{equation*}
\bnorm{\frac{x-z_m}{2}}\le 1-\frac{(1-\vep)^pt^p}{c}=1-\frac{(1-\vep)^p}{c}\inf_{i\in\bN}\liminf_{j\to \infty}\norm{z_i-z_j}^p.
\end{equation*}
Since $\vep>0$ was arbitrary, the conclusion holds.
\end{proof}

Lemma \ref{lem:1} and Lemma \ref{lem:betterbeta} are now used to prove a homogenous inequality in Banach spaces with property $(\beta_p)$ for $p>1$. 

\begin{lemm}
\label{lem:pumbelBanach}
If $p \in (1,\infty)$ and $\banXn$ has property $(\beta_p)$ with constant $c$, then for all $w,z\in \ban X$ and $\{x_n\}_{n\in \bN}\subseteq \ban X$ 
\begin{equation}
\label{eq:pumbel}
\frac{1}{2^p}\inf_{n\in \bN} \norm{w-x_n}^p+\frac{1}{K}\inf_{i\in\bN}\liminf_{j\in \bN}\norm{x_i-x_j}^p\le \frac{1}{2}\norm{w-z}^p + \frac{1}{2}\sup_{n\in\bN} \norm{x_n-z}^p
\end{equation}
where $K$ is the least solution in $[2c,\infty)$ to the inequality
\begin{equation}
\label{eq:condition}
\frac{1}{2^p}\left(\frac{2c}{K}+\left(2-\left(\frac{2c}{K}\right)^p\right)^{1/p}\right)^p + \frac{2^{p+1}}{K} \leq 1.
\end{equation}
\end{lemm}

\begin{proof}
Before we begin the proof, note that inequality \eqref{eq:condition} has a solution $K \in [2c,\infty)$ because $p>1$.

Let $w,z\in \banX$ and $\{x_n\}_{n \in \bN} \subset \banX$. Since the distance induced by the norm of $\banX$ is translation invariant, we may assume $z=0$. We may also assume without loss of generality that $\sup_{n\in\bN}\norm{x_n}<\infty$, and by scale invariance of \eqref{eq:pumbel} we can assume that $\frac{1}{2}\|w\|^p + \frac{1}{2}\sup_{n\in\bN}\|x_n\|^p \le 1$. Thus equation \eqref{eq:pumbel} reduces to
$$
\inf_{n \in \bN} \left\|\frac{w-x_n}{2}\right\|^p + \frac{1}{K}\inf_{i \in \bN} \liminf_{j \to \infty} \norm{x_i-x_j}^p \leq 1.
$$
If $\inf_{i \in \bN} \liminf_{j \to \infty} \norm{x_i-x_j} = 0$, the above inequality holds trivially by the triangle inequality and convexity, so we may assume $\inf_{i \in \bN} \liminf_{j \to \infty} \norm{x_i-x_j} > 0$.

Set $W \eqd \norm{w}$ and $X \eqd \sup_{n \in \bN}\norm{x_n}$, so that
$$
\frac{1}{2}W + \frac{1}{2}X \le \left(\frac{1}{2}W^p + \frac{1}{2}X^p\right)^{1/p} \le 1.
$$
In particular remember that $\max\{W^p,X^p\} \le 2$. Set $\vep \eqd \frac{2c}{K}$, and note that $\vep \in (0,1]$. We consider separately the two cases $\min\{W,X\} \geq \vep$ and $\min\{W,X\} \le \vep$.

Assume first that $\min\{W,X\} \geq \vep$ holds. Lemma \ref{lem:betterbeta} implies
\begin{eqnarray*}
    \inf_{n \in \bN} \left\|\frac{w}{2W}-\frac{x_n}{2X}\right\| \leq 1-\frac{1}{c}\inf_{i\in\bN}\liminf_{j\to \infty}\left\|\frac{x_i}{X}-\frac{x_j}{X}\right\|^p \\
    \leq 1-\frac{1}{cX^p}\inf_{i\in\bN}\liminf_{j\to \infty}\left\| x_i-x_j \right\|^p \\
    \leq 1-\frac{1}{2c}\inf_{i\in\bN}\liminf_{j\to \infty}\left\|x_i-x_j\right\|^p.
\end{eqnarray*}
Let $\eta \in (0,1)$ be arbitrary. Then by the above inequality and our assumption that \\ $\inf_{i \in \bN} \liminf_{j \to \infty} \norm{x_i-x_j} > 0$, we may choose $m \in \bN$ such that
$$\left\|\frac{w}{2W}-\frac{x_m}{2X}\right\| \leq 1-\frac{1-\eta}{2c}\inf_{i\in\bN}\liminf_{j\to \infty}\left\|x_i-x_j\right\|^p.$$
This inequality shows that the hypotheses of Lemma \ref{lem:1} are fulfilled, and thus by the definition of $\vep$, the fact that $\left\|\frac{w-x_m}{2}\right\| \le 1$, and Lemma \ref{lem:1} we get
\begin{equation*}
\inf_{n \in \bN} \left\|\frac{w-x_n}{2}\right\|^p + \frac{1-\eta}{K}\inf_{i \in \bN} \liminf_{j \to \infty} \norm{x_i-x_j}^p  \le \left\|\frac{w-x_m}{2}\right\| + \frac{(1-\eta)\vep}{2c}\inf_{i \in \bN} \liminf_{j \to \infty} \norm{x_i-x_j}^p \leq 1.
\end{equation*}
Since $\eta \in (0,1)$ was arbitrary, we achieve the required inequality.

Now assume we are in the second case $\min\{W,X\} \leq \vep$. We just treat the subcase $W \leq \vep$; the other subcase follows from nearly the same argument. We have
\begin{eqnarray*}
\inf_{n \in \bN} \left\|\frac{w-x_n}{2}\right\|^p + \frac{1}{K}\inf_{i \in \bN} \liminf_{j \to \infty} \norm{x_i-x_j}^p &\leq& \left(\frac{W+X}{2}\right)^p + \frac{1}{K}(2X)^p \\
&\leq& \left(\frac{W+(2-W^p)^{1/p}}{2}\right)^p + \frac{1}{K}(2X)^p \\
&\leq& \left(\frac{\vep+(2-\vep^p)^{1/p}}{2}\right)^p + \frac{1}{K}2^{p+1} \\
&=&    \frac{1}{2^p}\left(\frac{2c}{K}+\left(2-\left(\frac{2c}{K}\right)^p\right)^{1/p}\right)^p + \frac{2^{p+1}}{K} \\
&\leq& 1
\end{eqnarray*}
where the last inequality is the definition of $K$, and the second-to-last inequality follows from the fact that $X^p \leq 2$, $W \leq \vep$, and the fact that $t \mapsto t+(2-t^p)^{1/p}$ is increasing on $[0,1]$.
\end{proof}

Since inequality \eqref{eq:pumbel} only involves the norm of differences of vectors, it will be convenient to introduce the following definition and terminology.

\begin{defi}
A metric space $\metXd$ is said to satisfy the \emph{$p$-umbel inequality} with constant $K\in(0,\infty)$ if for all $w,z\in \met X$ and $\{x_i\}_{i \in \bN}\subseteq \met X$ we have 
\begin{equation} \label{eq:pumbelmetric}
\frac{1}{2^p}\inf_{i \in \bN} \dX(w,x_i)^p+\frac{1}{K^p}\inf_{i\in\bN}\liminf_{j \to \infty} \dX(x_i,x_j)^p\le \frac{1}{2}\dX(z,w)^p + \frac{1}{2}\sup_{i \in \bN} \dX(z,x_i)^p
\end{equation} 
\end{defi}

The $p$-umbel inequality is a strengthening of the triangle inequality for sequences $\{x_i\}_{i \in \bN}$ that do not admit any Cauchy subsequence. The next theorem is the main result of this section.

\begin{theo}
\label{thm:pumbel->umbelp}
Let $p\in (0,\infty)$.
If $\metXd$ satisfies the $p$-umbel inequality with constant $K>0$, then $\metXd$ is umbel $p$-convex. Moreover, $\Pi_p^u(\metX)\le \max\{1,2^{\frac{1}{p}-1}\} \cdot K$.
\end{theo}

\begin{proof}
We will show a bit more than what is needed for Theorem \ref{thm:pumbel->umbelp}, and in this proof we allow $\dX$ to be a quasi-metric and not necessarily a genuine metric, \ie that instead of the triangle inequality we assume that there exists a constant $c\ge 1$ such that $\dX(x,y)\le c(\dX(x,z)+\dX(z,y))$ for all $x,y,z\in\metX$. We will show by induction on $k$ that for all maps $f\colon \wtree{\bN}{\le 2^k}\to \metX$ and all $r\in \bN$,
\begin{align*}
\frac{1}{K^p}  \sum_{s=1}^{k-1}\frac{1}{2^{k-1-s}}\sum_{t=1}^{2^{k-1-s}}\inf_{\nbar\in\wtree{\bN}{t2^{s+1}-2^s}}\inf_{\stackrel{\bar{\vep}\in\wtree{\bN}{2^s}:}{(\nbar,\bar{\vep})\in [\bN]^{\le 2^k}}}\liminf_{j\to \infty}\inf_{\stackrel{\bar{\delta}\in\wtree{\bN}{2^s-1}:}{(\nbar,j,\bar{\delta})\in\wtree{\bN}{\le 2^k}}}\frac{\dX \big( f(\nbar,\bar{\vep}),f(\nbar,j,\bar{\delta}) \big)^p}{2^{sp}}\\			
+ \inf_{\stackrel{\nbar \in \wtree{\bN}{2^k-1:}}{(r,\nbar)\in\wtree{\bN}{2^{k}}}}\frac{\dX(f(\emptyset),f(r,\nbar))^p}{2^{kp}} \le \max\{1,2^{1-p}\} \cdot c^p\frac{1}{2^{k}}\sum_{\ell=1}^{2^{k}}\sup_{\nbar\in \wtree{\bN}{\ell}}\dX \big( f(n_1,\dots,n_{\ell-1}),f(n_1,\dots,n_{\ell}) \big)^p.
\end{align*}
The conclusion of the theorem follows by discarding the additional non-negative term which is solely needed for the induction proof.

For the base case $k=1$, the inequality reduces to
\begin{equation*}
\inf_{\stackrel{n\in \bN}{n>r}}\frac{\dX(f(\emptyset),f(r,n))^p}{2^p} \le \max\{1,2^{1-p}\}\frac{c^p}{2} \left(\sup_{n\in\bN}\dX(f(\emptyset),f(n))^p + \sup_{(n_1,n_2)\in \wtree{\bN}{2}}\dX(f(n_1),f(n_1,n_2))^p\right).
\end{equation*}
Observing that this inequality is an immediate consequence of the quasi-triangle and H\"older inequalities, the base case is settled. For convenience, we assume throughout the remainder of the proof that $p \geq 1$, so that $\max\{1,2^{1-p}\} = 1$. The proof carries through line-by-line in the case $p \leq 1$, but with an additional factor of $2^{1-p}$ on the right-hand side.

We now proceed with the inductive step and fix $i\in \bN$ and $f\colon \wtree{\bN}{\le 2^{k+1}}\to \metX$. Given $\vep>0$, we pick $\mbar\in \wtree{\bN}{\le 2^k-1}$ such that 
\begin{equation*}
\frac{\dX(f(\troot),f(i,\mbar))^p}{2^{kp}}\le \inf_{\nbar\in\wtree{\bN}{2^{k}-1}}\frac{\dX(f(\emptyset),f(i,\nbar))^p}{2^{kp}}+\vep,
\end{equation*}
and for each $r\in\bN$, choose $\bar{u}(r)\in\wtree{\bN}{2^{k}-1}$ so that
\begin{equation*}
\frac{\dX(f(i,\mbar),f(i,\mbar,r,\bar{u}(r)))^p}{2^{kp}}\le \inf_{\nbar\in\wtree{\bN}{2^{k}-1}}\frac{\dX(f(i,\mbar),f(i,\mbar,r,\nbar))^p}{2^{kp}}+\vep.
\end{equation*}
In order to simplify the (otherwise awkward and tedious) notation we have implicitly assumed above that $\mbar$, $r$, and $\bar{u}(r)$ are such that $i<m_1<\dots<m_{l}<r<u_1(r)<\dots<u_{l'}(r)$, or in other words that $(i,\mbar,r,\bar{u}(r))$ truly belongs to $\wtree{\bN}{\le 2^{k+1}}$. We will follow this notational convention here and in the ensuing proofs. 

By the induction hypothesis applied to the restriction of $f$ to $\wtree{\bN}{\le 2^{k}}$ (and with $r=i$) we get
\begin{align} 
\label{eq:umbelaux1} 
\nonumber \frac{1}{K^p} & \sum_{s=1}^{k-1}\frac{1}{2^{k-1-s}}\sum_{t=1}^{2^{k-1-s}}\inf_{\nbar\in\wtree{\bN}{t2^{s+1}-2^s}}\inf_{\bar{\vep}\in [\bN]^{2^s}}\liminf_{j\to \infty}\inf_{\bar{\delta}\in\wtree{\bN}{2^s-1}}\frac{\dX(f(\nbar,\bar{\vep}),f(\nbar,j,\bar{\delta}))^p}{2^{sp}}\\
 &+ \inf_{\nbar\in\wtree{\bN}{2^{k}-1}}\frac{\dX(f(\emptyset),f(i,\nbar))^p}{2^{kp}}  			
					 \le  \frac{c^p}{2^k}\sum_{\ell=1}^{2^k}\sup_{\nbar\in \wtree{\bN}{\ell}}\dX(f(n_1,\dots,n_{\ell-1}),f(n_1,\dots,n_{\ell}))^p.
\end{align}
On the other hand, the induction hypothesis applied to $g(\nbar)\eqd f((i,\mbar),\nbar)$ where $\nbar\in \wtree{\bN}{\le 2^{k}}$ gives
\begin{align*}
\frac{1}{K^p} & \sum_{s=1}^{k-1}\frac{1}{2^{k-1-s}}\sum_{t=1}^{2^{k-1-s}}\inf_{\nbar\in\wtree{\bN}{t2^{s+1}-2^s}}\inf_{\bar{\vep}\in [\bN]^{2^s}}\liminf_{j\to \infty}\inf_{\bar{\delta}\in\wtree{\bN}{2^s-1}}\frac{\dX(g(\nbar,\bar{\vep}),g(\nbar,j,\bar{\delta}))^p}{2^{sp}}\\			
					& +\inf_{\nbar\in\wtree{\bN}{2^{k}-1}}\frac{\dX(g(\troot),g(\nbar))^p}{2^{kp}} \le \frac{c^p}{2^k}\sum_{\ell=1}^{2^k}\sup_{\nbar\in \wtree{\bN}{\ell}}\dX(g(n_1,\dots,n_{\ell-1}),g(n_1,\dots,n_{\ell}))^p.
\end{align*}
Observe first that, for any $1 \leq \ell \leq 2^k$,
\begin{align*}
\sup_{\nbar\in \wtree{\bN}{\ell}}\dX(g(n_1,\dots,n_{\ell-1}),g(n_1,\dots,n_{\ell}))^p = \sup_{\nbar\in \wtree{\bN}{\ell}}\dX(f(i,\mbar,n_1,\dots,n_{\ell-1}),f(i,\mbar,n_1,\dots,n_{\ell}))^p\\
\le \sup_{\nbar\in \wtree{\bN}{2^k+\ell}}\dX(f(n_1,\dots,n_{2^k+\ell-1}),f(n_1,\dots,n_{2^k+\ell}))^p,
\end{align*}
since we are taking the supremum over the set of all edges between level $2^k+\ell-1$ and level $2^k+\ell$ instead of a subset of it. Also, for each $s = 1, \dots, k-1$,
\begin{align*}
\frac{1}{2^{k-1-s}} & \sum_{t=1}^{2^{k-1-s}}\inf_{\nbar\in\wtree{\bN}{t2^{s+1}-2^s}} \inf_{\bar{\vep}\in [\bN]^{2^s}} \liminf_{j\to \infty}\inf_{\bar{\delta}\in\wtree{\bN}{2^s-1}}\frac{\dX(g(\nbar,\bar{\vep}),g(\nbar,j,\bar{\delta}))^p}{2^{sp}}\\
&  =  \frac{1}{2^{k-1-s}}\sum_{t=1}^{2^{k-1-s}}\inf_{\nbar\in\wtree{\bN}{t2^{s+1}-2^s}}\inf_{\bar{\vep}\in [\bN]^{2^s}}\liminf_{j\to \infty}\inf_{\bar{\delta}\in\wtree{\bN}{2^s-1}}\frac{\dX(f(i,\mbar,\nbar,\bar{\vep}),f(i,\mbar,\nbar,j,\bar{\delta}))^p}{2^{sp}}\\
& \ge  \frac{1}{2^{k-1-s}}\sum_{t=1}^{2^{k-1-s}} \inf_{\nbar\in\wtree{\bN}{2^k+t2^{s+1}-2^s}}\inf_{\bar{\vep}\in [\bN]^{2^s}}\liminf_{j\to \infty}\inf_{\bar{\delta}\in\wtree{\bN}{2^s-1}}\frac{\dX(f(\nbar,\bar{\vep}),f(\nbar,j,\bar{\delta}))^p}{2^{sp}},
\end{align*}
since $(i,\mbar,\nbar)\in \wtree{\bN}{2^k+t2^{s+1}-2^s}$ for all $\nbar\in\wtree{\bN}{t2^{s+1}-2^s}$.

Therefore, it follows from the two relaxations above (and a reindexing) that 
\begin{align}
\label{eq:umbelaux2}
\nonumber \frac{1}{K^p} & \sum_{s=1}^{k-1}\frac{1}{2^{k-1-s}}\sum_{t=2^{k-1-s}+1}^{2^{k-s}}\inf_{\nbar\in\wtree{\bN}{t2^{s+1}-2^s}}\inf_{\bar{\vep}\in [\bN]^{2^s}}\liminf_{j\to \infty}\inf_{\bar{\delta}\in\wtree{\bN}{2^s-1}}\frac{\dX(f(\nbar,\bar{\vep}),f(\nbar,j,\bar{\delta}))^p}{2^{sp}}\\			
					& + \inf_{\nbar\in\wtree{\bN}{2^{k}-1}}\frac{\dX(f(i,\mbar),f(i,\mbar,r,\nbar))^p}{2^{kp}} \le \frac{c^p}{2^{k}}\sum_{\ell=2^k+1}^{2^{k+1}} \dX(f(n_1,\dots,n_{\ell-1}),f(n_1,\dots,n_{\ell}))^p.
\end{align}
Taking the supremum over $r$ in \eqref{eq:umbelaux2} and then averaging the resulting inequality with \eqref{eq:umbelaux1} yields
\begin{align}
\label{eq:umbelaux3}
\nonumber \frac{1}{2^{kp}}\left(\frac{1}{2}\inf_{\nbar\in\wtree{\bN}{2^{k}-1}}\dX(f(\emptyset),f(i,\nbar))^p + \frac{1}{2}\sup_{r\in \bN} \inf_{\nbar\in\wtree{\bN}{2^{k}-1}}\dX(f(i,\mbar),f(i,\mbar,r,\nbar))^p\right)\\
 + \frac{1}{K^p}\sum_{s=1}^{k-1}\frac{1}{2^{k-s}}\sum_{t=1}^{2^{k-s}}\inf_{\nbar\in\wtree{\bN}{t2^{s+1}-2^s}}\inf_{\bar{\vep}\in [\bN]^{2^s}}\liminf_{j\to \infty}\inf_{\bar{\delta}\in\wtree{\bN}{2^s-1}}\frac{\dX(f(\nbar,\bar{\vep}),f(\nbar,j,\bar{\delta}))^p}{2^{sp}}\\			
\nonumber					\le  \frac{c^p}{2^{k+1}}\sum_{\ell=1}^{2^{k+1}}\sup_{\nbar\in \wtree{\bN}{\ell}}\dX(f(n_1,\dots,n_{\ell-1}),f(n_1,\dots,n_{\ell}))^p.
\end{align}
If we let $w\eqd f(\troot)$, $z\eqd f(i,\mbar)$, and $x_r\eqd f(i,\mbar,r, \bar{u}(r))$, it follows from how $\mbar$ and $\bar{u}(r)$ were chosen, that
\begin{align}
\label{eq:umbelaux4}
\frac{1}{2^{kp}} & \left(\frac{1}{2}\dX(w,z)^p + \frac{1}{2}\sup_{r\in \bN}\dX(z,x_r)^p\}\right) \nonumber \\
 & \le \frac{1}{2^{kp}} \left(\frac{1}{2}\inf_{\nbar\in\wtree{\bN}{2^{k}-1}}\dX(f(\emptyset),f(i,\nbar))^p + \frac{1}{2}\sup_{r\in \bN} \inf_{\nbar\in\wtree{\bN}{2^{k}-1}}\dX(f(i,\mbar),f(i,\mbar,r,\nbar))^p\right) + \vep
\end{align}
The $p$-umbel inequality combined with \eqref{eq:umbelaux3} and \eqref{eq:umbelaux4} gives 
\begin{align}
\label{eq:umbelaux5}
\nonumber \frac{1}{2^{(k+1)p}} & \inf_{r\in \bN}\dX(w,x_r)^p+\frac{1}{K^p}\frac{1}{2^{kp}}\inf_{r\in \bN}\liminf_{q \to \infty}\dX(x_r,x_q)^p\\
 & + \frac{1}{K^p}\sum_{s=1}^{k-1} \frac{1}{2^{k-s}}\sum_{t=1}^{2^{k-s}} \inf_{\nbar\in\wtree{\bN}{t2^{s+1}-2^s}}\inf_{\bar{\vep}\in [\bN]^{2^s}}\liminf_{j\to \infty}\inf_{\bar{\delta}\in\wtree{\bN}{2^s-1}}\frac{\dX(f(\nbar,\bar{\vep}),f(\nbar,j,\bar{\delta}))^p}{2^{sp}}\\			
\nonumber 					 & \le   \frac{c^p}{2^{k+1}}\sum_{\ell=1}^{2^{k+1}} \sup_{\nbar\in \wtree{\bN}{\ell}}\dX(f(n_1,\dots,n_{\ell-1}),f(n_1,\dots,n_{\ell}))^p+\vep.
\end{align}
Now observe that
\begin{equation*}
\inf_{r\in \bN}\dX(w,x_r)^p  =  \inf_{r\in \bN}\dX( f(\troot),f(i,\mbar,r, \bar{u}(r)))^p \ge \inf_{\nbar\in \wtree{\bN}{2^{k+1}-1}}\dX( f(\troot),f(i,\nbar))^p,
\end{equation*}
and
\begin{align*}
\inf_{r\in \bN}\liminf_{q \to \infty}\dX(x_r,x_q)^p & = \inf_{r\in \bN}\liminf_{q \to \infty}\dX(f(i,\mbar,r, \bar{u}(r)),f(i,\mbar,q, \bar{u}(q)))^p \\
								& \ge  \inf_{r\in \bN}\liminf_{q \to \infty}\inf_{\bar{\delta}\in \wtree{\bN}{2^k-1}}\dX(f(i,\mbar,r, \bar{u}(r)),f(i,\mbar,q, \bar{\delta}))^p\\
								& \ge \inf_{\bar{\vep}\in \wtree{\bN}{2^k}}\liminf_{q \to \infty}\inf_{\bar{\delta}\in \wtree{\bN}{2^k-1}}\dX(f(i,\mbar,\bar{\vep}),f(i,\mbar,q, \bar{\delta}))^p\\
								& \ge \inf_{\nbar\in \wtree{\bN}{2^k}} \inf_{\bar{\vep}\in \wtree{\bN}{2^k}}\liminf_{q \to \infty}\inf_{\bar{\delta}\in \wtree{\bN}{2^k-1}}\dX(f(\nbar,\bar{\vep}),f(\nbar,q, \bar{\delta}))^p.
\end{align*}
Plugging in the two relaxed inequalities above in \eqref{eq:umbelaux5} we obtain
\begin{align*}
\inf_{\nbar\in \wtree{\bN}{2^{k+1}}}\frac{\dX( f(\troot),f(i,\nbar))^p}{2^{(k+1)p}}+\frac{1}{K^p}\inf_{\nbar\in \wtree{\bN}{2^k}} \inf_{\bar{\vep}\in \wtree{\bN}{2^k}}\liminf_{q \to \infty}\inf_{\bar{\delta}\in \wtree{\bN}{2^k-1}}\frac{\dX(f(\nbar,\bar{\vep}),f(\nbar,q, \bar{\delta}))^p}{2^{kp}}\\
 + \frac{1}{K^p}\sum_{s=1}^{k-1} \frac{1}{2^{k-s}}\sum_{t=1}^{2^{k-s}} \inf_{\nbar\in\wtree{\bN}{t2^{s+1}-2^s}}\inf_{\bar{\vep}\in [\bN]^{2^s}}\liminf_{j\to \infty}\inf_{\bar{\delta}\in\wtree{\bN}{2^s-1}}\frac{\dX(f(\nbar,\bar{\vep}),f(\nbar,j,\bar{\delta}))^p}{2^{sp}}\\			
					\le  \frac{c^p}{2^{k+1}}\sum_{\ell=1}^{2^{k+1}} \sup_{\nbar\in \wtree{\bN}{\ell}}\dX(f(n_1,\dots,n_{\ell-1}),f(n_1,\dots,n_{\ell}))^p+\vep,
\end{align*}
and hence 
\begin{align*}
\frac{1}{K^p} & \sum_{s=1}^{k} \frac{1}{2^{k-s}}\sum_{t=1}^{2^{k-s}} \inf_{\nbar\in\wtree{\bN}{t2^{s+1}-2^s}}\inf_{\bar{\vep}\in [\bN]^{2^s}}\liminf_{j\to \infty}\inf_{\bar{\delta}\in\wtree{\bN}{2^s-1}}\frac{\dX(f(\nbar,\bar{\vep}),f(\nbar,j,\bar{\delta}))^p}{2^{sp}}\\			
				& + \inf_{\nbar\in \wtree{\bN}{2^{k+1}}}\frac{\dX( f(\troot),f(i,\nbar))^p}{2^{(k+1)p}} \le  \frac{c^p}{2^{k+1}}\sum_{\ell=1}^{2^{k+1}}  \sup_{\nbar\in \wtree{\bN}{\ell}}\dX(f(n_1,\dots,n_{\ell-1}),f(n_1,\dots,n_{\ell}))^p+\vep.
\end{align*}
Since $\vep$ is arbitrary the induction step is completed.
\end{proof}


The next corollary is an immediate consequence of Lemma \ref{lem:pumbelBanach} and Theorem \ref{thm:pumbel->umbelp}.

\begin{coro}\label{cor:betap->umbelp}
A Banach space with property $(\beta_p)$ for some $p\in(1,\infty)$ is umbel $p$-convex.
\end{coro}

Recall that it follows from \cite{BKL10} that if a Banach space $\banX$ is reflexive and does not contain equi-bi-Lipschitz copies of the countably branching trees, then $\banX$ admits an equivalent norm with property $(\beta)$. Therefore, to complete the proof of Theorem \ref{thmA}, it remains to show that a Banach space that is umbel $p$-convex for some $p\in(1,\infty)$ satisfies those requirements. 

We will first show that reflexivity is implied by umbel $p$-convexity. The umbel $p$-convexity inequality \eqref{eq:umbel-p-convex} is rather complex, and for many applications, such as the reflexivity problem at stake, certain simpler relaxed inequalities will suffice. For example, the following relaxation of the umbel $p$-convexity inequality will be sufficient to ensure reflexivity:

There exists $C>0$ such that for all $k\ge 1$ and all $f\colon \tree^\omega_{2^k}=(\wtree{\bN}{\le 2^k},\sd_\tree)\to \banX$,

\begin{align}
\label{eq:relaxed-umbel-p}
\sum_{s=1}^{k-1}\inf_{\nbar\in\wtree{\bN}{\le 2^{k}-2^s}}\inf_{\stackrel{\bar{\delta}\in\wtree{\bN}{2^s}\colon}{(\nbar,\bar{\delta})\in \wtree{\bN}{\le 2^k}}}\liminf_{j\to\infty}\inf_{\stackrel{\bar{\eta}\in\wtree{\bN}{2^s-1}\colon}{(\nbar,j,\bar{\eta})\in \wtree{\bN}{\le 2^k}}}\frac{\sd(f(\nbar,\bar{\delta}),f(\nbar,j,\bar{\eta}))^p}{2^{sp}} \le C^p\lip(f)^p.
\end{align}

\begin{rema}
\label{rem:relaxed-p-umbel}
Consider the following relaxation of the $p$-umbel inequality:

For all $w,z\in \metX$ and $\{x_n\}_{n\in \bN}\subseteq \metX$ 
\begin{equation}
\label{eq:relaxed-p-umbel}
\frac{1}{2^p}\inf_{n\in \bN} \dX(w,x_n)^p+\frac{1}{K^p}\inf_{i\in\bN}\liminf_{j\in \bN}\dX(x_i,x_j)^p\le \max\{\dX(w,z)^p, \sup_{n\in\bN} \dX(x_n,z)^p\}.
\end{equation}
Using similar and slightly simpler arguments to those in the proof of Lemma \ref{lem:betterbeta}, we could show that if $p \in (0,\infty)$ and $\banXn$ has property $(\beta_p)$, then the metric induced by the norm on $\banX$ satisfies inequality \eqref{eq:relaxed-p-umbel}. Moreover, the relaxation of the umbel $p$-convexity inequality \eqref{eq:relaxed-umbel-p} can then be derived from the relaxation of the $p$-umbel inequality in a similar way umbel $p$-convexity was derived from the $p$-umbel inequality (and the proof also works for quasi-metrics).
\end{rema}

The following lemma can be deduced from one of James' characterization of reflexivity, and we refer to \cite[Lemma 3.0.1]{DKR16} for its proof. Recall that $[\bN]^{<\omega}$ denotes the set of all finite subset of $\bN$, and $([\bN]^{<\omega}, \sd_\tree)$ is the countably branching tree of infinite height equipped with the tree metric.

\begin{lemm}\label{lem:James}
If $\banX$ is non-reflexive, then for every $\theta\in(0,1)$, there exists a $1$-Lipschitz map $g\colon ([\bN]^{<\omega}, \sd_\tree)\to \banX$ such that for all $\ubar=(n_1,\dots,n_{s},n_{s+1},\dots,n_{s+k})$ and  $\bar{v}=(n_1,\dots,n_{s},m_1,\dots, m_k)$ where $n_1<\dots<n_s<n_{s+1}<\dots<n_{s+k}<m_1<\dots<m_k$,
\begin{equation}\label{eq:James-sep}
\norm{g(\ubar)-g(\bar{v})}_\banX \ge  \frac{\theta}{3}\sd_\tree(\ubar,\bar{v}) 
\end{equation}
\end{lemm}

\begin{rema}
In fact, the conclusion of Lemma \ref{lem:James} holds under the weaker assumption that the Banach space does not have the alternating Banach-Saks property (cf. \cite{Beauzamy79}).
\end{rema}

\begin{prop}\label{prop:reflexivity}
Let $\banXn$ be a Banach space. If $\banX$ supports the inequality \eqref{eq:relaxed-umbel-p} for some 
$p\in(0,\infty)$, then $\banX$ is reflexive.
In particular, if $\banX$ is umbel $p$-convex for some $p\in(0,\infty)$, then $\banX$ is reflexive.
\end{prop}

\begin{proof}
Assume that $\banX$ supports the inequality \eqref{eq:relaxed-umbel-p} for some 
$p\in(0,\infty)$ but is not reflexive. Consider the restriction to $[\bN]^{\le 2^k}$ of the map $g$ from Lemma \ref{lem:James}. Then, for all $\nbar\in \wtree{\bN}{\le 2^k-2^{s}}$, $\bar{\delta}=(\delta_1,\dots,\delta_{2^s})\in \wtree{\bN}{2^{s}}$, $j \in \bN$ and $\bar{\eta} \in \wtree{\bN}{2^{s}-1}$ such that $(\nbar , \bar{\delta}), (\nbar , j , \bar{\eta})\in\wtree{\bN}{\le 2^k}$, it follows from \eqref{eq:James-sep} that if $j>\delta_{2^s}$, then $\norm{g(\nbar, \bar{\delta})-g(\nbar , j , \bar{\eta})}_\banX^p\ge \frac{\theta^p}{3^p}2^{(s+1)p}$. Therefore, 
	\begin{equation*}
		\sum_{s=1}^{k-1}\inf_{\nbar\in\wtree{\bN}{\le 2^k-2^{s}}}\inf_{\stackrel{\bar{\delta}\in\wtree{\bN}{2^s}\colon}{(\nbar,\bar{\delta})\in \wtree{\bN}{\le 2^k}}}\liminf_{j\to\infty}\inf_{\stackrel{\bar{\eta}\in\wtree{\bN}{2^s-1}\colon}{(\nbar,j,\bar{\eta})\in \wtree{\bN}{\le 2^k}}}\frac{\norm{g(\nbar,\bar{\delta})-g(\nbar,j,\bar{\eta})}_\banX^p}{2^{sp}} \ge (k-1)\Big(\frac{2\theta }{3}\Big)^p,
	\end{equation*}
and since $g$ is $1$-Lipschitz, inequality \eqref{eq:relaxed-umbel-p} gives $C^p\ge (k-1)\Big(\frac{2\theta }{3}\Big)^p$ for all $k\ge 1$; a contradiction.
\end{proof}

An argument similar to the proof of Proposition \ref{prop:reflexivity} show that there is no equi-bi-Lipschitz embeddings of the countably branching trees of finite but arbitrarily large height, into a metric space that supports the inequality \eqref{eq:relaxed-umbel-p} for some $p\in(0,\infty)$. The simple argument is deferred to Proposition \ref{prop:tree-dist} in the next section. Theorem \ref{thmA} can be derived from Theorem \ref{thm:renorming} and the following corollary which follows from the above discussion.

\begin{coro}
\label{cor:metchar}
Let $\banX$ be a Banach space. The following assertions are equivalent.
\begin{enumerate}
\item $\banX$ admits an equivalent norm with property $(\beta_p)$ for some $p\in(1,\infty)$.
\item $\banX$ is umbel $p$-convex for some $p\in(1,\infty)$.
\item $\banX$ supports the relaxation of the umbel $p$-convexity inequality \eqref{eq:relaxed-umbel-p} for some 
$p\in(1,\infty)$.
\end{enumerate}
\end{coro}

Following Ostrovskii \cite{Ostrovskii14b}, we say that a class $\mathscr C$ of Banach spaces admits a submetric test-space characterization if there exists a metric space $\metX$ and a marked subset $S\subset \metX\times \metX$ such that $\banY\notin \mathscr C$ if and only if $\metX$ admits a \emph{partial bi-Lipschitz embedding} into $\banY$, i.e. there exists a constant $D\ge 1$ such that for all $(x,y)\in S$, $\dX(x,y)\le \norm{f(x)-f(y)}_\banY\le D\dX(x,y)$.

 Below is the submetric test-space characterization of the class $\langle (\beta) \rangle$ mentioned in the introduction.

\begin{coro}\label{cor:submetric}
A Banach space $\banX$ does not admit an equivalent norm with property $(\beta)$ if and only if there exist a constant $A>0$ and a $1$-Lipschitz map $g\colon ([\bN]^{<\omega},\sd_\tree)\to \banX$ such that for all $\ubar=(n_1,\dots,n_{s},n_{s+1},\dots,n_{s+k})$ and  $\bar{v}=(n_1,\dots,n_{s},m_1,\dots, m_k)$ where $n_1<\dots<n_s<n_{s+1}<\dots<n_{s+k}<m_1<\dots<m_k$,
\begin{equation*}
\norm{g(\ubar)-g(\bar{v})}_\banX \ge \frac{1}{A} \sd_\tree(\ubar,\bar{v}).
\end{equation*}
\end{coro}

\begin{proof}
Assume that $\banX$ does not admit an equivalent norm with property $(\beta)$. If $\banX$ is reflexive, then by \cite{BKL10} it contains a bi-Lipschitz embedding of $([\bN]^{<\omega},\sd_\tree)$, the countably branching tree of infinite height, and the condition is clearly satisfied.
If $\banX$ is not reflexive, then we can take the map from Lemma \ref{lem:James}. Assuming now that $\banX$ admits an equivalent norm with property $(\beta)$, then we can assume that $\banX$ supports the inequality \eqref{eq:relaxed-umbel-p} for some $p\in(1,\infty)$. It remains to observe that the proof of Proposition \ref{prop:reflexivity} shows that there cannot exist an $\banX$-valued map satisfying the conditions listed in the statement of Corollary \ref{cor:submetric}.
\end{proof}

\section{Distortion and compression rate of embeddings of countably branching trees}
\label{sec:distortion}

As we hinted at in the previous section, umbel convexity and its relaxation \eqref{eq:relaxed-umbel-p} are obstructions to the faithful embeddability of the countably branching tree. In fact, if we are only concerned with embeddability obstructions, a further relaxation of \eqref{eq:relaxed-umbel-p}, namely the infrasup-umbel $p$-convexity inequality as defined in Definition \ref{def:umbelcotypep}, is sufficient. Recall that $\metXd$ is infrasup-umbel $p$-convex if there exists a constant $C>0$ such that for all $k\ge 1$ and all $f\colon \tree^\omega_{2^k}\to \met X$,

\begin{align}\label{eq:umbelcotypep}
\left(\sum_{s=1}^{k-1}\inf_{\nbar\in\wtree{\bN}{\le 2^{k}-2^s}}\inf_{i \neq j \in \bN} \inf_{\stackrel{\mbar,\mbar' \in \wtree{\bN}{2^s-1} \colon }{(\nbar,i,\mbar),(\nbar,j,\mbar')\in \wtree{\bN}{\le 2^k}}}\frac{\dX(f(\nbar,i,\mbar),f(\nbar,j,\mbar'))^p}{2^{sp}}\right)^{\frac{1}{p}}\le C\lip(f).
\end{align}
We will denote by $\Pi^{isu}_p(\metX)$ the least constant for which \eqref{eq:umbelcotypep} holds for all $k \ge 1$ and all maps $f\colon \tree^\omega_{2^k}\to \met X$.

Consider the following further relaxation of the $p$-umbel inequality,  which we will refer to as the \emph{infrasup $p$-umbel inequality}. For all $w,z\in \metX$ and $\{x_n\}_{n\in \bN}\subseteq \metX$,
\begin{equation}
\label{eq:superrelaxed-p-umbel}
\frac{1}{2^p}\inf_{n\in \bN} \dX(w,x_n)^p+\frac{1}{K^p}\inf_{i\neq j\in\bN}\dX(x_i,x_j)^p\le \max\{\dX(w,z)^p, \sup_{n\in\bN} \dX(x_n,z)^p\}.
\end{equation}
If $p \in [1,\infty)$ and $\banXn$ has property $(\beta_p)$, then the metric induced by the norm on $\banX$ satisfies inequality \eqref{eq:superrelaxed-p-umbel}. Moreover, infrasup-umbel $p$-convexity can be derived from the infrasup $p$-umbel inequality (here as well the proof works for quasi-metrics).

It is easily verified that the inequalities \eqref{eq:umbel-p-convex}, \eqref{eq:relaxed-umbel-p}, and \eqref{eq:umbelcotypep} generate metric invariants, in the sense that $\Pi_p(\metX)\le \cdist{\metY}(\metX)\Pi_p(\metY)$ (and similarly for the other two inequalities). Note also that $\Pi_p^{isu}(\metX)\le \Pi_p^u(\metX)$. The terminology ``infrasup-umbel convexity'' is reminiscent of the terminology infratype and sup-cotype (see \cite{Pisier74}, \cite{Talagrand92}, \cite{Talagrand02}). The notion of infratype is obtained by relaxing the Rademacher type inequality by taking a minimum (instead of an average) over sign choices. Similarily for the notion of sup-cotype a maximum over sign choices replaces the traditional average. In our case, we replace the averages over levels by an infimum on the left-hand side and by a supremum on the right-hand side of \eqref{eq:umbel-p-convex}. First of all, it is obvious that infrasup-umbel $p$-convexity $p$ implies infrasup-umbel $q$-convexity for every $q>p$. Also, if the $\ell_p$-sum in \eqref{eq:superrelaxed-p-umbel} is replaced with an $\ell_\infty$-sum then the resulting inequality follows from the triangle inequality and holds in any metric space. We will then say that $\metX$ has \emph{non-trivial infrasup-umbel convexity} if it has infrasup-umbel convexity $p$ for some $p<\infty$.

We now record a lower bound on the infrasup-umbel convexity $p$ constant and on the distortion of countably branching trees.

\begin{prop}
\label{prop:tree-dist}
For all $p\in(0,\infty)$, $\Pi^{isu}_p \big( \tree_{2^k} \big)\ge 2(k-1)^{1/p} $, and hence $\cdist{\metY}\big(\tree_{k}\big)=\Omega \big( (\log k)^{1/p} \big)$
for every infrasup-umbel $p$-convex metric space $\metYd$.
\end{prop}

\begin{proof}
It suffices to consider $f$ to be the identity map on the tree. For all $\nbar\in \wtree{\bN}{\le 2^k-2^{s}}$, $i \neq j \in \bN$, and $\bar{\delta},\bar{\eta} \in \wtree{\bN}{2^{s}-1}$, it holds that
$\sd_\tree((\nbar, i, \bar{\delta}),(\nbar , j , \bar{\eta}))^p\ge 2^{(s+1)p}$, and thus
\begin{equation*}
\sum_{s=1}^{k-1}\inf_{\nbar\in\wtree{\bN}{\le 2^{k}-2^s}}\inf_{i \neq j \in \bN} \inf_{\stackrel{\bar{\delta},\bar{\eta}\in\wtree{\bN}{2^s-1}\colon}{(\nbar,i,\bar{\delta}),(\nbar,j,\bar{\eta})\in \wtree{\bN}{\le 2^k}}}\frac{\dY(f(\nbar,i,\bar{\delta}),f(\nbar,j,\bar{\eta}))^p}{2^{sp}}\ge (k-1)2^p,
\end{equation*}
which in turn implies that $\Pi^{isu}_p(\tree_{2^k}) \ge 2(k-1)^{1/p}$ since obviously $\lip(f)\le 1$.
\end{proof}

The lower bound above is known to be tight since $\ell_p$ has infrasup-umbel $p$-convex (it has property $(\beta_p)$) and it follows from Bourgain's tree embedding \cite{Bourgain86} (see also \cite{Matousek99} or \cite{BaudierZhang16}) that $\cdist{\ell_p} \big( \tree^\omega_k \big)=O \big( (\log k)^{1/p} \big)$. 

There are several instances where the saturation of a Poincar\'e-type inequality associated with the geometry of a graph implies the containment of a hardly distorted copy of the graph, most notably in the metric dichotomies regarding [BMW-metric type $\mid$ Hamming cubes] in \cite{BMW86}, [metric cotype $\mid$ $\ell_\infty$-discrete tori] in \cite{MN08}, and [diamond convexity $\mid$ diamond graphs] in \cite{EMN}. We refer to \cite{Mendel09} for a discussion of metric dichotomies. Even though in general there is no metric dichotomy (see \cite{MendelNaor13}) for tree metrics, it is likely possible that the saturation of deterministic Poincar\'e-type inequalities associated to binary trees (e.g. the binary tree convexity inequality in \cite[page 382]{LMN02}) induces embeddings with small distortion. In the setting of min/max Poincar\'e-type inequalities (as opposed to standard Poincar\'e-type inequalities which typically involve averages), the saturation argument is rather elementary to implement. For this purpose, observe that if we denote by $\Pi^{isu}_{p,k}(\metX)$ the least constant $C$ with which the infrasup-umbel $p$ inequality \eqref{eq:umbelcotypep} holds for all maps $f \colon \tree^\omega_{2^k} \to \metX$, then it follows from the triangle inequality that $\Pi^{isu}_{p,k}(\metX)\le 2(k-1)^{1/p}$. If for some map $f\colon \tree^\omega_{2^k}\to \met X$, the infrasup-umbel $p$ inequality is saturated, then $\metX$ will contain a barely distorted copy of $\tree^\omega_{2^k}$.

\begin{prop}
\label{prop:saturation}
Let $p\in(0,\infty)$. If $\Pi^{isu}_{p,k}(\metX) = 2(k-1)^{1/p}$ then $\cdist{\metX}\Big(\tree^\omega_{2^k}\Big)=1$. 
\end{prop}

\begin{proof}
Since we are assuming that $\Pi^{isu}_{p,k}(\metX) = 2(k-1)^{1/p}$, given any $\nu>0$ there is a map $f\colon \tree^\omega_{2^k}\to \met X$, such that 
	\begin{equation}
		\sum_{s=1}^{k-1}\inf_{\nbar\in\wtree{\bN}{\le 2^{k}-2^s}}\inf_{i \neq j \in \bN} \inf_{\stackrel{\bar{\delta},\bar{\eta}\in\wtree{\bN}{2^s-1}\colon}{(\nbar,i,\bar{\delta}),(\nbar,j,\bar{\eta})\in \wtree{\bN}{\le 2^k}}}\frac{\dX(f(\nbar,i,\bar{\delta}),f(\nbar,j,\bar{\eta}))^p}{2^{sp}} \ge (1-\nu)(k-1)2^p\lip(f)^p.
	\end{equation} 
For $s \in \{1, \dots, k-1\}$, $i \neq j\in\bN$, $\nbar\in\wtree{\bN}{\le 2^{k}-2^s}$, and $\bar{\delta},\bar{\eta}\in\wtree{\bN}{2^s-1}$ such that $(\nbar,i,\bar{\delta}),(\nbar,j,\bar{\eta})\in \wtree{\bN}{\le 2^k}$, it follows from the triangle inequality that 
	\begin{equation}\label{eq:wsataux1}
		\frac{\dX(f(\nbar,i,\bar{\delta}),f(\nbar,j,\bar{\eta}))^p}{2^{sp}} \le 2^p \lip(f)^p.
	\end{equation}
Therefore,
	\begin{equation}
		\frac{\dX(f(\nbar,i,\bar{\delta}),f(\nbar,j,\bar{\eta}))^p}{2^{sp}} \ge (1-\nu(k-1))2^p\lip(f)^p.
	\end{equation}
Since $(1-x)^{1/p}\ge 1-cx$ when $x\in(0,1)$ (take $c=1/p$ if $1/p>1$ and $c=1$ if $1/p\in(0,1]$), we have
	\begin{equation}\label{eq:wsataux2}
		\dX(f(\nbar,i,\bar{\delta}),f(\nbar,j,\bar{\eta})) \ge (1- c\nu (k-1))2^{s+1}\lip(f).
	\end{equation}
The combination of \eqref{eq:wsataux1} and \eqref{eq:wsataux2} gives that $f$ is a scaled-isometry (up to some small error) for pairs of vertices with equal height, i.e. of the form $(\nbar,i,\bar{\delta}),(\nbar,j,\bar{\eta})\in \wtree{\bN}{\le 2^k}$ where $i \neq j$. $\nbar\in\wtree{\bN}{\le 2^{k}-2^s}$, and $\bar{\delta},\bar{\eta}\in\wtree{\bN}{2^s-1}$ for some $s \in \{1, \dots, k-1\}$. More precisely, for such pairs of vertices we have $\sd_\tree((\nbar,i,\bar{\delta}),(\nbar, j, \bar{\eta})) = 2^{s+1}$ and 
	\begin{equation}\label{eq:wsataux3}
		 2^{s+1}\lip(f)-c\nu (k-1)2^{k}\lip(f) \le \dX(f(\nbar,i,\bar{\delta}),f(\nbar,j,\bar{\eta})) \le  2^{s+1}\lip(f).
	\end{equation}
It remains to estimate from below the distances between the images of an arbitrary pair of vertices. Let $\bar{u} \neq \bar{v} \in \tree^\omega_{2^k}$ and assume without loss of generality that we are in the following fork configuration:
\begin{figure}[H]
\label{fig:config}
\caption{Fork configuration}
\vskip 0.2cm
\hskip 0cm\includegraphics[scale=.5]{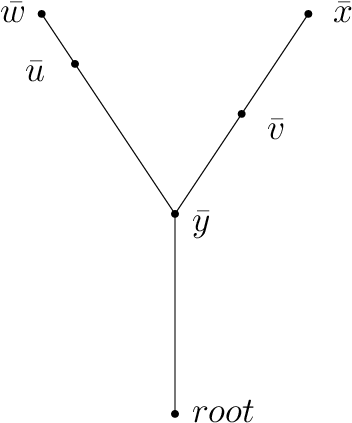}
\end{figure}

In the figure, $\bar{y}$ is the highest common ancestor of $\bar{u},\bar{v}$, and $\bar{w},\bar{x}$ are chosen so that $\sd_\tree(\bar{w},\bar{y}) = \sd_\tree(\bar{x},\bar{y})$ and both are even. We allow the possibility that $\bar{y}$ is the root, $\bar{y} = \bar{v}$, or $\bar{w} = \bar{u}$. We have
 	\begin{align*}
		\dX(f(\bar{u}),f(\bar{v})) & \ge \dX(f(\bar{w}),f(\bar{x})) - \dX(f(\bar{w}),f(\bar{u})) - \dX(f(\bar{v}),f(\bar{x}))\\
		& \stackrel{\eqref{eq:wsataux3}}{\ge} \sd_\tree(\bar{w},\bar{x})\lip(f)-c\nu(k-1)2^{k}\lip(f) - \sd_\tree(\bar{w},\bar{u})\lip(f) - \sd_\tree(\bar{v},\bar{x})\lip(f)\\
		& = \sd_\tree(\bar{u},\bar{v})\lip(f)-c\nu(k-1)2^{k}\lip(f)\\
		& \ge \sd_\tree(\bar{u},\bar{v})\lip(f)(1-c\nu(k-1)2^{k}).
	\end{align*}
Consequently, the distortion of $f$ is at most $\frac{1}{1-c\nu (k-1)2^{k}}$ which can be made as close to $1$ as we wished by choosing $\nu$ sufficiently small.
\end{proof}

The notion of infrasup-umbel convexity is not a coarse invariant, e.g. it was shown in \cite{BLS_JAMS} that the countably branching tree of infinite height embeds coarsely into every infinite-dimensional Banach space. However, it is a strong enough strengthening of the triangle inequality which provides estimates on the compression rate of coarse embeddings of countably branching trees. Having established that there are spaces which have non-trivial infrasup-umbel convexity, we can now derive Theorem \ref{thm:tree-compression} essentialy in the same way Tessera derived Theorem \ref{thm:Tessera} from inequality \eqref{eq:q-tess}. 

\begin{theo} 
\label{thm:tree-compression}
Let $p\in(0,\infty)$ and assume that there are non-decreasing maps $\rho,\omega\colon [0,\infty)\to [0,\infty)$ and for all $k\ge 1$ a map $f_k\colon \tree_{2^k}\to \metY$ such that for all $x,y\in \tree_{2^k}$,
\begin{equation*}
\rho(\sd_\tree(x,y))\le \dY(f_k(x),f_k(y))\le \omega(\sd_\tree(x,y)).
\end{equation*}
Then,
\begin{equation*}
\int_{1}^\infty \Big(\frac{\rho(t)}{t}\Big)^p\frac{dt}{t} \le \frac{2^p-1}{p}\Pi^{isu}_p(\metY)^p \omega(1)^p.
\end{equation*}
In particular, the compression rate of any equi-coarse embedding of $\{\tree_{m}\}_{m\ge 1}$ into an infrasup-umbel $p$-convex metric space satisfies  
\begin{equation}\label{eq:comp-bound}
\int_{1}^\infty \Big(\frac{\rho(t)}{t}\Big)^p\frac{dt}{t}<\infty.
\end{equation}
\end{theo}

\begin{proof}
Assume that $\metYd$ infrasup-umbel $p$-convex and let $C=\Pi^{isu}_p(\metY)$. Then,

\begin{align*}
\sum_{s=1}^{k-1} & \inf_{\nbar\in\wtree{\bN}{\le 2^{k}-2^s}}\inf_{i \neq j \in \bN} \inf_{\stackrel{\bar{\delta},\bar{\eta}\in\wtree{\bN}{2^s-1}\colon}{(\nbar,i,\bar{\delta}),(\nbar,j,\bar{\eta})\in \wtree{\bN}{\le 2^k}}}\frac{\dY(f_k(\nbar,i,\bar{\delta}),f_k(\nbar,j,\bar{\eta}))^p}{2^{sp}} \ge\\
\sum_{s=1}^{k-1} & \inf_{\nbar\in\wtree{\bN}{\le 2^{k}-2^s}}\inf_{i \neq j \in \bN} \inf_{\stackrel{\bar{\delta},\bar{\eta}\in\wtree{\bN}{2^s-1}\colon}{(\nbar,i,\bar{\delta}),(\nbar,j,\bar{\eta})\in \wtree{\bN}{\le 2^k}}}\frac{\rho(\sd_\tree((\nbar,i,\bar{\delta}) , (\nbar,j,\bar{\eta}) ) )^p}{2^{sp}} \ge
\sum_{s=1}^{k-1}\frac{\rho(2^{(s+1)})^p}{2^{sp}},
\end{align*}

and hence it follows from \eqref{eq:umbelcotypep} and the upper coarse inequality that

\begin{equation*}
\sum_{s=1}^{k-1}\frac{\rho(2^{s})^p}{2^{sp}} \le C^p \omega(1)^p.
\end{equation*} 

But, 

\begin{eqnarray*}
\int_{2^{s-1}}^{2^{s}} \frac{\rho(t)^p}{t^{p}}\frac{dt}{t} & \le & \rho(2^{s})^p \int_{2^{s-1}}^{2^{s}}\frac{dt}{t^{p+1}}= \rho(2^{s})^p\frac{2^{-(s-1)p}-2^{-sp}}{p} =  \frac{2^p-1}{p}\frac{\rho(2^{s})^p}{2^{sp}},
\end{eqnarray*}

and hence 

\begin{equation*}
\int_{1}^{2^{k-1}}\frac{\rho(t)^p}{t^{p+1}}dt=\sum_{s=1}^{k-1}\int_{2^{s-1}}^{2^{s}} \frac{\rho(t)^p}{t^{p}}\frac{dt}{t}\le \frac{2^p-1}{p} \sum_{s=1}^{k-1}\frac{\rho(2^{s})^p}{2^{sp}}\le \frac{2^p-1}{p} C^p \omega(1)^p<\infty. 
\end{equation*}
\end{proof}

It is well known that Banach spaces of the form $(\sum_{n=1}^\infty \banF_n)_{\ell_p}$, where $p\in (1,\infty)$ and $\{\banF_n\}_{n\ge 1}$ is a sequence of finite-dimensional spaces, have property $(\beta_p)$ (see \cite[Proposition 5.1]{DKLR14}), and thus they are infrasup-umbel $p$-convex and Theorem \ref{thm:tree-compression} applies. No bounds such as \eqref{eq:comp-bound} were previously known for the countably branching trees, even for those simple Banach spaces.

An interesting application to hyperbolic geometry is the following. It is well known that the infinite binary tree admits a bi-Lipschitz embedding into the hyperbolic plane $\mathbf{H}^2$ (an almost isometric embedding of every finite weighted tree can be found in \cite{Sarkar12}). It follows from \cite[Theorem 1.1 (ii)]{BIM05} that $\tree_\infty^\omega$ admits a bi-Lipschitz embedding into the hyperbolic space $\mathbf{H}^\infty$ of countably infinite dimension. The importance of studying infinite-dimensional hyperbolic spaces was put forth by Gromov in \cite[Section 6]{Gromov93}. The geometry of binary trees, via either Markov convexity or Bourgain's metric characterization, can be used to show that $\cdist{\ell_2}\big(\mathbf{H^\infty}\big)\ge \cdist{\ell_2}\big(\mathbf{H^2}\big)=\infty$. Since $\xbl$ contains a bi-Lipschitz copy of $\bin{\infty}$, the geometry of binary trees does not provide any obstruction in such spaces. Nevertheless, resorting to the geometry of countably branching trees we can conclude that $\cdist{\banY}\big(\mathbf{H^\infty}\big)=\infty$ when $\banY$ is any Banach space of the form $(\sum_{n=1}^\infty \banF_n)_{\ell_p}$, where $p\in(1,\infty)$. Similar arguments give restrictions on the coarse compression rate for the infinite-dimensional hyperbolic space.

\begin{coro}\label{cor:hyperbolic}
Let $\mathbf{H^\infty}$ be the infinite-dimensional hyperbolic space and $\metY$ be an infrasup-umbel $p$-convex metric space with $p\in(0,\infty)$. Then, the compression rate of any coarse embedding of $\mathbf{H^\infty}$ into $\metY$ satisfies 
		\begin{equation*}
\int_{1}^\infty \Big(\frac{\rho(t)}{t}\Big)^p\frac{dt}{t}<\infty.
\end{equation*}
In particular, $\cdist{\metY}\big(\mathbf{H^\infty}\big)=\infty$.
\end{coro}

The tightness of Theorem \ref{thm:tree-compression} follows from \cite[Theorem 7.3]{Tessera08}. It turns out that Bourgain's tree embedding, which takes value into $\ell_p$-spaces, can be extended to target spaces containing $\ell_p$ in some asymptotic fashion. It is rather straightforward to show that $\cdist{\banY}(\tree^\omega_k)=O((\log k)^{1/p})$ if $\banY$ has an $\ell_p$-spreading model generated by a weakly-null sequence. To show that the same bound holds for the larger class of Banach spaces admitting an $\ell_p$-asymptotic model  generated by a weakly-null array requires a bit more care and a recent observation from \cite{BLMS_IJM}. Our embedding is an adjustment of Bourgain's tree embedding, but in this context new complications arise when estimating the co-Lipschitz constant.

We refer to \cite{BLMS_JIMJ} for a discussion of the relationship between spreading models, asymptotic models, and asymptotic structure. Here it suffices to say that $\banY$ has an $\ell_p$-asymptotic model generated by a weakly-null array if there exists a normalized weakly-null array $\big(y^{(i)}_j:i,j\kin\bN\big)$  in $\banY$ such that for all $k\in\bN$ and $\delta>0$, we may pass to appropriate subsequences of the array so that for any $k\le j_1<\cdots<j_k$ and any $a_1,\ldots,a_k$ in $[-1,1]$ we have
\begin{equation}
\label{eq:am}
\Bigg|\Big\|\sum_{i=1}^ka_ix^{(i)}_{j_i}\Big\| - \Big(\sum_{i=1}^k\abs{a_i}^p\Big)^{\frac{1}{p}}\Bigg| < \delta.
\end{equation}
The extreme cases in the proposition below extend prior results obtained in \cite{BLS_JAMS} for spreading models.
 
\begin{prop}\label{prop:am-dist}
If $\banY$ has an $\ell_p$-asymptotic model generated by a weakly-null array for some $p\in(1,\infty)$, then
\begin{equation*}
\cdist{\banY}(\tree^\omega_k)=O\Big((\log k)^{\frac{1}{p}}\Big).
\end{equation*}
If $\banY$ has an $\ell_1$-asymptotic model or a $\co$-asymptotic model generated by a weakly-null array  then $\sup_{k\in\bN} \cdist{\banY}\big(\tree_k^\omega\big)<\infty$.
\end{prop}

\begin{proof}
Let $k\ge 1$ and fix a compatible bijection $\Phi\colon \wtree{\bN}{\le k}\to \{2k,2k+1,\dots\}$, meaning $\Phi((n_1,n_2,\dots, n_\ell)) \leq \Phi((n_1,n_2,\dots, n_\ell,n_{\ell+1}))$ for all $(n_1,n_2,\dots, n_\ell,n_{\ell+1}) \in [\bN]^{\leq k}$. In addition to \eqref{eq:am} and by applying \cite[Lemma 3.8]{BLMS_IJM}, we may also assume that for any $i_1,\ldots,i_{2k}$ in $\{1,\ldots,k\}$ and any pairwise different $l_1,\ldots,l_{2k}$ in $\bN$, the sequence $(y^{(i_j)}_{l_j})_{j=1}^{2k}$ is $(1+\delta)$-suppression unconditional for $\delta>0$ arbitrarily small. Define a Bourgain-style map $f\colon (\wtree{\bN}{\le k},\sd_\tree)\to \banY$ by 
\begin{equation*}
f(n_1,\dots,n_j)=\sum_{i=0}^j(j-i+1)^{\frac{1}{q}} y^{(i)}_{\Phi(n_1,\dots,n_i)},
\end{equation*}
where $\frac 1q+\frac 1p=1$, and where it is understood that for $i=0$, $y^{(i)}_{\Phi(n_1,\dots,n_i)}=y^{(i)}_{\Phi(\emptyset)}$.
Consider $\nbar,\mbar \in \wtree{\bN}{\le k}$ such that $\nbar=(\ubar,n_1,\dots,n_j)$ and $\mbar=(\ubar,m_{1},\dots, m_{h})$ for some $\ubar\in\wtree{\bN}{s}$ with $s\le k-j$ and $j\ge h$. Then,
\begin{align*}
\norm{f(\nbar)-f(\mbar)}_\banY & = \bnorm{\sum_{i=0}^s \underbrace{\Big((s+j-i+1)^{\frac{1}{q}}-(s+h-i+1)^{\frac{1}{q}}\Big)}_{\alpha_i}y^{(i)}_{\Phi(u_1,\dots, u_i)}+\\
 					&   \sum_{i=1}^j \underbrace{\Big(s+j-(s+i)+1\Big)^{\frac{1}{q}}}_{\beta_i}y^{(s+i)}_{\Phi(\ubar, n_1,\dots, n_i)}-\sum_{i=1}^h \underbrace{\Big(s+h-(s+i)+1\Big)^{\frac{1}{q}}}_{\gamma_i} y^{(s+i)}_{\Phi(\ubar, m_1,\dots, m_i)}}_\banY\\
					& \le \bnorm{\sum_{i=0}^s\alpha_iy^{(i)}_{\Phi(u_1,\dots, u_i)}}_\banY + \bnorm{\sum_{i=1}^j \beta_i y^{(s+i)}_{\Phi(\ubar, n_1,\dots, n_i)}}_\banY+\bnorm{\sum_{i=1}^h \gamma_i y^{(s+i)}_{\Phi(\ubar, m_1,\dots, m_i)}}_\banY.
\end{align*}
Recall that for all $y>x>0$ and $a\in(0,1)$,
\begin{equation}\label{eq:useful-ineq2}
y^{a}-x^{a}\le \frac{y-x}{y^{1-a}}.
\end{equation}
Observe now that $\max\{s,s+j,s+h\}\le k$ and since $\Phi$ is a compatible bijection taking values into $\{2k, 2k+1,\dots\}$ it follows from \eqref{eq:am} that
\begin{align*}
\bnorm{\sum_{i=0}^s\alpha_iy^{(i)}_{\Phi(u_1,\dots, u_i)}}_\banY & \le \Big(\sum_{i=1}^s \alpha_i^p\Big)^{\frac{1}{p}} +\delta \le \Big(\sum_{i=0}^s \Big((s+j-i+1)^{\frac{1}{q}}-(s+h-i+1)^{\frac{1}{q}}\Big)^p\Big)^{\frac{1}{p}} +\delta\\
		  & \stackrel{\eqref{eq:useful-ineq2}}{\le} \Big(\sum_{i=0}^s \Big( \frac{s+j-i+1-(s+h-i+1)}{(s+j-i+1)^{1-\frac{1}{q}}}\Big)^p\Big)^{\frac{1}{p}} +\delta\\
		  & \le \Big( \sum_{i=0}^s \frac{(j-h)^p}{s+j-i+1}\Big)^{\frac{1}{p}} +\delta= (j-h)\Big(\sum_{i=j+1}^{s+j+1} \frac{1}{i}\Big)^{\frac{1}{p}}+\delta= O_\delta((j-h)(\log k)^{\frac{1}{p}}).
\end{align*}
Also,
\begin{align*}
\bnorm{\sum_{i=0}^j\beta_i y^{(i)}_{\Phi(u,n_1,\dots, n_i)}}_\banY & \le \Big(\sum_{i=1}^j \beta_i^p\Big)^{\frac{1}{p}} +\delta \le  \Big(\sum_{i=1}^j (j-i+1)^{\frac{p}{q}}\Big)^{\frac{1}{p}}+\delta \\
			& = \Big(\sum_{i=1}^{j} i^{\frac{p}{q}}\Big)^{\frac{1}{p}}+\delta = O(j)+\delta\le O_\delta(j)
\end{align*}
and a similar computation gives 
\begin{equation*}
\bnorm{\sum_{i=1}^h \gamma_i y^{(s+i)}_{\Phi(\ubar, m_1,\dots, m_i)}}_\banY  = O_\delta(h).
\end{equation*}
Therefore, 
\begin{equation*}
\norm{f(\nbar)-f(\mbar)}_\banY \le  O_\delta((j+h)(\log k)^{\frac{1}{p}})=O_\delta((\log k)^{\frac{1}{p}})\sd_\tree(\nbar,\mbar).
\end{equation*}
For the lower bound, it follows from the suppression unconditionally condition that 
\begin{align*}
\norm{f(\nbar)-f(\mbar)}_\banY & \ge   \frac{1}{1+\delta}\bnorm{\sum_{i=1}^j \beta_i y^{(s+i)}_{\Phi(\ubar, n_1,\dots, n_i)}}_\banY \ge  \frac{1}{1+\delta}\Big(\sum_{i=1}^j \beta_i^p\Big)^{\frac{1}{p}} -\frac{\delta}{1+\delta}\\
 		& \ge  \frac{1}{1+\delta}\Big(\sum_{i=1}^j (j-i+1)^{\frac{p}{q}}\Big)^{\frac{1}{p}}-\frac{\delta}{1+\delta} \\
		&=  \frac{1}{1+\delta}\Big(\sum_{i=1}^{j} i^{\frac{p}{q}}\Big)^{\frac{1}{p}}-\frac{\delta}{1+\delta} = \frac{1}{1+\delta}\Omega(j)-\frac{\delta}{1+\delta} \ge \Omega_\delta(j),
\end{align*}
and similarly 
\begin{equation*}
\norm{f(\nbar)-f(\mbar)}_\banY\ge \Omega_\delta(h).
\end{equation*}
Therefore,
\begin{equation*}
\norm{f(\nbar)-f(\mbar)}_\banY \ge \Omega_\delta((j+h))=\Omega_\delta(\sd_\tree(\nbar,\mbar))
\end{equation*}
and the conclusion follows.

For the case $p=\infty$ the map $f$ takes the form $f(n_1,\dots,n_j)=\sum_{i=0}^j (j-i+1)y^{(i)}_{\Phi(n_1,\dots,n_i)}$ and the argument above gives a bounded distortion. In the case $p=1$, it can easily be verified that the map $f\colon (\wtree{\bN}{\le k},\sd_\tree)\to \banY$ given by 
\begin{equation*}
f(n_1,\dots,n_j)=\sum_{i=0}^j y^{(i)}_{\Phi(n_1,\dots,n_i)},
\end{equation*}
is a bi-Lipschitz embedding.
\end{proof}


\begin{coro} 
\label{cor:betap->lq-asm}
If $\banX$ is infrasup-umbel $p$-convex for some $p\in(1,\infty)$, then $\banX$ does not have any $\ell_q$-asymptotic model generated by a weakly-null array for any $q>p$.
\end{coro}

\begin{proof}
By Proposition \ref{prop:am-dist}, $\tree_k^\omega$ embeds into $\banX$ with distortion at most $O((\log k)^{1/q})$, but this impossible by Proposition \ref{prop:tree-dist}.
\end{proof}

\section{Stability under nonlinear quotients}
\label{sec:quotients}
Recall that a map $f\colon \metXd\to \metYd$ between metric spaces is called a \emph{Lipschitz quotient map}, and $\metY$ is simply said to be a \emph{Lipschitz quotient} of $\metX$, if there exist constants $L,C>0$ such that for all $x\in \metX$ and $r\in(0,\infty)$ one has
\begin{equation}\label{eq:lipquotient}
B_{\met Y}\Big(f(x),\frac{r}{C}\Big)\subset f(B_\metX(x,r))\subset B_{\met Y}(f(x),Lr).
\end{equation}
Note that the right inclusion in \eqref{eq:lipquotient} is equivalent to $f$ being Lipschitz with $\lip(f)\le L$. If the left inclusion in \eqref{eq:lipquotient} is satisfied, then $f$ is said to be \emph{co-Lipschitz}, and the infimum of all such $C$'s, denoted by $\colip(f)$, is called the co-Lipschitz constant of $f$. We define the \emph{codistortion} of a Lipschitz quotient map $f$ as $\codist(f)\eqd\lip(f)\cdot\colip(f)$. A metric space $\met Y$ is said to be a \emph{Lipschitz subquotient} of $\met X$ with codistortion $\alpha\in[1,\infty)$ (or simply $\met Y$ is an $\alpha$-Lipschitz subquotient of $\met X$) if there is a subset $Z\subset \met X$ and a Lipschitz quotient map $f\colon Z\to \met Y$ such that $\codist(f)\le\alpha$. We define the $\met X$-quotient codistortion of $\met Y$ as
\begin{equation*}
\mathsf{qc}_{\metX}(\met Y) \eqd \inf\{\alpha\colon \met Y\textrm{ is an $\alpha$-Lipschitz subquotient of } \metX\}.
\end{equation*}
We set $\mathsf{qc}_\metX(\met Y)=\infty$ if $\met Y$ is not a Lipschitz quotient of any subset of $\metX$.

As is the case for Markov $p$-convexity, umbel $p$-convexity and its relaxations are also stable under taking Lipschitz quotients. 

\begin{prop} \label{prop:lipsubquot}
Let $p\in(0,\infty)$ and $\metXd$ be a metric space that is umbel $p$-convex. If $\metY$ is a Lipschitz subquotient of $\met X$ then $\met Y$ is umbel $p$-convex. Moreover, $\Pi^u_p(X)\le \mathsf{qc}_{\metX}(Y)\Pi^u_p(\metX)$.
\end{prop}

We omit the proof of Proposition \ref{prop:lipsubquot} as it can be extracted from the more delicate argument given in Proposition \ref{prop:gen-quotient} below.
 
In the Banach space setting, umbel $p$-convexity is also stable under more general notions of nonlinear quotients, most notably uniform quotients or coarse quotients, as defined in \cite{BJLPS99} and \cite{Zhang15} respectively. We will treat these nonlinear quotients all at once, and we need to introduce some more notation. The $K$-neighborhood of a set $A$ in a metric space $\metXd$, denoted $A_K$, is the set $A_K\eqd \{ z\in \metX \colon \exists a\in A \text{ such that }\dX(z,a)\le K\}$. 
The following simple general lifting lemma will be crucial in the ensuing arguments about nonlinear quotients. 

\begin{lemm}\label{lem:general-lifting}
Let $f\colon Z\subseteq \metX \to \metY$ and $g\colon \wtree{\bN}{\le m} \to \metY$, where $g$ is any map and $f$ is a map such that there exist constant $C>0$ and $K\ge 0$ with $\metY=f(Z)_K$, and for all $x \in Z$ and  $r>0$,
\begin{equation*}
B_{\met Y}\Big(f(x),\frac{r}{C}\Big)\subset f(B_\metX(x,r)\cap Z)_K.
\end{equation*}
Then, there is a map $h\colon \wtree{\bN}{\le m} \to Z$ such that for all $\nbar \in \wtree{\bN}{\le m}$,
\begin{equation}\label{eq:gen-upper-lift}
\dX(h(n_1,\dots,n_k),h(n_1,\dots,n_{k-1})\le C \cdot \sd_{\metY}(g(n_1,\dots,n_k),g(n_1,\dots,n_{k-1})) + CK
\end{equation}
and 
\begin{equation}
\label{eq:gen-id-lift}\dY(f(h(\nbar)),g(\nbar))\le K.
\end{equation}
\end{lemm}

\begin{proof}
The proof is a simple induction on $m$. If $m=0$, let $y\in f(Z)$ such that $\dY(g(\troot),y)\le K$, pick an arbitrary $z\in Z$ such that $f(z)=y$, and then let $h(\emptyset)\eqd z$. Obviously, $\dY(f(h(\troot)),g(\troot))\le K$ and the other condition is vacuously true. Assume that the map $h$ has been constructed on $\wtree{\bN}{\le m}$. We extend $h$ to $\wtree{\bN}{\le m+1}$ as follows. Given $\nbar \in \wtree{\bN}{m}$ and $n_{m+1}\in\bN$, let $r\eqd \sd_{\metY}(g(\nbar),g(\nbar,n_{m+1}))$. Since $\dY(f(h(\nbar)),g(\nbar))\le K$ we have 
\begin{equation*}
g(\nbar,n_{m+1})\in B_{\metY}(f(h(\nbar)),r+K)\subseteq f(B_{\metX}(h(\nbar), C(r+K))\cap Z)_K.
\end{equation*}
Let $y\in f(B_{\metX}(h(\nbar), C(r+K)) \cap Z)\subseteq \metY$ such that $\dY(y,g(\nbar,n_{m+1}))\le K$, then pick arbitrarily $z\in B_{\metX}(h(\nbar), C(r+K)) \cap Z$ such that $f(z)=y$, and finally set $h(\nbar,n_{m+1})\eqd z$ from which it immediately follows that 
$$\dY(f(h(\nbar,n_{m+1})),g(\nbar,n_{m+1}))\le K.$$ 
Finally, observe that by definition 
$$\dX(h(\nbar),h(\nbar,n_{m+1})\le C(r+K) = C\sd_{\met Y}(g(\nbar),g(\nbar,n_{m+1})) + CK.$$
\end{proof}

\begin{prop}
\label{prop:gen-quotient}
Let $\metYd$ be a self-similar\footnote{A metric space $\metXd$ is \emph{self-similar} if for every $t > 0$, there exists a bijection $\delta_t \colon \metX \to \metX$ with $\dX(\delta_t(x),\delta_t(y)) = t \cdot \dX(x,y)$ for every $x,y \in \metX$.} metric space. Assume that there is a map $f\colon Z\subseteq \metXd \to \metY$, that is coarse Lipschitz, i.e., there exist $L>0$ and $A\ge 0$ such that for all $x,y\in Z$ 
\begin{equation}
\label{eq:LLD}
\dY(f(x),f(y))\le L\dX(x,y) + A.
\end{equation}
Assume also that there are constant $C>0$ and $K\ge 0$ with $\metY=f(Z)_K$, such that for all $x \in Z$ and  $r>0$,
\begin{equation}
\label{eq:co-condition}
B_{\met Y}\Big(f(x),\frac{r}{C}\Big)\subset f(B_\metX(x, r)\cap Z)_K.
\end{equation}
If $\metX$ is umbel $p$-convex for some $p \in (0,\infty)$, then $\metY$ is umbel $p$-convex.
\end{prop}

\begin{proof}
Let $f: Z\subseteq \metX \to \metY$ be a map as above. We need to show that there exists a constant $\Pi>0$ such that for every map $g: \wtree{\bN}{\le 2^k} \to \metY$,
\begin{align*}
\sum_{s=1}^{k-1}\frac{1}{2^{k-1-s}}\sum_{t=1}^{2^{k-1-s}}\inf_{\nbar\in\wtree{\bN}{t2^{s+1}-2^s}}\inf_{\stackrel{\bar{\delta}\in\wtree{\bN}{2^s}\colon}{(\nbar,\bar{\delta})\in \wtree{\bN}{\le 2^k}}}\liminf_{j\to\infty}\inf_{\stackrel{\bar{\eta}\in\wtree{\bN}{2^s-1}\colon}{(\nbar,j,\bar{\eta})\in \wtree{\bN}{\le 2^k}}}\frac{\dY(g(\nbar,\bar{\delta}),g(\nbar,j,\bar{\eta}))^p}{2^{sp}}\\
\le \Pi^p \frac{1}{2^{k}}\sum_{\ell=1}^{2^{k}} \sup_{\nbar\in \wtree{\bN}{\ell}}\dY(g(n_1,\dots,n_{\ell-1}),g(n_1,\dots,n_{\ell}))^p.
\end{align*} 
Observe that if the right-hand side vanishes, then the left-hand side vanishes as well and there is nothing to prove. Then by scale-invariance of the inequality and the self-similarity of $\metY$, we may assume
$$\frac{1}{2^{k}}\sum_{\ell=1}^{2^{k}} \sup_{\nbar\in \wtree{\bN}{\ell}}\dY(g(n_1,\dots,n_{\ell-1}),g(n_1,\dots,n_{\ell}))^p=1.$$
Let $\Pi = \Pi_p^\beta(\metX)$. Then by umbel $p$-convexity of $\metX$ applied to $h: \wtree{\bN}{\le 2^k} \to Z\subset \metX$, where $h$ is the lifting of $g$ as defined in Lemma \ref{lem:general-lifting}, we have
\begin{align}
\label{eq:h}
\nonumber \sum_{s=1}^{k-1}\frac{1}{2^{k-1-s}}\sum_{t=1}^{2^{k-1-s}}\inf_{\nbar\in\wtree{\bN}{t2^{s+1}-2^s}}\inf_{\stackrel{\bar{\delta}\in\wtree{\bN}{2^s}\colon}{(\nbar,\bar{\delta})\in \wtree{\bN}{\le 2^k}}}\liminf_{j\to\infty}\inf_{\stackrel{\bar{\eta}\in\wtree{\bN}{2^s-1}\colon}{(\nbar,j,\bar{\eta})\in \wtree{\bN}{\le 2^k}}}\frac{\dX(h(\nbar,\bar{\delta}),h(\nbar,j,\bar{\eta}))^p}{2^{sp}}\\
\le \Pi^p \frac{1}{2^{k}}\sum_{\ell=1}^{2^{k}} \sup_{\nbar\in \wtree{\bN}{\ell}}\dX(h(n_1,\dots,n_{\ell-1}),h(n_1,\dots,n_{\ell}))^p.
\end{align} 
It follows from \eqref{eq:gen-upper-lift} that 
\begin{align}
\label{eq:h2}
\nonumber \Pi^p \frac{1}{2^{k}}\sum_{\ell=1}^{2^{k}} & \sup_{\nbar\in \wtree{\bN}{\ell}}\dX(h(n_1,\dots,n_{\ell-1}),h(n_1,\dots,n_{\ell}))^p \\
\nonumber & \le \Pi^p \frac{1}{2^{k}}\sum_{\ell=1}^{2^{k}} \sup_{\nbar\in \wtree{\bN}{\ell}}(C\dY(g(n_1,\dots,n_{\ell-1}),g(n_1,\dots,n_{\ell}))+CK)^p\\
 & \le \Pi^p\max\{1,2^{p-1}\}(C^p+(CK)^p).
\end{align} 
Let $1\le s\le k-1$ and $1 \le t \leq 2^{k-1-s}$. Then either
\begin{equation}\label{eq:gen-case1}
\inf_{\nbar\in\wtree{\bN}{t2^{s+1}-2^s}}\inf_{\stackrel{\bar{\delta}\in\wtree{\bN}{2^s}\colon}{(\nbar,\bar{\delta})\in \wtree{\bN}{\le 2^k}}}\liminf_{j\to\infty}\inf_{\stackrel{\bar{\eta}\in\wtree{\bN}{2^s-1}\colon}{(\nbar,j,\bar{\eta})\in \wtree{\bN}{\le 2^k}}}\frac{\dY(g(\nbar,\bar{\delta}),g(\nbar,j,\bar{\eta}))^p}{2^{sp}}\le \frac{(L+A+ 2K)^p}{2^{sp}}
\end{equation}
or
\begin{equation}\label{eq:gen-case2}
\inf_{\nbar\in\wtree{\bN}{\le t2^{s+1}-2^{s}}}\inf_{\stackrel{\bar{\delta}\in\wtree{\bN}{2^s}\colon}{(\nbar,\bar{\delta})\in \wtree{\bN}{\le 2^k}}}\liminf_{j\to\infty}\inf_{\stackrel{\bar{\eta}\in\wtree{\bN}{2^s-1}\colon}{(\nbar,j,\bar{\eta})\in \wtree{\bN}{\le 2^k}}}\frac{\dY(g(\nbar,\bar{\delta}),g(\nbar,j,\bar{\eta}))^p}{2^{sp}}> \frac{(L+ A + 2K)^p}{2^{sp}}.
\end{equation}
If \eqref{eq:gen-case2} holds, then for all $\nbar\in\wtree{\bN}{t2^{s+1}-2^{s}}$ and $\bar{\delta}\in\wtree{\bN}{2^s}$, we have $$\liminf_{j\to\infty}\inf_{\stackrel{\bar{\eta}\in\wtree{\bN}{2^s-1}\colon}{(\nbar,j,\bar{\eta})\in \wtree{\bN}{\le 2^k}}}\dY(g(\nbar,\bar{\delta}),g(\nbar,j,\bar{\eta}))>  L + A + 2K,$$
and thus there exists $j_0$ such that for all $j\ge j_0$ and all $\bar{\eta}\in\wtree{\bN}{2^s-1}$ we have 
\[\dY(g(\nbar,\bar{\delta}),g(\nbar,j,\bar{\eta}))> L + A + 2K.\]
It follows from triangle inequality and \eqref{eq:gen-id-lift} that
\begin{align*}
\dY(f(h(\nbar,\bar{\delta})),f(h(\nbar,j,\bar{\eta}))) & \ge \dY(g(\nbar,\bar{\delta}),g(\nbar,j,\bar{\eta})) - \dY(f(h(\nbar,\bar{\delta})),g(\nbar,\bar{\delta}))) \\
									& \quad  - \dY(g(\nbar, j, \bar{\eta})),f(h(\nbar,j,\bar{\eta})))\\
									& \ge L + A + 2K - K - K \\
									& = L + A.
\end{align*}
Observe now that $\dY(f(x),f(y))< L+A$ whenever $\dX(x,y) < 1$ and based on the inequality above, necessarily $\dX(h(\nbar,\bar{\delta})),h(\nbar,j,\bar{\eta})))\ge 1$. Thus in this case, it follows from \eqref{eq:gen-id-lift} and \eqref{eq:LLD} that 
\begin{align*}
\dY(g(\nbar,\bar{\delta}),g(\nbar,j,\bar{\eta})) & \le \dY(f(h(\nbar,\bar{\delta})),f(h(\nbar,j,\bar{\eta}))) + 2K\\
								& \le L \dX(h(\nbar,\bar{\delta}),h(\nbar,j,\bar{\eta})) + A + 2K\\
								& \le (L+ A + 2K)\dX(h(\nbar,\bar{\delta}),h(\nbar,j,\bar{\eta})).
\end{align*}
Then, letting $\gamma\eqd (L + A + 2K)$ for simplicity,
\begin{align}\label{eq:gen-auxquot}
\nonumber \inf_{\nbar\in\wtree{\bN}{t2^{s+1}-2^{s}}} & \inf_{\stackrel{\bar{\delta}\in\wtree{\bN}{2^s}\colon}{(\nbar,\bar{\delta})\in \wtree{\bN}{\le 2^k}}}\liminf_{j\to\infty}\inf_{\stackrel{\bar{\eta}\in\wtree{\bN}{2^s-1}\colon}{(\nbar,j,\bar{\eta})\in \wtree{\bN}{\le 2^k}}}\frac{\dY(g(\nbar,\bar{\delta}),g(\nbar,j,\bar{\eta}))^p}{2^{sp}}\\
 		& \le \gamma^p\inf_{\nbar\in\wtree{\bN}{t2^{s+1}-2^{s}}}\inf_{\stackrel{\bar{\delta}\in\wtree{\bN}{2^s}\colon}{(\nbar,\bar{\delta})\in \wtree{\bN}{\le 2^k}}}\liminf_{j\to\infty}\inf_{\stackrel{\bar{\eta}\in\wtree{\bN}{2^s-1}\colon}{(\nbar,j,\bar{\eta})\in \wtree{\bN}{\le 2^k}}}\frac{\dX(h(\nbar,\bar{\delta}),h(\nbar,j,\bar{\eta}))^p}{2^{sp}}.
\end{align}
Consequently,
\begin{align*}
&\sum_{s=1}^{k-1} \frac{1}{2^{k-1-s}}\sum_{t=1}^{2^{k-1-s}}\inf_{\nbar\in\wtree{\bN}{t2^{s+1}-2^{s}}}\inf_{\stackrel{\bar{\delta}\in\wtree{\bN}{2^s}\colon}{(\nbar,\bar{\delta})\in \wtree{\bN}{\le 2^k}}}\liminf_{j\to\infty}\inf_{\stackrel{\bar{\eta}\in\wtree{\bN}{2^s-1}\colon}{(\nbar,j,\bar{\eta})\in \wtree{\bN}{\le 2^k}}}\frac{\dY(g(\nbar,\bar{\delta}),g(\nbar,j,\bar{\eta}))^p}{2^{sp}} \stackrel{\eqref{eq:gen-case1} \land \eqref{eq:gen-auxquot}}{\le}\\
& \sum_{s=1}^{k-1}\frac{1}{2^{k-1-s}}\sum_{t=1}^{2^{k-1-s}} \max\Big\{\frac{\gamma^p}{2^{sp}}, \gamma^p \inf_{\nbar\in\wtree{\bN}{t2^{s+1}-2^{s}}}\inf_{\stackrel{\bar{\delta}\in\wtree{\bN}{2^s}\colon}{(\nbar,\bar{\delta})\in \wtree{\bN}{\le 2^k}}}\liminf_{j\to\infty}\inf_{\stackrel{\bar{\eta}\in\wtree{\bN}{2^s-1}\colon}{(\nbar,j,\bar{\eta})\in \wtree{\bN}{\le 2^k}}}\frac{\dX(h(\nbar,\bar{\delta}),h(\nbar,j,\bar{\eta}))^p}{2^{sp}}\Big\}\\
& \le \gamma^p \sum_{s=1}^{k-1}\frac{1}{2^{k-1-s}}\sum_{t=1}^{2^{k-1-s}}\inf_{\nbar\in\wtree{\bN}{t2^{s+1}-2^{s}}}\inf_{\stackrel{\bar{\delta}\in\wtree{\bN}{2^s}\colon}{(\nbar,\bar{\delta})\in \wtree{\bN}{\le 2^k}}}\liminf_{j\to\infty}\inf_{\stackrel{\bar{\eta}\in\wtree{\bN}{2^s-1}\colon}{(\nbar,j,\bar{\eta})\in \wtree{\bN}{\le 2^k}}}\frac{\dX(h(\nbar,\bar{\delta}),h(\nbar,j,\bar{\eta}))^p}{2^{sp}} + \sum_{s=1}^{k-1} \frac{\gamma^p}{2^{sp}}\\
 & \stackrel{\eqref{eq:h} \land \eqref{eq:h2}}{\le}  \gamma^p\Pi^p\max\{1,2^{p-1}\}(C^p+(CK)^p) + \sum_{s=1}^{\infty} \frac{\gamma^p}{2^{sp}} < \infty,
\end{align*} 
which concludes the proof since $ \frac{1}{2^{k}}\sum_{\ell=1}^{2^{k}} \sup_{\nbar\in \wtree{\bN}{\ell}}\norm{g(n_1,\dots,n_{\ell-1})-g(n_1,\dots,n_{\ell})}_\banY^p=1$, and the constant 
\[(L + A + 2K)^p\left(\Pi^p\max\{1,2^{p-1}\}(C^p+(CK)^p) + \sum_{s=1}^{\infty} \frac{1}{2^{sp}}\right) \] is independent of $k$ and $g$.
\end{proof}

Note that a Lipschitz subquotient map satisfies the assumptions of Proposition \ref{prop:gen-quotient} with $L=\lip(f)$, $A=0$, $C=\colip(f)$, $K=0$, and the proof of Proposition \ref{prop:lipsubquot} can be simplified and carried over for arbitrary metric spaces (without the self-similarity assumption). The more general notions of nonlinear quotients which we will consider satisfy the hypotheses of Proposition \ref{prop:gen-quotient} under further assumptions on the metric spaces. 

A map $f\colon \metXd\to \metYd$ between metric spaces is called a \emph{uniform quotient map}, and $\metY$ is simply said to be a \emph{uniform quotient} of $\metX$, if $f$ is surjective, uniformly continuous and \emph{co-uniformly continuous}, i.e., for every $r>0$ there exists $\delta(r)>0$ such that for all $x\in \metX$, one has
\begin{equation*}
\label{eq:uniquotient}
B_{\met Y}\left(f(x),\delta(r)\right)\subset f(B_\metX(x,r)).
\end{equation*} 
It is a standard fact that a co-uniformly continuous map into a connected space is surjective. 

The more recent notion of coarse quotient introduced in \cite{Zhang15} is the following. A map $f\colon \metXd\to \metYd$ between metric spaces is called a \emph{coarse quotient map}, and $\metY$ is simply said to be a \emph{coarse quotient} of $\metX$, if $f$ is coarsely continuous and \emph{co-coarsely continuous with constant $K$} for some $K \ge 0$, i.e. for every $r> 0$ there exists $\delta(r)>0$ such that for all $x\in \metX$, one has
\begin{equation*}
\label{eq:coarsequotient}
B_{\met Y}\left(f(x),r)\right)\subset f(B_\metX(x,\delta(r)))_K.
\end{equation*}  
A co-coarsely continuous map may not be surjective, but nevertheless it is easily seen to be $K$-dense in the sense that $\metY = f(\metX)_K$. In fact, it can be shown, using a very clever argument due to Bill Johnson (see \cite{Zhang15}), that if a Banach space $\banY$ is a coarse quotient of a Banach space $\banX$, then there exists a coarse quotient mapping with vanishing constant $K=0$ from $\banX$ onto $\banY$. 

It is a standard fact that a map on a metrically convex space that is either uniformly continuous or coarsely continuous, is automatically coarse Lipschitz (one can take for instance $L=\max\{1,2\omega_f(1)\}$ and $c=\omega_f(1)$ where $\omega_f$ is the expansion modulus). Also,  every co-uniformly continuous, or co-coarsely continuous, map taking values into metrically convex spaces satisfies \eqref{eq:co-condition} for some $C > 0$ and $K \ge 0$ (see \cite[Corollary 4.3]{Zhang22}).

The following corollary follows from the discussion above and Proposition \ref{prop:gen-quotient}.

\begin{coro}\label{cor:quotient}
Let $\metXd$ be a metrically convex space that is umbel $p$-convex. If a self-similar metrically convex metric space $\metYd$ is a uniform or coarse quotient of $\metX$, then $\metY$ is umbel $p$-convex.
\end{coro}

\begin{rema}
\label{rem:umbelcotypequo}
Straightforward modifications of the proofs of Proposition \ref{prop:lipsubquot} and Proposition \ref{prop:gen-quotient} give that infrasup-umbel $p$-convexity is stable under Lipschitz subquotients and by taking  uniform or coarse quotient maps from metrically convex spaces into self-similar metrically convex metric space. 
\end{rema}

Corollary \ref{cor:uniquo} below, which is an immediate consequence of the stability of umbel convexity under nonlinear quotients and Corollary \ref{cor:metchar}, was proved for the first time in \cite[Theorem 2.0.1]{DKR16} (for uniform quotients) \footnote{under a separability assumption which was later lifted in \cite{DKLR17}.} using the delicate ``fork argument'' and in \cite{Zhang22} (for uniform or coarse quotients) using a more elementary self-improvement argument.

\begin{coro}
\label{cor:uniquo}
Let $\banX$ be a Banach space that has an equivalent norm with property $(\beta)$. If a Banach space $\banY$ is a uniform or coarse quotient of $\banX$, then $\banY$ has an equivalent norm with property $(\beta)$.
\end{coro}

Equipped with the stability under nonlinear quotients of umbel convexity and infrasup-umbel convexity, and the fact that countably branching trees are neither umbel $p$-convex for any $p$ nor have non-trivial infrasup-umbel convexity, we are now in position to prove, via a metric invariant approach, generalized versions of a number of known results pertaining to the nonlinear geometry of Banach spaces with property $(\beta)$ (e.g. \cite[Theorem 4.1, Theorem 4.2, Theorem 4.3]{LimaLova12}, \cite[Corollary 4.3, Corollary 4.5, Corollary 5.2, Corollary 5.3]{DKLR14}, \cite[Theorem 3.0.2]{DKR16}, and \cite[Theorem 2.1, Theorem 4.6, Theorem 4.7]{BaudierZhang16}).  
We will just give one example here illustrating the flexibility of the metric invariant approach.

Corollary 4.5 in \cite{DKLR14} states that the space $(\sum_{i=1}^\infty\ell_{p_i})_{\ell_2}$, where $\{p_i\}_{i\ge 1}$ is a decreasing sequence such that $\lim_{i\to\infty}p_i=1$, is not a uniform quotient of a Banach space that admits an equivalent norm with property $(\beta)$. The original proof uses a combination of substantial results from the nonlinear geometry of Banach spaces which are interesting in their own rights:
\begin{itemize}
    \item Ribe's result that $(\sum_{i=1}^\infty\ell_{p_i})_{\ell_2}$ is uniformly homeomorphic to $\ell_1\oplus(\sum_{i=1}^\infty\ell_{p_i})_{\ell_2}$,
    \item the fact that $\co$ is a linear quotient of $\ell_1\oplus(\sum_{i=1}^\infty\ell_{p_i})_{\ell_2}$,
    \item a quantitative comparison of the $(\beta)$-modulus with the modulus of asymptotic uniform smoothness under uniform quotients (or the qualitative Lima-Randrianarivony theorem \cite{LimaLova12} which states that $\co$ is not a uniform quotient of a Banach space that admits an equivalent norm with property $(\beta)$).
\end{itemize}
Alternatively, using the main result of \cite{DKR16}, one could argue that the assumption implies that $(\sum_{i=1}^\infty\ell_{p_i})_{\ell_2}$ admits an equivalent norm with property $(\beta)$, hence an equivalent norm that is asymptotically uniformly smooth with power type $p$ for some $p > 1 = \lim_{i\to\infty} p_i$, and derive a contradiction using linear arguments pertaining to upper and lower tree estimates, which can be found in \cite{KOS99} or \cite{OdellSchlumprecht_RACSAM} for instance.

The metric invariant approach helps streamline and extend the argument as follows. It is easy to verify that the map $\tree^\omega_k\ni \nbar\mapsto \sum_{i=1}^l e_{(n_1,\dots,n_i)}$ where $l$ is the length of $\nbar$ and $\{e_{\nbar}\}_{\nbar \in \tree^\omega_k}$ is the canonical basis of $\ell_{p_i}$ is a bi-Lipschitz embedding of $\tree_k^\omega$ into $\ell_{p_i}$ with distortion at most $2$, say, if $p_i$ is chosen small enough. Therefore, $(\sum_{i=1}^\infty\ell_{p_i})_{\ell_2}$ does not have non-trivial infrasup-umbel convexity by Proposition \ref{prop:tree-dist}, and $(\sum_{i=1}^\infty\ell_{p_i})_{\ell_2}$ is not a uniform quotient of a metrically convex metric space with non-trivial infrasup-umbel convexity by Remark \ref{rem:umbelcotypequo}.

\section{More examples of metric spaces with non-trivial infrasup-umbel convexity}
\label{sec:examples}
In this section, we give more examples of metric spaces which are umbel convex or have non-trivial infrasup-umbel convexity. We begin with the simple observation that umbel convexity is trivial for proper\footnote{A metric space is \emph{proper} if all its closed balls are compact.} metric spaces in the same way that property $(\beta)$ is trivial for finite-dimensional normed spaces.

\begin{exam}
A proper metric space $\metXd$ satisfies the $p$-umbel inequality \eqref{eq:pumbelmetric} for every $p\in[1,\infty)$ and every $K>0$, i.e.,
for all $w,z\in \met X$ and $\{x_i\}_{i \in \bN}\subseteq \met X$ we have 
\begin{equation*}(9)\hskip .5cm
\frac{1}{2^p}\inf_{i \in \bN} \dX(w,x_i)^p+\frac{1}{K^p}\inf_{i\in\bN}\liminf_{j \to \infty} \dX(x_i,x_j)^p\le \frac{1}{2}\dX(z,w)^p + \frac{1}{2}\sup_{i \in \bN} \dX(z,x_i)^p.
\end{equation*}
Consequently, $\Pi^u_p(\metX) = 0$.
\end{exam}

\begin{proof}
Obviously, if the right hand side of \eqref{eq:pumbelmetric} is infinite, there is nothing to prove, so assume it is finite. This implies that $\{x_i\}_{i \in \bN}$ is contained in a bounded, and hence a compact, set. Then $\{x_i\}_{i \in \bN}$ has a convergent subsequence, from which it follows that
$$\frac{1}{K^p}\inf_{i\in\bN}\liminf_{j \to \infty} \dX(x_i,x_j)^p = 0.$$
This fact together with a convexity argument easily imply \eqref{eq:pumbelmetric}.
\end{proof}

It is not difficult to see that the $p$-fork inequality \eqref{eq:qfork}, or in fact a natural relaxation of it, implies the $p$-umbel inequality \eqref{eq:pumbelmetric}. Thus by the implication of Theorem \ref{thm:D} that was proved in \cite{AustinNaor}, it follows that the class of metric spaces that are umbel $2$-convex contains all non-negatively curved spaces. This observation is reminiscent of the fact that local properties of Banach spaces imply their asymptotic counterparts. To our knowledge, all known examples of metric spaces that are Markov $p$-convex satisfy the $p$-fork inequality and thus they are all umbel $p$-convex. However, it seems unclear whether Markov $p$-convexity implies umbel $p$-convexity; the converse is obviously false. An interesting class of examples comes from Banach Lie groups. Let $\banX$ be a Banach space and $\omega_\banX$ an antisymmetric, bounded bilinear form on $\banX$. The \emph{Heisenberg group over $\omega_\banX$}, denoted $\bH(\omega_\banX)$, is the set $\banX \times \bR$ equipped with the product 
\begin{equation*}
\bold{x}*\bold{y}=(x,s)*(y,t) \eqd (x+y,s+t+\omega(x,y)).
\end{equation*}
In the sequel we will sometimes abbreviate $\bold{x}*\bold{y}$ by $\bold{x}\bold{y}$.
If $\omega_\banX \equiv 0$, then $\bH(\omega_\banX)$ is simply the abelian direct sum $\banX \oplus \bR$, but otherwise $\bH(\omega_\banX)$ is nonabelian. The identity element is $\bold{0} = (0,0)$, and inverses are given by $(x,s)^{-1} = (-x,-s)$. We always equip a Heisenberg group with the product topology on $\banX \times \bR$, and under this topology it becomes a topological group.

There is a natural automorphic action of $(0,\infty)$ on $\bH(\omega_\banX)$ given by $t \mapsto \delta_t$ where $$\delta_t((x,s)) \eqd (tx,t^2s).$$ The maps $\{\delta_t\}_{t>0}$ are called \emph{dilations}. A function $\sd: \bH(\omega_\banX) \times \bH(\omega_\banX) \to [0,\infty)$ with the topological compatibility property $\sd(\bold{x}_n,\bold{0}) \to_{n\to \infty} 0 \Leftrightarrow \bold{x}_n \to_{n\to\infty} \bold{0}$ is called
\begin{itemize}
    \item \emph{left-invariant} if $\sd(\bold{g*x},\bold{g*y}) = \sd(\bold{x},\bold{y})$ for all $\bold{g},\bold{x},\bold{y} \in \bH(\omega_\banX)$ and
    \item \emph{homogeneous} if $\sd(\delta_t(\bold{x}),\delta_t(\bold{y})) = t \cdot \sd(\bold{x},\bold{y})$ for all $\bold{x},\bold{y} \in \bH(\omega_\banX)$ and $t > 0$.
\end{itemize}
If $\sd^1$ and $\sd^2$ are two left-invariant, homogeneous functions, then the formal identity from $(\bH(\omega_\banX),\sd^1)$ onto $(\bH(\omega_\banX),\sd^2)$ is a bi-Lipschitz equivalence. Indeed, by symmetry it suffices to show that the map is Lipschitz. By left-invariance this reduces to $\sd^1(\bold{x},\bold{0}) \lesssim \sd^2(\bold{x},\bold{0})$, and by homogeneity this further reduces to the existence of a constant $c > 0$ such that $\sd^1(\bold{x},\bold{0}) \leq 1$ whenever $\sd^2(\bold{x},\bold{0}) \leq c$. This claim is true by the topological compatibilities  of $\sd^1,\sd^2$.

When $\omega_\banX \not\equiv 0$, there is a canonical left-invariant, homogeneous metric on $\bH(\omega_\banX)$ called the \emph{Carnot-Carath\'eodory metric}, denoted $\sd_{cc}$. A pair $(\gamma,z)$ of Lipschitz curves $\gamma: [0,1] \to \banX$, $z: [0,1] \to \bR$ is called a \emph{horizontal curve} if $\gamma$ is differentiable almost everywhere and $z'(t) = \omega(\gamma(t),\gamma'(t))$ for almost every $t \in [0,1]$. The \emph{horizontal length} of a horizontal curve $(\gamma,z)$ is defined to be the length of $\gamma$. Then the Carnot-Carath\'eodory distance between $\bold{x}$ and $\bold{y}$ is defined to be the infimum of horizontal lengths of horizontal curves joining $\bold{x}$ and $\bold{y}$. It is exactly the assumption $\omega_\banX \not\equiv 0$ that ensures that any two points in $\bH(\omega_\banX)$ can be joined by a horizontal curve. Obviously, $\sd_{cc}$ satisfies the triangle inequality and is a length metric. Any left-invariant, homogeneous, symmetric function on $\bH(\omega_\banX)$ is a quasi-metric since it is bi-Lipschitz equivalent to the metric $\sd_{cc}$.
A particularly handy way to obtain such functions is via Koranyi-type norms. For $p\in[1,\infty]$ and a given $\lambda > 0$, define a function $N_{p,\lambda} \colon \bH(\omega_\banX) \to [0,\infty)$ by 
\begin{equation*}
N_{p,\lambda}((x,s)) \eqd \begin{cases}
					(\norm{x}_\banX^{2p} + \lambda^{2p} |s|^p)^{\frac{1}{2p}}, \text{ if } p\in[1,\infty)\\
					\max\{\norm{x}_\banX, \lambda \sqrt{|s|}\} \text{ if } p=\infty.
					\end{cases}
\end{equation*}
Then we define the function $\sd_{p,\lambda}(\bold{x},\bold{y}) \eqd N_{p,\lambda}(\bold{y}^{-1}*\bold{x})$. Clearly, $\sd_{p,\lambda}$ is a symmetric, left-invariant, homogeneous function, and hence is a quasi-metric equivalent to $\sd_{cc}$.

Banach Lie groups constructed this way have been investigated in \cite{MagnaniRajala14} under the name \emph{Banach homogeneous groups}. When $\banX = \bR^n \oplus \bR^n$ and 
$$
\omega_\banX((x_1,x_2),(y_1,y_2)) = \frac{1}{2}\langle x_1,y_2 \rangle - \frac{1}{2}\langle x_2,y_1 \rangle,
$$
$\bH(\omega_\banX)$ is a (finite-dimensional) Lie group called the \emph{$n$th Heisenberg group}, and simply denoted $\bH(\bR^n)$. The space $\bH(\bR^n)$ is very well-studied by metric space geometers, see \cite{CDPT07} for an introduction. We will denote by $\bH(\ell_2)$ the infinite-dimensional Heisenberg group $\bH(\omega_{\ell_2})$ where $\omega_{\ell_2}((x_1,y_1),(x_2,y_2)) = \frac12 \langle x_1, y_2 \rangle - \frac12 \langle y_1 , x_2 \rangle$. Note that $\omega_{\ell_2}\not\equiv 0$. It was shown by Li in \cite{Li16} that the set of $p$'s for which $\bH(\ell_2)$ is Markov $p$-convex is exactly $[4,\infty)$. We believe that Li's proof can be adjusted to show that $\bH(\omega_\banX)$ is Markov $2p$-convex whenever $\banX$ is $p$-uniformly convex. In Section \ref{sec:Heisenberg-parallelogram}, we will provide a more direct argument - based on that found in \cite[Section 4.2]{GartlandCarnot} - to prove this result (cf. Theorem \ref{thm:Sean}).


The next theorem shows that a Heisenberg group over a Banach space with property $(\beta_p)$ is infrasup-umbel $p$-convex. These examples are interesting since these Heisenberg groups do not admit bi-Lipschitz embeddings into any Banach space with an equivalent norm with property $(\beta)$, and thus are genuine metric examples. Indeed, an infinite-dimensional Heisenberg group contains $\bH(\bR)$ bi-Lipschitzly, and it was crucially observed by Semmes \cite{Semmes96} that $\bH(\bR)$ does not embed bi-Lipschitzly into any Banach space with the Radon-Nikod\'ym property (in particular a reflexive one) since Pansu's differentiability theorem \cite{Pansu89} extends to RNP-target spaces (cf. \cite{LeeNaor06} and \cite{CheegerKleiner06} for more details).

\begin{theo} \label{thm:Heisenberg}
Let $p\in [2,\infty)$ and $\omega_\banX$ be any bounded antisymmetric bilinear form on a Banach space $\banX$ that satisfies the relaxation of the $p$-umbel inequality \eqref{eq:relaxed-p-umbel} with constant $C$. 
Then $(\bH(\omega_\banX),\sd_{\infty,1})$ satisfies the relaxation of the $p$-umbel inequality \eqref{eq:relaxed-p-umbel} with constant $\max\{C,2 \cdot 8^{1/p}\}$. 

Consequently, for any $p \in [2,\infty)$ and any non-zero, antisymmetric, bounded bilinear form $\omega_\banX$ on a Banach space $\banX$ with property $(\beta_p)$, $(\bH(\omega_\banX),\sd_{cc})$ is infrasup-umbel $p$-convex.
\end{theo}

\begin{proof}
Assume that we have shown that the quasi-metric $\sd_{\infty,1}$ satisfies the relaxation of the $p$-umbel inequality \eqref{eq:relaxed-p-umbel}, then by Remark \ref{rem:relaxed-p-umbel} and that fact $\sd_{\infty,1}$ is equivalent to $\sd_{cc}$, it will follow that $(\bH(\omega_\banX),\sd_{cc})$ satisfies inequality \eqref{eq:relaxed-umbel-p} and hence is infrasup-umbel $p$-convex.

Assume that $\banX$ satisfies the relaxation of the $p$-umbel inequality \eqref{eq:relaxed-p-umbel} with constant $C$. Set $K \eqd \max\{C,2 \cdot 8^{1/p}\}$, and simply write $\sd=\sd_{\infty,1}$ and $N=N_{\infty,1}$ in this proof. By left-invariance, we may assume $\bold{z} = \bold{0}$. There is nothing to prove if the right hand side of \eqref{eq:relaxed-p-umbel} is infinite, so assume it is finite. This implies $\{\bold{x}_i\}_{i \in \bN} = \{(x_i,s_i)\}_{i \in \bN}$ is a bounded subset of $\bH(\omega_\banX)$, and hence $\{x_i\}_{i \in \bN}$ and $\{s_i\}_{i \in \bN}$ are bounded subsets of $\banX$ and $\bR$, respectively. By Proposition \ref{prop:reflexivity} and Remark \ref{rem:relaxed-p-umbel}, $\banX$ is reflexive, and there is $\bM \in [\bN]^{\omega}$ such that weak-$\lim_{i \in \bM} x_i = x$ and $\lim_{i \in \bM} s_i = s$ for some $x \in \banX$ and $s \in \bR$. Then (denoting $\bold{w}=(w,t))$
\begin{align*}
\inf_{i \in \bN}\frac{\sd(\bold{x}_i,\bold{w})^{p}}{2^{p}} & + \inf_{i \in \bN} \liminf_{j \in \bN} \frac{\sd(\bold{x}_i,\bold{x}_j)^{p}}{K^p} = \inf_{i \in \bN} \frac{N(\bold{x}_i^{-1}\bold{w})^{p}}{2^p} + \inf_{i \in \bN} \liminf_{j \in \bN} \frac{N(\bold{x}_i^{-1}\bold{x}_j)^{p}}{K^p} \\
= & \inf_{i \in \bN}\max\Big\{\frac{\norm{w-x_i}_\banX^{p}}{2^p}, \frac{|t-s_i-\omega(w,x_i)|^{\frac{p}{2}}}{2^p}\Big\}\\\
 & \qquad + \inf_{i \in \bN} \liminf_{j \in \bN} \max\Big\{\frac{\norm{x_i-x_j}_\banX^{p}}{K^p}, \frac{|s_i-s_j-\omega(x_i,x_j)|^{\frac{p}{2}}}{K^p}\Big\} \\
  \leq &  \liminf_{i \in \bM}\max\Big\{\frac{\norm{w-x_i}_\banX^{p}}{2^p},   \frac{|t-s_i-\omega(w,x_i)|^{\frac{p}{2}}}{2^p}\Big\} \\
   & + \liminf_{i \in \bM} \liminf_{j \in \bM} \max\Big\{\frac{\norm{x_i-x_j}_\banX^{p}}{K^p},   \frac{|s_i-s_j-\omega(x_i,x_j)|^{\frac{p}{2}}}{K^p}\Big\} \\
= & \max\Big\{\liminf_{i \in \bM}\frac{\norm{w-x_i}_\banX^{p}}{2^p},  \frac{|t-s-\omega(w,x)|^{\frac{p}{2}}}{2^p}\Big\} \\
 & + \max\Big\{\liminf_{i \in \bM} \liminf_{j \in \bM}\frac{\norm{x_i-x_j}_\banX^{p}}{K^p},   \frac{|s-s-\omega(x,x)|^{\frac{p}{2}}}{K^p}\Big\}.
 \end{align*}
Since $\omega_\banX$ is antisymmetric, $\omega_\banX(x,x)=0$ and hence 
 \begin{align*}
 \inf_{i \in \bN} \frac{\sd(\bold{x}_i,\bold{w})^{p}}{2^{p}} & + \inf_{i \in \bN} \liminf_{j \in \bN} \frac{\sd(\bold{x}_i,\bold{x}_j)^{p}}{(2K)^p} \\
\le & \underbrace{\max\Big\{\liminf_{i \in \bM}\frac{\norm{w-x_i}_\banX^{p}}{2^p},   \frac{|t-s-\omega(w,x)|^{\frac{p}{2}}}{2^p}\Big\} + \liminf_{i \in \bM} \liminf_{j \in \bM}\frac{\norm{x_i-x_j}_\banX^{p}}{K^p}}_{(*)}.
\end{align*}
Assume the first term in the maximum is larger. Then
\begin{align*}
(*) &= \liminf_{i \in \bM}\frac{\norm{w-x_i}_\banX^{p}}{2^p} + \liminf_{i \in \bM} \liminf_{j \in \bM} \frac{\norm{x_i-x_j}_\banX^{p}}{K^p} \overset{\eqref{eq:relaxed-p-umbel}}{\leq} \max\Big\{\norm{w}_\banX^{p},\sup_{i \in \bM}\norm{x_i}_\banX^{p}\Big\} \\
 &  \le \max\Big\{N(\bold{w})^{p},\sup_{i \in \bN} N(\bold{x}_i)^{p}\Big\} = \max\Big\{\sd(\bold{0},\bold{w})^p,\sup_{i \in \bN} \sd(\bold{0},\bold{x}_i)^p\Big\}
\end{align*}
so \eqref{eq:relaxed-p-umbel} holds in this case. Now assume the second term in the maximum is larger. Then
\begin{align*}
(*) &= \frac{|t-s-\omega(w,x)|^{\frac{p}{2}}}{2^p} + \liminf_{i \in \bM} \liminf_{j \in \bM} \frac{\norm{x_i-x_j}_\banX^{p}}{K^p} \\
&\leq 3^{\frac{p}{2}-1}\frac{|t|^{\frac{p}{2}}+|s|^{\frac{p}{2}}+|\omega(w,x)|^{\frac{p}{2}}}{2^p} + \frac{2^p}{K^{p}}\sup_{i \in \bN}\norm{x_i}_\banX^{p} \\
&\leq \frac{1}{4}(|t|^{\frac{p}{2}}+|s|^{\frac{p}{2}}+\norm{w}_\banX^{\frac{p}{2}}\norm{x}_\banX^{\frac{p}{2}}) +\frac{2^p}{K^{p}}\sup_{i \in \bN}\norm{x_i}_\banX^{p} \\
&\leq \frac{1}{4}(|t|^{\frac{p}{2}}+|s|^{\frac{p}{2}}+\frac{1}{2}(\norm{w}_\banX^p+\norm{x}_\banX^p)) + \frac{2^p}{K^{p}}\sup_{i \in \bN}\norm{x_i}_\banX^{p}\\
&= \frac{1}{4}|t|^{\frac{p}{2}}+\frac{1}{4}|s|^{\frac{p}{2}}+\frac{1}{8}\norm{w}_\banX^p + \Big(\frac{1}{8}+ \frac{2^p}{K^{p}}\Big)\sup_{i \in \bN}\norm{x_i}_\banX^{p} \\
&\leq \frac{1}{4}|t|^{\frac{p}{2}}+\frac{1}{4}|s|^{\frac{p}{2}}+\frac{1}{4}\norm{w}_\banX^p + \frac14 \sup_{i \in \bN}\norm{x_i}_\banX^{p}\\
&\leq \max\left\{|t|^{\frac{p}{2}},|s|^{\frac{p}{2}},\norm{w}_\banX^p,\sup_{i \in \bN}\norm{x_i}_\banX^{p}\right\}\\
&= \max\left\{N(\bold{w})^{p},\sup_{i \in \bN} N(\bold{x}_i)^{p}\right\} = \max\left\{\sd(\bold{0},\bold{w})^p,\sup_{i \in \bN} \sd(\bold{0},\bold{x}_i)^p\right\}.
\end{align*}
\end{proof}

In general, the value of $p$ for which $\bH(\omega_\banX)$ is infrasup-umbel $p$-convex cannot be taken smaller. When $p \in [2,\infty)$, $\banX = \ell_p \oplus_p (\ell_p)^*$ has property $(\beta_p)$ and take $\omega_{\banX}((x_1,y^*_1),(x_2,y^*_2)) = y^*_2(x_1) - y^*_1(x_2)$ (which is obviously nonzero). The map from $\ell_p$ to 
$\bH(\omega_\banX)$ defined by $x \mapsto ((x,0),0)$ is an isometric embedding, but Corollary \ref{cor:betap->lq-asm} implies that $\ell_p$ is not infrasup-umbel $q$-convex for any $q < p$.

Finally, we explain how we can construct more spaces that are umbel $p$-convex by taking finite $\ell_p$-sums of spaces satisfying the $p$-umbel inequality. The next lemma, which is a simple consequence of Ramsey's theorem, will be crucial to achieve this goal.

\begin{lemm} 
\label{lem:metpforkRamsey}
Every metric space $\metXd$ satisfying the $p$-umbel inequality \eqref{eq:pumbelmetric} satisfies the following formally stronger property: 

\noindent For any $w,z,x_i \in \metX$ with $\sup_{i \in \bN} \dX(z,x_i) < \infty$ and $\vep>0$, there exists an infinite subset $\bM$ of $\bN$ such that
\begin{equation}
\label{eq:strongpumbel}
\frac{1}{2^p}\sup_{i \in \bM} \dX(w,x_i)^p+\frac{1}{K^p}\sup_{i \neq j \in \bM} \dX(x_i,x_j)^p\le \frac{1}{2}\dX(z,w)^p + \frac{1}{2}\inf_{i \in \bM} \dX(z,x_i)^p + \vep.
\end{equation}
\end{lemm}

\begin{proof}
Choose $N \in \bN$ large enough so that $\frac{3}{N}<\vep$ and let
\begin{equation*}
B \eqd \max\left\{\frac{1}{2^p}\sup_{i \in \bN} \dX(w,x_i)^p , \frac{1}{K^p}\sup_{i \neq j \in \bN} \dX(x_i,x_j)^p , \frac{1}{2}\sup_{i \in \bN} \dX(z,x_i)^p\right\}.
\end{equation*}
Consider the finite cover $[0,B] \subset \bigcup_{k=1}^{\lceil NB \rceil} [\frac{k-1}{N},\frac{k}{N}]$. Since
$$\frac{1}{2^p} \dX(w,x_i)^p, \frac{1}{K^p} \dX(x_i,x_j)^p, \frac{1}{2}\dX(z,x_i)^p \in [0,B]$$
for every $i \neq j \in \bN$, the pigeonhole principle and Ramsey's theorem gives us an infinite subset $\bM \subset \bN$ and natural numbers $k_1, k_2, k_3 \leq \lceil NB \rceil$ such that, for every $i \neq j \in \bM$,
\begin{align*}
\frac{1}{2^p} \dX(w,x_i)^p \in \left[\frac{k_1-1}{N},\frac{k_1}{N}\right], \qquad \frac{1}{2}\dX(z,x_i)^p \in \left[\frac{k_2-1}{N},\frac{k_2}{N}\right],
\end{align*}
and 
\begin{equation*}
\frac{1}{K^p} \dX(x_i,x_j)^p \in \left[\frac{k_3-1}{N},\frac{k_3}{N}\right].
\end{equation*}
Therefore,
\begin{align*}
\frac{1}{2^p} (\sup_{i \in \bM} \dX(w,x_i)^p - \inf_{i \in \bM} \dX(w,x_i)^p) \leq \frac{1}{N}, \qquad \frac{1}{2} (\sup_{i \in \bM} \dX(z,x_i)^p - \inf_{i \in \bM} \dX(z,x_i)^p) \leq \frac{1}{N},
\end{align*}
and
\begin{equation*}
\frac{1}{K^p} (\sup_{i \neq j \in \bM} \dX(x_i,x_j)^p - \inf_{i \neq j \in \bM} \dX(x_i,x_j)^p) \leq \frac{1}{N}.
\end{equation*}
Then we apply \eqref{eq:pumbelmetric} to $w,z,\{x_i\}_{i \in \bM}$ together with the inequalities above and get
\begin{align*}
\frac{1}{2^p}\sup_{i \in \bM} \dX(w,x_i)^p+\frac{1}{K^p}\sup_{i \neq j \in \bM} \dX(x_i,x_j)^p &\leq  \frac{1}{2^p}\inf_{i \in \bM} \dX(w,x_i)^p + \frac{1}{N} + \frac{1}{K^p} \inf_{i \neq j \in \bM}\dX(x_i,x_j)^p + \frac{1}{N} \\
&\overset{\eqref{eq:pumbelmetric}}{\leq} \frac{2}{N} + \frac{1}{2}\dX(z,w)^p + \frac{1}{2}\sup_{i \in \bM} \dX(z,x_i)^p \\
&\leq \frac{3}{N} + \frac{1}{2}\dX(z,w)^p + \frac{1}{2}\inf_{i \in \bM} \dX(z,x_i)^p \\
&\leq \vep + \frac{1}{2}\dX(z,w)^p + \frac{1}{2}\inf_{i \in \bM} \dX(z,x_i)^p.
\end{align*}
\end{proof}

A consequence of the theorem below, whose proof requires Ramsey's theorem via Lemma \ref{lem:metpforkRamsey}, is that a finite $\ell_p$-sum $(\sum_{i=1}^j \metX_i)_{\ell_p}$ is umbel $p$-convex whenever $\{\metX_i\}_{i=1}^j$ are metric spaces satisfying the $p$-umbel inequality for some universal constant $K > 0$. It is worth pointing out that an arbitrary $\ell_p$-sum of metric spaces which are Markov $p$-convex (with some universal Markov convexity constant) is Markov $p$-convex.

\begin{theo}
Let $p \in [1,\infty)$ and let $\metXd,\metYd$ be metric spaces satisfying the $p$-umbel inequality \eqref{eq:pumbelmetric} for some constant $K > 0$. Then $\metX \oplus_p \metY$ satisfies the $p$-umbel inequality \eqref{eq:pumbelmetric} with constant $K$.
\end{theo}

\begin{proof}
Let $(w^1,w^2),(z^1,z^2)\in \metX \oplus_p \metY$ and $\{(x^1_i,x^2_i)\}_{i \in \bN}\subseteq \metX \oplus_p \metY$. If the right hand side of \eqref{eq:pumbelmetric} is infinite, there is nothing to prove, so assume it is finite. Let $\vep > 0$ be arbitrary. Then by Lemma \ref{lem:metpforkRamsey}, we can find $\bM \in [\bN]^\omega$ such that
\begin{equation*} 
\frac{1}{2^p}\sup_{i \in \bM} \dX(w^1,x^1_i)^p+\frac{1}{K^p}\sup_{i \neq j \in \bM} \dX(x^1_i,x^1_j)^p\le \frac{1}{2}\dX(z^1,w^1)^p + \frac{1}{2}\inf_{i \in \bM} \dX(z^1,x^1_i)^p+\vep
\end{equation*}
and
\begin{equation*} 
\frac{1}{2^p}\sup_{i \in \bM} \dY(w^2,x^2_i)^p+\frac{1}{K^p}\sup_{i \neq j \in \bM} \dY(x^2_i,x^2_j)^p \le \frac{1}{2}\dY(z^2,w^2)^p + \frac{1}{2}\inf_{i \in \bM} \dY(z^2,x^2_i)^p + \vep.
\end{equation*}
Adding these two equations yields
\begin{align} \nonumber
\frac{1}{2^p}\left(\sup_{i \in \bM} \dX(w^1,x^1_i)^p+\sup_{i \in \bM} \dY(w^2,x^2_i)^p\right)+\frac{1}{K^p}\left(\sup_{i \neq j \in \bM} \dX(x^1_i,x^1_j)^p + \sup_{i \neq j \in \bM} \dY(x^2_i,x^2_j)^p\right) \\ \label{eq:auxsum}
\le \frac{1}{2}(\dX(z^1,w^1)^p+\dY(z^2,w^2)^p) + \frac{1}{2}\left(\inf_{i \in \bM} \dX(z^1,x^1_i)^p+\inf_{i \in \bM} \dY(z^2,x^2_i)^p\right) + 2\vep.
\end{align}
Then using the definition of the metric $\dX \oplus_p \dY$, we get
\begin{align*}
\frac{1}{2^p} & \inf_{i \in \bN} \dX \oplus_p \dY((w^1,w^2),(x^1_i,x^2_i))^p+\frac{1}{K^p}\inf_{i \in \bN}\liminf_{j \to \infty} \dX \oplus_p \dY((x^1_i,x^2_i),(x^1_j,x^2_j))^p \\
\le  & \frac{1}{2^p}\sup_{i \in \bM} (\dX(w^1,x^1_i)^p+ \dY(w^2,x^2_i)^p)+\frac{1}{K^p}\sup_{i \neq j \in \bM} (\dX(x^1_i,x^1_j)^p + \dY(x^2_i,x^2_j)^p) \\
\leq & \frac{1}{2^p}\left(\sup_{i \in \bM} \dX(w^1,x^1_i)^p+\sup_{i \in \bM} \dY(w^2,x^2_i)^p\right)+\frac{1}{K^p}\left(\sup_{i \neq j \in \bM} \dX(x^1_i,x^1_j)^p + \sup_{i \neq j \in \bM} \dY(x^2_i,x^2_j)^p\right) \\
\overset{\eqref{eq:auxsum}}{\leq} & \frac{1}{2}(\dX(z^1,w^1)^p+\dY(z^2,w^2)^p) + \frac{1}{2}\left(\inf_{i \in \bM} \dX(z^1,x^1_i)^p+\inf_{i \in \bM} \dY(z^2,x^2_i)^p\right) + 2\vep \\
\leq & \frac{1}{2}(\dX(z^1,w^1)^p+\dY(z^2,w^2)^p) + \frac{1}{2}\inf_{i \in \bM} (\dX(z^1,x^1_i)^p+\dY(z^2,x^2_i)^p) + 2\vep\\
\leq & \frac{1}{2}\dX \oplus_p \dY((z^1,x^2),(w^1,w^2))^p+ \frac{1}{2}\sup_{i \in \bN} \dX \oplus_p \dY((z^1,z^2),(x^2_i,x^1_i))^p + 2\vep.
\end{align*}
Since $\vep>0$ was arbitrary, inequality \eqref{eq:pumbelmetric} follows.
\end{proof}

\section{Markov and diamond convexity of Heisenberg groups}
\label{sec:Heisenberg-parallelogram}

In this section we fulfill our promise from Section \ref{sec:examples} and show that Heisenberg groups over $p$-uniformly convex Banach spaces are Markov $2p$-convex. This fact will follow from a ``parallelogram convexity inequality" analogous to the following parallelogram inequality holding in a Banach space $\banX$ that is $p$-uniformly convex with constant $K$: for all $x,y\in \banX$
	\begin{equation}\label{eq:UCmod}
		\frac{\norm{x}_\banX^p+\norm{x-y}_\banX^p}{2} \ge \bnorm{\frac{y}{2}}^p_\banX + \frac{1}{K^p}\bnorm{x-\frac{y}{2}}^p_\banX.
	\end{equation}
Inequality \eqref{eq:UCmod} can be derived easily from inequality \eqref{eq:p-convex} (and vice versa).
	
\begin{prop}\label{prop:parallelogram}
Let $p \in [2,\infty)$ and $\omega_\banX$ be any non-zero, antisymmetric, bounded bilinear form on a $p$-uniformly convex Banach space $\banX$. Then, there is a constant $C:=C(\banX,\omega_\banX)>0$ and a Koranyi-type norm $N_{p,\lambda}$ for some $\lambda:=\lambda(\banX,\omega_\banX)>0$ such that for every $\bold{a} = (a,s), \bold{b} = (b,t) \in \bH(\omega_\banX)$,
\begin{equation} \label{eq:parallelogram}
\frac{1}{2}N_{p,\lambda}(\bold{a})^{2p} + \frac{1}{2}N_{p,\lambda}(\bold{b}^{-1}\bold{a})^{2p} \geq  N_{p,\lambda}(\delta_{1/2}(\bold{b}))^{2p} + \frac{1}{C^{2p}} N_{p,\lambda}(\delta_{1/2}(\bold{b})^{-1}\bold{a})^{2p}.
\end{equation}
\end{prop}

\begin{proof}
Assume that $\banX$ is $p$-uniformly convex with constant $K$. Let $\omega \eqd \omega_\banX$ and $N\eqd N_{p,\lambda}$, where $\lambda^{2p} \eqd \left(\frac{1}{3}+\frac{1}{3^p \cdot 6}\right)^{-1}\frac{2^{1-p}}{K^p\|\omega\|^p}$ and $\|\omega\| < \infty$ is the least constant $B$ satisfying $|\omega(a,b)| \leq B\|a\|_\banX\|b\|_\banX$.
We have
\begin{align}
\label{eq:auxSean1}
\nonumber \frac{1}{2} N(\bold{a})^{2p} + \frac{1}{2}N(\bold{b}^{-1}\bold{a})^{2p}  =  \frac{1}{2}\left\|a\right\|_\banX^{2p} + \frac{\lambda^{2p} }{2}|s|^p + &\frac{1}{2}\left\|a-b\right\|_\banX^{2p} + \frac{\lambda^{2p} }{2}|s-t+\omega(a,b)|^p \\
\nonumber \geq  \left(\frac{1}{2}\left\|a\right\|_\banX^p +  \frac{1}{2}\left\|a-b\right\|_\banX^p\right)^2 + & \frac{\lambda^{2p} }{2}|s|^p + \frac{\lambda^{2p} }{2}|s-t+\omega(a,b)|^p \quad (\text{convexity})\\
\stackrel{\eqref{eq:UCmod}}{\geq}  \underbrace{\left(\left\|\frac{b}{2}\right\|_\banX^p +  \frac{1}{K^p}\left\|a-\frac{b}{2}\right\|_\banX^p\right)^2}_{\alpha} + & \frac{\lambda^{2p} }{2}|s|^p + \frac{\lambda^{2p} }{2}|s-t+\omega(a,b)|^p.
\end{align}
Since $\omega$ is antisymmetric and bounded,
\begin{align}
\label{eq:auxSean2}
\nonumber 
\alpha & =  \left\|\frac{b}{2}\right\|_\banX^{2p} +  \frac{1}{K^{2p}}\left\|a-\frac{b}{2}\right\|_\banX^{2p} +  \frac{2}{K^p}\left(\left\|\frac{b}{2}\right\|_\banX\left\|a-\frac{b}{2}\right\|_\banX\right)^p\\	
\nonumber	& \geq \left\|\frac{b}{2}\right\|_\banX^{2p} +  \frac{1}{K^{2p}}\left\|a-\frac{b}{2}\right\|_\banX^{2p} + \frac{2}{K^p\|\omega\|^p}\left|\omega\left(\frac{b}{2},a-\frac{b}{2}\right)\right|^p\\
\nonumber	& = \left\|\frac{b}{2}\right\|_\banX^{2p} + \frac{1}{K^{2p}}\left\|a-\frac{b}{2}\right\|_\banX^{2p} + \frac{2^{1-p}}{K^p\|\omega\|^p}|\omega(a,b)|^p \\
    & = \left\|\frac{b}{2}\right\|_\banX^{2p} + \frac{1}{K^{2p}}\left\|a-\frac{b}{2}\right\|_\banX^{2p} + \lambda^{2p} \big(\frac{1}{3}+\frac{1}{3^p \cdot 6}\Big)|\omega(a,b)|^p
\end{align}
where we have used the definition of $\lambda$ in the last equality. Incorporating \eqref{eq:auxSean2} into \eqref{eq:auxSean1} we thus have,
\begin{align}
\label{eq:auxSean3}
\nonumber\frac{1}{2} N(\bold{a})^{2p} & + \frac{1}{2}N(\bold{b}^{-1}\bold{a})^{2p}  \geq \left\| \frac{b}{2}\right\|_\banX^{2p} + \frac{1}{K^{2p}}\left\| a - \frac{b}{2}\right\|_\banX^{2p}\\
& + \lambda^{2p}  \underbrace{\Big(\frac{|\omega(a,b)|^p + |s|^p + |s-t+\omega(a,b)|^p}{3} + \frac{|\omega(a,b)|^p}{3^{p} \cdot 6} + \frac{|s|^p}{6} + \frac{|s-t+\omega(a,b)|^p}{6}\Big)}_{\beta}.
\end{align}
We now proceed to estimate $\beta$ as follows,
\begin{align}
\label{eq:auxSean4}
\nonumber \beta \ge & \left|\frac{t}{3}\right|^p + \frac{|\omega(a,b)|^p}{3^{p} \cdot 6} + \frac{|s|^p}{6} + \frac{|s-t+\omega(a,b)|^p}{6} \quad (\text{convexity and triangle inequality}) \\
\nonumber \ge & \left|\frac{t}{3}\right|^p + \frac{4^p}{3^p \cdot 6}\left(\left|\frac{\omega(a,b)}{4}\right|^p + \left|\frac{3s}{4}\right|^p + \left|\frac{s}{4}-\frac{t}{4}+\frac{\omega(a,b)}{4}\right|^p\right) \\
 \ge &  \left|\frac{t}{4}\right|^p + \frac{4^p}{3^{2p-1} \cdot 6}\left|s-\frac{t}{4}+\frac{\omega(a,b)}{2}\right|^p \quad (\text{convexity and triangle inequality}).
\end{align}
If we let $C \eqd \max\left\{K, \Big(\frac{3^{2p-1} \cdot 6}{4^p}\Big)^{\frac{1}{2p}}\right\}$, combining \eqref{eq:auxSean4} with \eqref{eq:auxSean3}, we have
\begin{align*}
\label{eq:auxSean5}
\frac{1}{2} N(\bold{a})^{2p}  + \frac{1}{2}N(\bold{b}^{-1}\bold{a})^{2p}  & \geq \left\| \frac{b}{2}\right\|_\banX^{2p} + \frac{1}{C^{2p}}\left\| a - \frac{b}{2}\right\|_\banX^{2p} + \lambda^{2p}   \left|\frac{t}{4}\right|^p  + \frac{\lambda^{2p} }{C^{2p}}\left|s-\frac{t}{4}+\frac{\omega(a,b)}{2}\right|^p\\
   		& \ge N(\delta_{1/2}(\bold{b}))^{2p} + \frac{1}{C^{2p}}N(\delta_{1/2}(\bold{b})^{-1}\bold{a})^{2p}.
\end{align*}
This completes the proof of \eqref{eq:parallelogram}.
\end{proof}

The following theorem follows from the fact that the $2p$-fork inequality is valid for the quasi-metric induced by a quasi-norm satisfying \eqref{eq:parallelogram}.

\begin{theo}
\label{thm:Sean}
For any $p \in [2,\infty)$ and any non-zero, antisymmetric, bounded bilinear form $\omega_\banX$ on a $p$-uniformly convex Banach space $\banX$, $(\bH(\omega_\banX),\sd_{cc})$ is Markov $2p$-convex.
\end{theo}

\begin{proof}
Let $p$, $\banX$, $\omega_\banX$ be as in the statement. Since Markov $2p$-convexity is a bi-Lipschitz invariant and the proof in \cite{MendelNaor13} showing that Markov $p$-convexity follows from the $p$-fork inequality is valid for quasi-metrics, it suffices to prove that $(\bH(\omega_\banX),\sd)$ is Markov $2p$-convex for some equivalent quasi-metric $\sd$. Because of this, we may again assume $\banX$ is equipped with a uniformly $p$-convex norm with constant $K$. Therefore, it suffices to exhibit a quasi-metric $\sd$  that satisfies the $2p$-fork inequality. In the remainder of this proof, we will let $N\eqd N_{p,\lambda}$ the Koranyi-type norm from Proposition \ref{prop:parallelogram} and $\sd \eqd \sd_{\lambda,p}$ the quasi-metric it induces. We will use \eqref{eq:parallelogram} to prove the $2p$-fork inequality:
\begin{equation} \label{eq:fork}
\frac{\sd(\bold{w},\bold{x})^{2p}}{2^{2p+1}}+\frac{\sd(\bold{w},\bold{y})^{2p}}{2^{2p+1}}+\frac{\sd(\bold{x},\bold{y})^{2p}}{(4C'C)^{2p}}\le \frac{1}{2}\sd(\bold{z},\bold{w})^{2p}+\frac{1}{4}\sd(\bold{z},\bold{x})^{2p}+\frac{1}{4}\sd(\bold{z},\bold{y})^{2p},
\end{equation}
where $C'$ is the quasi-triangle inequality constant of $\sd$.

First apply \eqref{eq:parallelogram} with $\bold{a} = \bold{z}$ and $\bold{b} = \bold{x}$ to obtain
\begin{equation*}
\frac{1}{2}N(\bold{z})^{2p} + \frac{1}{2}N(\bold{x}^{-1}\bold{z})^{2p} \geq N(\delta_{1/2}(\bold{x}))^{2p} + \frac{1}{C^{2p}}N(\delta_{1/2}(\bold{x})^{-1}\bold{z})^{2p}.
\end{equation*}
Then apply \eqref{eq:parallelogram} with $\bold{a} = \bold{z}$ and $\bold{b} = \bold{y}$ to obtain
\begin{equation*}
\frac{1}{2}N(\bold{z})^{2p} + \frac{1}{2}N(\bold{y}^{-1}\bold{z})^{2p} \geq N(\delta_{1/2}(\bold{y}))^{2p} + \frac{1}{C^{2p}}N(\delta_{1/2}(\bold{y})^{-1}\bold{z})^{2p}.
\end{equation*}
Averaging these two inequalities and using the definition and homogeneity of $\sd$ yields
\begin{align*}
\frac{\sd(\bold{0},\bold{x})^{2p}}{2^{2p+1}}+\frac{\sd(\bold{0},\bold{y})^{2p}}{2^{2p+1}} & + \frac{\sd(\bold{z},\delta_{1/2}(\bold{x}))^{2p}}{2C^{2p}} + \frac{\sd(\bold{z},\delta_{1/2}(\bold{y}))^{2p}}{2C^{2p}} \\
& \le \frac{1}{2}\sd(\bold{z},\bold{0})^{2p}+\frac{1}{4}\sd(\bold{z},\bold{x})^{2p}+\frac{1}{4}\sd(\bold{z},\bold{y})^{2p}.
\end{align*}
Then by convexity, the $C'$-quasi-triangle inequality of $\sd$, and homogeneity of $\sd$, we get
\begin{align*}
\frac{\sd(\bold{0},\bold{x})^{2p}}{2^{2p+1}}+\frac{\sd(\bold{0},\bold{y})^{2p}}{2^{2p+1}} & + \frac{\sd(\bold{x},\bold{y})^{2p}}{(4C'C)^{2p}} \le \frac{1}{2}\sd(\bold{z},\bold{0})^{2p}+\frac{1}{4}\sd(\bold{z},\bold{x})^{2p}+\frac{1}{4}\sd(\bold{z},\bold{y})^{2p}
\end{align*}
This proves \eqref{eq:fork} for $\bold{w} = \bold{0}$. The general inequality follows from left-invariance.
\end{proof}

Two new metric invariants, called diamond convexity and graphical diamond convexity, were introduced in \cite{EMN}. Diamond convexity is an inequality involving stochastic processes (like Markov convexity), and graphical diamond convexity is a deterministic Poincar\'e-type inequality that refers explicitly to diamond graphs. In \cite{EMN}, it was shown that if a metric space $\metX$ is Markov $p$-convex, then $\metX$ is diamond $p$-convex, and hence the Heisenberg groups as in Theorem \ref{thm:Sean} are diamond $2p$-convex. It is currently not known whether Markov $p$-convexity or diamond $p$-convexity implies graphical diamond $p$-convexity. However it was shown that diamond $p$-convexity (cf. \cite{EMN}) and graphical diamond $p$-convexity (cf. \cite[Chapter 2]{Eskenazis_PhD}) follow from the following \emph{$p$-short diagonals inequality for uniform convexity with constant $K\in(0,\infty)$}: for all $w,x,y,z\in \metXd$ 
	\begin{equation}\label{eq:short-q}
		\frac{1}{2^p}\dX(w,y)^p + \frac{1}{(2K)^p}\dX(x,z)^p \le \frac14\dX(w,x)^p + \frac14\dX(x,y)^p + \frac14\dX(y,z)^p + \frac14\dX(z,w)^p  
	\end{equation}
Since, as we will show, the $p$-short diagonals inequality for uniform convexity is valid for the quasi-metric induced by a quasi-norm satisfying \eqref{eq:parallelogram}, we have:

\begin{theo}
\label{thm:diamond}
For any $p \in [2,\infty)$ and any non-zero, antisymmetric, bounded bilinear form $\omega_\banX$ on a $p$-uniformly convex Banach space $\banX$, $(\bH(\omega_\banX),\sd_{cc})$ is graphical diamond $2p$-convex.
\end{theo}

\begin{proof}
The setup is the same as in the proof of Theorem \ref{thm:Sean}. The proof showing that graphical diamond $p$-convexity follows from \eqref{eq:short-q} is valid for quasi-metrics (see \cite[Proposition 2.9]{Eskenazis_PhD} for instance). It thus remains to show that \eqref{eq:parallelogram} implies that 
	\begin{equation} \label{eq:short-2p}
		\frac{1}{2^{2p}}\sd(\bold{w},\bold{y})^{2p}+\frac{1}{(2C'C)^{2p}}\sd(\bold{x},\bold{z})^{2p} \le \frac14\sd(\bold{w},\bold{x})^{2p} + \frac14 \sd(\bold{x},\bold{y})^{2p} + \frac14 \sd(\bold{y},\bold{z})^{2p} + \frac14 \sd(\bold{z},\bold{w})^{2p},
	\end{equation}
where $C'$ is the quasi-triangle inequality constant of $\sd$.

First plug in $\bold{a} = \bold{w}^{-1}\bold{z}$ and $\bold{b} = \bold{w}^{-1}\bold{y}$ in \eqref{eq:parallelogram} to obtain
\begin{align*}
N(\delta_{1/2}(\bold{w}^{-1}\bold{y}))^{2p} + \frac{1}{C^{2p}}N([\delta_{1/2}(\bold{w}^{-1}\bold{y})]^{-1}\bold{w}^{-1}\bold{z})^{2p} & \le \frac{1}{2}N(\bold{w}^{-1}\bold{z})^{2p} + \frac{1}{2}N((\bold{w}^{-1}\bold{y})^{-1}\bold{w}^{-1}\bold{z})^{2p} \\\
	& = \frac{1}{2}N(\bold{w}^{-1}\bold{z})^{2p} + \frac{1}{2}N(\bold{y}^{-1}\bold{z})^{2p}.
\end{align*}
Then plug in $\bold{a} = \bold{w}^{-1}\bold{x}$ and $\bold{b} = \bold{w}^{-1}\bold{y}$ in \eqref{eq:parallelogram} and get 
\begin{align*}
N(\delta_{1/2}(\bold{w}^{-1}\bold{y}))^{2p} + \frac{1}{C^{2p}}N([\delta_{1/2}(\bold{w}^{-1}\bold{y})]^{-1}\bold{w}^{-1}\bold{x})^{2p} & \le \frac{1}{2}N(\bold{w}^{-1}\bold{x})^{2p} + \frac{1}{2}N((\bold{w}^{-1}\bold{y})^{-1}\bold{w}^{-1}\bold{x})^{2p} \\\
	& = \frac{1}{2}N(\bold{w}^{-1}\bold{x})^{2p} + \frac{1}{2}N(\bold{y}^{-1}\bold{x})^{2p}.
\end{align*}
Averaging the two inequalities above and using the definition and homogeneity of $\sd$ yields
\begin{align*}
\frac{\sd(\bold{w},\bold{y})^{2p}}{2^{2p}} & + \frac{1}{C^{2p}}\Big( \frac{\sd(\bold{w}^{-1}\bold{z},\delta_{1/2}(\bold{w}^{-1}\bold{y}))^{2p} + \sd(\bold{w}^{-1}\bold{x},\delta_{1/2}(\bold{w}^{-1}\bold{y}))^{2p}}{2} \Big) \\
& \le \frac14 \sd(\bold{z},\bold{w})^{2p} + \frac14 \sd(\bold{z},\bold{y})^{2p} + \frac14\sd(\bold{x},\bold{w})^{2p} + \frac14 \sd(\bold{x},\bold{y})^{2p}.
\end{align*}
Then by convexity, the $C'$-quasi-triangle inequality of $\sd$, and the left-invariance of $\sd$, we get
\begin{align*}
\frac{\sd(\bold{w},\bold{y})^{2p}}{2^{2p}} + \frac{\sd(\bold{x},\bold{z})^{2p}}{(2C'C)^{2p}}  \le \frac14\sd(\bold{w},\bold{x})^{2p} + \frac14 \sd(\bold{x},\bold{y})^{2p} + \frac14 \sd(\bold{y},\bold{z})^{2p} + \frac14 \sd(\bold{z},\bold{w})^{2p},
\end{align*}
which is exactly \eqref{eq:short-2p}.
\end{proof}

\section{Relaxations of the fork inequality and of Markov convexity}
\label{sec:relaxed-Markov}

In this section, we discuss some natural relaxations of the fork inequality and of Markov convexity. The following is a local analogue of umbel convexity.

\begin{defi}
We will say that a metric space $\metXd$ is \emph{fork $p$-convex} if there exists $\Pi>0$ such that for all $k\ge 1$ and all $f\colon \bin{2^k}\to \metX$,
\begin{align}
\label{eq:forkqconvex}
\nonumber \sum_{s=1}^{k-1}\frac{1}{2^{k-1-s}}\sum_{t=1}^{2^{k-1-s}} & \min_{\vep\in\cube{t2^{s+1}-2^s}}\min_{\delta,\delta'\in\cube{{2^s}-1}}\frac{\dX(f(\vep,-1,\delta),f(\vep,1,\delta'))^p}{2^{sp}}\\			
					 & \le \Pi^p \frac{1}{2^k}\sum_{\ell=1}^{2^k}\max_{\vep\in \cube{\ell}}\dX(f(\vep_1,\dots,\vep_{\ell-1}),f(\vep_1,\dots, \vep_{\ell}))^p,
\end{align}
and we will denote by $\Pi_p^f(\metX)$ the least constant $\Pi$ such that \eqref{eq:forkqconvex} holds.
\end{defi}

We will see that the fork $p$-convexity inequality \eqref{eq:forkqconvex} follows from the following relaxation of the $p$-fork inequality:

\noindent For all $w,x,y,z\in \met X$, 
\begin{equation}
\label{eq:relaxed-q-fork}
\frac{1}{2^p}\min\{\dX(w,x)^p,\dX(w,y)^p\}+\frac{\dX(x,y)^p}{4^{p}K^p}\le \frac12 \dX(z,w)^p+\frac12 \max\{\dX(z,x)^p,\dX(z,y)^p\}
\end{equation} 

The fact that fork $p$-convexity implies umbel $p$-convexity is not completely immediate due to the limit inferior in the definition of umbel $p$-convexity. To prove it, we first need a technical lemma.

\begin{lemm} \label{lem:binaryembedding}
For each $k \geq 0$, let $V_k$ denote the subset of $\wtree{\bN}{\leq k} \times \wtree{\bN}{\leq k}$ consisting of all pairs $(\nbar,\mbar)$ such that $\nbar$ extends $\mbar$ (abbreviated by $\nbar \preceq \mbar$ and meaning that $\nbar=(n_1,\dots, n_i)$ and $\mbar=(m_1,\dots, m_j)$ satisfy $i\le j$ and $n_1=m_1, \dots, n_i=m_i$). For every $k \in \bN$ and function $J \colon V_k \to \bN$, there exists a map $\phi:=\phi(k,J) \colon \bin{k} \to \tree_k^\omega$ satisfying the following property 
$$ (\star)\begin{cases}
				\phi \text{ is a height and extension preserving graph morphism, }\\
				\text{ and }\\
				\text{for every }\vep,\delta \text{ for which } (\vep,1,\delta) \in \bin{k}, 
				\text{ there exists} \\
				\text{an integer }j' = j'(\vep,\delta) \geq J(\phi(\vep),\phi(\vep,1,\delta)) \text{ such that, for every } \delta' \text{ for which} \\ (\vep,-1,\delta') \in \bin{k}, \text{ there exists }
				 \bar{\eta} = \bar{\eta}(\delta') \in \tree_k^\omega \text{ such that 
				 } \phi(\vep,-1,\delta') = (\phi(\vep),j',\bar{\eta}).
			\end{cases}$$
\end{lemm}

\begin{proof}
The proof is by induction on $k$. The base case $k=0$ is vacuous. Suppose the lemma holds for some $k \geq 0$. Let $J \colon V_{k+1} \to \bN$ be any function. Observe that for all $r\ge 1$, $[\{r,r+1,\dots\}]^{\le k}$ equipped with the natural tree order is isomorphic to $[\bN]^{\le k}$. Denote by $V_k(r)$ the subset of $[\{r,r+1,\dots\}]^{\le k} \times [\{r,r+1,\dots\}]^{\le k}$ consisting of all pairs $(\nbar,\mbar)$ such that $\nbar$ extends $\mbar$. Define a function $J_1 \colon V_k(2) \to \bN$ by $J_1(\nbar,\mbar) \eqd J((1,\nbar),(1,\mbar))$. Apply the inductive hypothesis to $J_1$ to obtain a function $\phi_1 \colon \bin{k} \to \big([\{2,3,\dots\}]^{\le k},\sd_\tree\big)$ satisfying $(\star)$. Set $j_0 \eqd \max\{J(\emptyset,(1,\phi_1(\delta))) \colon \delta \in \bin{k}\}$, and note that this maximum exists since $\bin{k}$ is finite. Define a function $J_0 \colon V_k(j_0+1) \to \bN$ by $J_0(\nbar,\mbar) \eqd J((j_0,\nbar),(j_0,\mbar))$. Apply the inductive hypothesis to $J_0$ to obtain a function $\phi_0 \colon \bin{k} \to \big([\{j_0+1,j_0+2,\dots\}]^{\le k},\sd_\tree\big)$ satisfying $(\star)$. Finally we define $\phi \colon \bin{k+1} \to \tree_{k+1}^\omega$ by
\begin{itemize}
    \item $\phi(\emptyset) \eqd \emptyset$,
    \item $\phi(1,\vep) \eqd (1,\phi_1(\vep))$, and
    \item $\phi(-1,\vep) \eqd (j_0,\phi_0(\vep))$.
\end{itemize}

We now check that $\phi$ satisfies the desired properties. Obviously, $\phi$ is a height and extension preserving graph morphism since both $\phi_1$ and $\phi_0$ are. Let $(\vep,1,\delta) \in \bin{k+1}$. If $\vep$ is of the form $(1,\vep')$, then $(\star)$ holds since it holds for $\phi_1$, and if $\vep$ is of the form $(-1,\vep')$, then $(\star)$ holds since it holds for $\phi_0$. It remains to consider $\vep = \emptyset$. In this case, we choose $j' \eqd j_0$ and for any $(-1,\delta') \in \bin{k+1}$, we choose $\bar{\eta} \eqd \phi_0(\delta')$. These choices witness the satisfaction of $(\star)$.
\end{proof}

\begin{prop}
Let $p \in(0,\infty)$. Every fork $p$-convex metric space $\metXd$ is umbel $p$-convex. Moreover, $\Pi_p^u(\metX) \leq \Pi_p^f(\metX)$.
\end{prop}

\begin{proof}
Let $k\ge 1$ and $f: \wtree{\bN}{\le 2^k}\to \metX$ a map. Without loss of generality, we may assume the right-hand-side of \eqref{eq:umbel-p-convex} is finite. Let $\gamma>0$ be arbitrary. For each $(\nbar,(\nbar,\bar{\delta})) \in V_k$, choose $J(\nbar,(\nbar, \bar{\delta})) \in \bN$ such that, for all $j \geq J(\nbar,(\nbar,\bar{\delta}))$,
\begin{equation*}
    \liminf_{j\to\infty}\inf_{\stackrel{\bar{\eta}\in\wtree{\bN}{2^s-1}\colon}{(\nbar,j,\bar{\eta})\in \wtree{\bN}{\le 2^k}}}\frac{\dX(f(\nbar,\bar{\delta}),f(\nbar,j,\bar{\eta}))^p}{2^{sp}}
     \le \inf_{\stackrel{\bar{\eta}\in\wtree{\bN}{2^s-1}\colon}{(\nbar, j, \bar{\eta})\in \wtree{\bN}{\le 2^k}}} \frac{\dX(f(\nbar,\bar{\delta}),f(\nbar, j, \bar{\eta}))^p}{2^{sp}} + \frac{\gamma}{k}.
\end{equation*}
Now apply Lemma \ref{lem:binaryembedding} to the function $J$ defined as above to get a height and extension preserving graph morphism $\phi \colon \bin{2^k} \to \tree^\omega_{2^k}$ satisfying $(\star)$. Let $A$ denote the left-hand-side of the fork $p$-convexity inequality \eqref{eq:forkqconvex} applied to the map $f \circ \phi \colon \bin{2^k} \to \metX$, i.e.

\begin{align*}
    A \eqd \sum_{s=1}^{k-1}\frac{1}{2^{k-1-s}}\sum_{t=1}^{2^{k-1-s}} & \min_{\vep\in\cube{t2^{s+1}-2^s}}\min_{\delta,\delta'\in\cube{{2^s}-1}} \frac{\dX(f(\phi(\vep,-1,\delta)),f(\phi(\vep,1,\delta')))^p}{2^{sp}}.
\end{align*}
Let $B$ denote the left-hand-side of the umbel $p$-convexity inequality \eqref{eq:umbel-p-convex} applied to the map $f\colon \wtree{\bN}{\le 2^k}\to \metX$, i.e.

\begin{align*}
    B \eqd \sum_{s=1}^{k-1}\frac{1}{2^{k-1-s}}\sum_{t=1}^{2^{k-1-s}}\inf_{\nbar\in\wtree{\bN}{t2^{s+1}-2^s}}\inf_{\stackrel{\bar{\delta}\in\wtree{\bN}{2^s}\colon}{(\nbar,\bar{\delta})\in \wtree{\bN}{\le 2^k}}}\liminf_{j\to\infty}\inf_{\stackrel{\bar{\eta}\in\wtree{\bN}{2^s-1}\colon}{(\nbar,j,\bar{\eta})\in \wtree{\bN}{\le 2^k}}}\frac{\dX(f(\nbar,\bar{\delta}),f(\nbar,j,\bar{\eta}))^p}{2^{sp}}.
\end{align*}

Given $\vep\in\cube{t2^{s+1}-2^s}$ and $\delta \in \cube{{2^s}-1}$, it follows from the definitions of $J$, $\phi$, and property $(\star)$ that there exists an integer $j' = j'(\vep,\delta) \geq J(\phi(\vep),\phi(\vep,1,\delta))$ such that, for every $\delta' \in \cube{{2^s}-1}$, there exists $\bar{\eta} = \bar{\eta}(\delta') \in \wtree{\bN}{2^s-1}$ such that $\phi(\vep,-1,\delta') = (\phi(\vep),j'(\vep,\delta),\bar{\eta}(\delta'))$. Since $\phi(\vep) \in \wtree{\bN}{t2^{s+1}-2^s}$ and $\phi(\vep,1,\delta)\in \wtree{\bN}{t2^{s+1}}$, we have
\begin{align*}
	&\inf_{\nbar\in\wtree{\bN}{t2^{s+1}-2^s}} \inf_{\stackrel{\bar{\delta}\in\wtree{\bN}{2^s}\colon}{(\nbar,\bar{\delta})\in \wtree{\bN}{\le 2^k}}}  \liminf_{j\to\infty} \inf_{\stackrel{\bar{\eta}\in\wtree{\bN}{2^s-1}\colon}{(\nbar, j , \bar{\eta})\in \wtree{\bN}{\le 2^k}}} \dX\big(f(\nbar,\bar{\delta}),f(\nbar, j , \bar{\eta}) \big)^p\\
    \leq & \min_{\vep\in\cube{t2^{s+1}-2^s}}\min_{\delta\in\cube{{2^s}-1}} \liminf_{j\to\infty} \inf_{\stackrel{\bar{\eta}\in\wtree{\bN}{2^s-1}\colon}{(\nbar, j , \bar{\eta})\in \wtree{\bN}{\le 2^k}}}\dX(f(\phi(\vep,1,\delta)),f(\phi(\vep),j,\bar{\eta}))^p \\
    \leq & \min_{\vep\in\cube{t2^{s+1}-2^s}}\min_{\delta\in\cube{{2^s}-1}} \inf_{\stackrel{\bar{\eta}\in\wtree{\bN}{2^s-1}\colon}{(\nbar, j , \bar{\eta})\in \wtree{\bN}{\le 2^k}}}\dX(f(\phi(\vep,1,\delta)),f(\phi(\vep),j'(\vep,\delta),\bar{\eta}))^p + \frac{\gamma}{k} \\
    \leq & \min_{\vep\in\cube{t2^{s+1}-2^s}}\min_{\delta\in\cube{{2^s}-1}} \min_{\delta'\in\cube{{2^s}-1}}\dX(f(\phi(\vep,1,\delta)),f(\phi(\vep),j'(\vep,\delta),\bar{\eta}(\delta')))^p + \frac{\gamma}{k} \\
    = & \min_{\vep\in\cube{t2^{s+1}-2^s}}\min_{\delta\in\cube{{2^s}-1}} \min_{\delta'\in\cube{{2^s}-1}}\dX(f(\phi(\vep,1,\delta)),f(\phi(\vep,-1,\delta')))^p + \frac{\gamma}{k}
	\end{align*}
Hence, after dividing by $2^{sp}$ and summing appropriately over $t$ and $s$, we have $B \leq A + \gamma$. Since $\gamma > 0$ was arbitrary, we have $A\le B$. To conclude that $\Pi_p^u(X) \leq \Pi_p^f(X)$, it remains to observe that  
\begin{align*}
    \frac{1}{2^k}\sum_{\ell=1}^{2^k}\max_{\vep\in \cube{\ell}}\dX(f(\phi(\vep_1,\dots,\vep_{\ell-1})),f(\phi(\vep_1,\dots, \vep_{\ell})))^p \\ \leq \frac{1}{2^{k}}\sum_{\ell=1}^{2^{k}}\sup_{\nbar\in \wtree{\bN}{\ell}}\dX(f(n_1,\dots,n_{\ell-1}),f(n_1,\dots,n_{\ell}))^p
\end{align*}
as $\phi$ preserves the extension relation.
\end{proof}

A further relaxation of the $p$-fork inequality is the following:

\noindent For all $w,x,y,z\in \met X$, 
\begin{equation}
\label{eq:super-relaxed-q-fork}
\frac{1}{2^p}\min\{\dX(w,x)^p,\dX(w,y)^p\}+\frac{\dX(x,y)^p}{4^{p}K^p}\le \max\{\dX(z,w)^p,\dX(z,x)^p,\dX(z,y)^p\}
\end{equation} 

\noindent By analogy with terminology surrounding the notion of infrasup-umbel convexity, we will refer to inequality \eqref{eq:super-relaxed-q-fork} as the \emph{infrasup $p$-fork inequality} with constant $K$. We also introduce the following definition.

\begin{defi}
Let $p\in(0,\infty]$. A metric space $\metXd$ is \emph{infrasup-fork $p$-convex} if there exists $C>0$ such that for all $k\ge 1$ and all $f\colon \bin{2^k}\to \metX$ 
\begin{equation}
\label{eq:forkcotypeq}
\left(\sum_{s=1}^{k-1}\min_{\vep \in \bin{2^k-2^s}}\min_{\delta, \delta'\in \cube{2^{s}-1}}\frac{\dX(f(\vep,-1,\delta),f(\vep,1,\delta'))^p}{2^{sp}}\right)^{\frac{1}{p}}\le C \lip(f).
\end{equation} 
We denote by $\Pi_p^{isf}(\metX)$ the least constant $C$ such that \eqref{eq:forkcotypeq} holds for all $k\ge 1$ and all maps $f\colon \bin{2^k}\to \metX$.
\end{defi}

In the next theorem we gather results that are local analogues of those in Section \ref{sec:umbel}.

\begin{theo}\
\label{thm:fork}
\begin{enumerate}
\item If a metric space $\metXd$ satisfies inequality \eqref{eq:relaxed-q-fork} with constant $K>0$, then $\metX$ is fork $p$-convex. Moreover, $\Pi_p^f(\metX)\le 4K$.
\item If a metric space $\metXd$ satisfies the infrasup $p$-fork inequality \eqref{eq:super-relaxed-q-fork} with constant $K>0$, then $\metX$ is infrasup-fork $p$-convex. Moreover, $\Pi_p^{isf}(\metX)\le 4K$.
\end{enumerate}
\end{theo}

\begin{proof}
The first assertion can be proven much in the same way as Theorem \ref{thm:pumbel->umbelp}, and we leave this verification to the dutiful reader.
We will prove the second assertion. The proof is rather similar to the proof of Theorem \ref{thm:pumbel->umbelp} albeit on some occasions where some slightly different justifications are needed.  It will be sufficient to show by induction on $k$ that for all maps $f\colon \bin{2^k}\to \metX$, and all $\rho\in \cube{}$ 
\begin{align*}
\min_{\delta\in\cube{2^{k}-1}}\frac{\dX(f(\emptyset),f(\rho,\delta))^p}{2^{kp}} + & \frac{1}{4^{p}K^p}\sum_{s=1}^{k-1}\min_{\vep\in\bin{2^{k}-2^s}}\min_{\delta,\delta'\in\cube{2^s-1}}\frac{\dX(f(\vep,-1,\delta),f(\vep,1,\delta'))^p}{2^{sp}}\\			
					 \le & \max_{1\le \ell \le 2^k}\max_{\vep\in\cube{\ell}} \dX(f(\vep_1,\dots, \vep_{\ell-1}),f(\vep_1,\dots, \vep_{\ell}))^p.
\end{align*}
For the base case $k=1$, the inequality reduces to
\begin{align*}
\min_{\delta\in\cube{}}\frac{\dX(f(\emptyset),f(\rho,\delta))}{2} \le   \max\Big\{\max_{\vep \in \cube{}}\dX(f(\emptyset),f(\vep)), \max_{\vep \in \cube{2}}\dX(f(\vep_1),f(\vep_1,\vep_2))\Big\},
\end{align*}
which is an immediate consequence of the triangle inequality.
We now proceed to the inductive step and fix $\iota\in \cube{}$ and $f\colon \bin{2^{k+1}}\to \metX$. Let $\mu\in \cube{2^k-1}$ such that 
\begin{equation*}
\frac{\dX(f(\troot),f(\iota,\mu))^p}{2^{kp}} = \min_{\delta\in\cube{2^{k}-1}}\frac{\dX(f(\troot),f(\iota,\delta))^p}{2^{kp}},
\end{equation*}
and for each $\rho\in\cube{}$, choose $\nu(\rho)\in\cube{2^{k}-1}$
\begin{equation*}
\frac{\dX(f(\iota,\mu),f(\iota,\mu,\rho,\nu(\rho)))^p}{2^{kp}} = \min_{\delta\in\cube{2^{k}-1}}\frac{\dX(f(\iota,\mu),f(\iota,\mu,\rho,\delta))^p}{2^{kp}}.
\end{equation*}
By the induction hypothesis applied to the restriction of $f$ to $\bin{2^{k}}$ (and with $\rho=\iota$) we get
\begin{align} \label{eq:aux1} \nonumber
\min_{\delta\in\cube{2^{k}-1}}\frac{\dX(f(\troot),f(\iota,\delta))^p}{2^{kp}} & + \frac{1}{4^{p}K^p}\sum_{s=1}^{k-1}\min_{\vep\in\bin{2^{k}-2^s}}\min_{\delta,\delta' \in \cube{2^s-1}}\frac{\dX(f(\vep,-1,\delta),f(\vep,1,\delta'))^p}{2^{sp}}\\			
					 & \le  \max_{1\le \ell \le 2^k}\max_{\vep\in \cube{\ell}}\dX(f(\vep_1,\dots, \vep_{\ell-1}),f(\vep_1,\dots, \vep_{\ell}))^p.
\end{align}
By taking the first minimum in the sum and the maximum over larger sets we get
\begin{align} \label{eq:forkcotype1} \nonumber
\min_{\delta\in\cube{2^{k}-1}}\frac{\dX(f(\troot),f(\iota,\delta))^p}{2^{kp}} & + \frac{1}{4^{p}K^p}\sum_{s=1}^{k-1}\min_{\vep\in\bin{2^{k+1}-2^s}}\min_{\delta,\delta' \in \cube{2^s-1}}\frac{\dX(f(\vep,-1,\delta),f(\vep,1,\delta'))^p}{2^{sp}}\\			
					 & \le  \max_{1\le \ell \le 2^{k+1}}\max_{\vep\in \cube{\ell}}\dX(f(\vep_1,\dots, \vep_{\ell-1}),f(\vep_1,\dots, \vep_{\ell}))^p.
\end{align}
On the other hand, the induction hypothesis applied to $g(\vep)\eqd f((\iota,\mu),\vep)$ where $\vep\in \bin{2^{k}}$ gives
\begin{align*}
\min_{\delta\in\cube{2^{k}-1}}\frac{\dX(g(\troot),g(\iota,\delta))^p}{2^{kp}} & +  \frac{1}{4^{p}K^p}\sum_{s=1}^{k-1}\min_{\vep\in\bin{2^{k}-2^s}}\min_{\delta,\delta' \in \cube{2^s-1}}\frac{\dX(g(\vep,-1,\delta),g(\vep,1,\delta'))^p}{2^{sp}}\\			
					& \le  \max_{1\le \ell \le 2^k}\max_{\vep\in \cube{\ell}}\dX(g(\vep_1,\dots, \vep_{\ell-1}),g(\vep_1,\dots, \vep_{\ell}))^p.
\end{align*}
Observe first that,
\begin{align*}
\max_{1\le \ell \le 2^k}\max_{\vep\in \cube{\ell}}\dX(g(\vep_1,\dots, \vep_{\ell-1}), & g(\vep_1,\dots, \vep_{\ell}))^p \\
 & = \max_{1\le \ell \le 2^k}\max_{\vep\in \cube{\ell}}\dX(f(\iota,\mu,\vep_1,\dots, \vep_{\ell-1}),f(\iota,\mu,\vep_1,\dots, \vep_{\ell}))^p\\
 			& \le \max_{1\le \ell \le 2^{k+1}}\max_{\vep \in \cube{\ell}}\dX(f(\vep_1,\dots, \vep_{\ell-1}), f(\vep_1,\dots, \vep_{\ell}))^p,
\end{align*}
since we are maximizing over the set of all the edges instead of a subset of it. Also, for each $s = 1, \dots, k-1$,
\begin{align*}
\min_{\vep\in\bin{2^{k}-2^s}}\min_{\delta,\delta' \in \cube{2^s-1}} & \frac{\dX(g(\vep,-1,\delta),g(\vep,1,\delta'))^p}{2^{sp}}\\
 & =  \min_{\vep\in\bin{2^{k}-2^s}} \min_{\delta,\delta' \in \cube{2^s-1}}\frac{\dX(f(\iota,\mu,\vep,-1,\delta),f(\iota,\mu,\vep,1,\delta'))^p}{2^{sp}}\\
 & \ge  \min_{\vep\in\bin{2^{k+1}-2^s}} \min_{\delta,\delta' \in \cube{2^s-1}}\frac{\dX(f(\vep,-1,\delta),f(\vep,1,\delta'))^p}{2^{sp}},
\end{align*}
since $(\iota,\mu,\vep)\in \bin{2^{k+1}-2^s}$ for all $\vep \in \bin{2^{k}-2^s}$.

Therefore, it follows from the two relaxations above that 
\begin{align}\label{eq:forkcotype2}
\nonumber \min_{\delta \in \cube{2^{k}-1}} \frac{\dX(f(\iota,\mu),f(\iota,\mu,\rho,\delta))^p}{2^{kp}} & + \frac{1}{4^{p}K^p}\sum_{s=1}^{k-1}\min_{\vep\in\bin{2^{k+1}-2^s}} \min_{\delta,\delta' \in \cube{2^s-1}}\frac{\dX(f(\vep,-1,\delta),f(\vep,1,\delta'))^p}{2^{sp}}\\			
					& \le \max_{1 \le \ell \le 2^{k+1}} \max_{\vep \in \cube{\ell}}\dX(f(\vep_1,\dots, \vep_{\ell-1}), f(\vep_1,\dots, \vep_{\ell}))^p.
\end{align}
Since the sum and right hand side in \eqref{eq:forkcotype2} do not depend on $\rho$, it follows from \eqref{eq:forkcotype1} and \eqref{eq:forkcotype2} that 
\begin{align}\label{eq:forkcotype3}
\nonumber \frac{1}{2^{kp}}\Big(\max\Big\{\min_{\delta \in \cube{2^{k}-1}}\dX(f(\troot),f(\iota,\delta))^p, \max_{\rho\in \cube{}} \min_{\delta \in \cube{2^{k}-1}}\dX(f(\iota,\mu),f(\iota,\mu, \rho,\delta))^p\Big\}\Big)\\
 + \frac{1}{4^{p}K^p}\sum_{s=1}^{k-1}\min_{\vep \in \bin{2^{k+1}-2^s}} \min_{\delta,\delta' \in \cube{2^s-1}}\frac{\dX(f(\vep,-1,\delta),f(\vep,1,\delta'))^p}{2^{sp}}\\			
\nonumber					\le  \max_{1\le \ell \le 2^{k+1}} \max_{\vep \in \cube{\ell}}\dX(f(\vep_1,\dots, \vep_{\ell-1}), f(\vep_1,\dots, \vep_{\ell}))^p.
\end{align}
If we let $w\eqd f(\troot)$, $z\eqd f(\iota,\mu)$, and $x_\rho\eqd f(\iota,\mu,\rho, \nu(\rho))$ (for $\rho\in \cube{}$) it follows from how $\mu$ and $\nu(\rho)$ were chosen, that
\begin{align}\label{eq:forkcotype4}
\frac{1}{2^{kp}} & \Big(\max\Big\{\dX(w,z)^p,\max_{\rho \in \cube{}}\dX(z,x_\rho)^p\Big\}\Big) = \\
\nonumber   & \frac{1}{2^{kp}}\Big(\max\Big\{\min_{\delta \in \cube{2^{k}-1}}\dX(f(\troot),f(\iota,\delta))^p, \max_{\rho\in \cube{}} \min_{\delta \in \cube{2^{k}-1}}\dX(f(\iota,\mu),f(\iota,\mu, \rho,\delta))^p\Big\}\Big).
\end{align}
Inequality \eqref{eq:relaxed-q-fork} combined with \eqref{eq:forkcotype3} and \eqref{eq:forkcotype4} gives 
\begin{align}\label{eq:forkcotype5}
\nonumber \frac{1}{2^{(k+1)p}}\min\{\dX(w,x_{-1})^p,\dX(w,x_{1})^p\}+\frac{1}{4^{p}K^p}\frac{1}{2^{kp}} \dX(x_{-1},x_{1})^p\\
 + \frac{1}{4^{p}K^p}\sum_{s=1}^{k-1}\min_{\vep \in \bin{2^{k+1}-2^s}} \min_{\delta,\delta' \in \cube{2^s-1}}\frac{\dX(f(\vep,-1,\delta),f(\vep,1,\delta'))^p}{2^{sp}}\\		
\nonumber 					\le  \max_{1\le \ell \le 2^{k+1}}\max_{\vep \in \cube{\ell}}\dX(f(\vep_1,\dots, \vep_{\ell-1}), f(\vep_1,\dots, \vep_{\ell}))^p.
\end{align}
Now observe that
\begin{align*}
\min\{\dX(w,x_{-1})^p,\dX(w,x_{1})^p\} & =  \min\{\dX( f(\troot),f(\iota,\mu, -1, \nu(-1)))^p, \dX( f(\troot),f(\iota,\mu, 1, \nu(1)))^p\} \\
				& \ge \min_{\delta \in \cube{2^{k+1}-1}}\dX( f(\troot),f(\iota,\delta))^p,
\end{align*}
and
\begin{align*}
\dX(x_{-1},x_1)^p  & = \dX(f(\iota,\mu,-1, \nu(-1)),f(\iota,\mu, 1, \nu(1)))^p \\
								& \ge   \min_{\delta,\delta' \in \cube{2^k-1}} \dX(f(\iota,\mu,-1, \delta),f(\iota,\mu, 1, \delta'))^p\\
							         & \ge   \min_{\vep \in \bin{2^{k}}} \min_{\delta,\delta' \in \cube{2^k-1}} \dX(f(\vep,-1, \delta),f(\vep, 1, \delta'))^p
\end{align*}
Plugging in the two relaxed inequalities above in \eqref{eq:forkcotype5} we obtain
\begin{align*}
\min_{\delta \in \cube{2^{k+1}-1}}\frac{\dX( f(\troot),f(\iota,\delta))^p}{2^{(k+1)p}}+\frac{1}{4^{p}K^p}\min_{\vep \in \bin{2^{k}}} \min_{\delta,\delta' \in \cube{2^k-1}} \frac{\dX(f(\vep,-1, \delta),f(\vep, 1, \delta'))^p}{2^{kp}}\\
 + \frac{1}{4^{p}K^p}\sum_{s=1}^{k-1}\min_{\vep \in \bin{2^{k+1}-2^{s}}}\min_{\delta,\delta' \in \cube{2^s-1}}\frac{\dX(f(\vep,-1,\delta),f(\vep,1,\delta'))^p}{2^{sp}}\\			
					\le  \max_{1\le \ell \le 2^{k+1}}\max_{\vep \in \cube{\ell}}\dX(f(\vep_1,\dots, \vep_{\ell-1}), f(\vep_1,\dots, \vep_{\ell}))^p,
\end{align*}
and hence 
\begin{align*}
\min_{\delta \in \cube{2^{k+1}-1}}\frac{\dX( f(\troot),f(\iota,\delta))^p}{2^{(k+1)p}} & + \frac{1}{4^{p}K^p}\sum_{s=1}^{k}\min_{\vep \in \bin{2^{k+1}-2^{s}}}\min_{\delta,\delta' \in \cube{2^s}}\frac{\dX(f(\vep,-1,\delta),f(\vep,1,\delta'))^p}{2^{sp}}\\			
					& \le  \max_{1\le \ell \le 2^{k+1}}\max_{\vep \in \cube{\ell}}\dX(f(\vep_1,\dots, \vep_{\ell-1}), f(\vep_1,\dots, \vep_{\ell}))^p,
\end{align*}
which completes the induction step.
\end{proof}

Infrasup-fork $p$-convexity is an obvious relaxation of fork $p$-convexity. It is less obvious that infrasup-fork $p$-convexity is also a relaxation of Markov $p$-convexity\footnote{It is unclear if fork $p$-convexity is implied by Markov $p$-convexity, see Problem~\ref{pb:markov->fork}.}, and we need a preliminary lemma that allows us to pass from the stochastic definition of Markov convexity to a deterministic inequality.

\begin{lemm} \label{lem:Markovfork1}
Let $\metXd$ be a metric space, $p>0$, $k \geq 1$, $\{W_t\}_{t \in \bZ}$ the simple directed random walk on $\bin{2^k}$ starting at the root, and $f\colon \bin{2^k} \to \metX$ a map. Then 
	\begin{enumerate}[(i)]
		\item for all  $0 \leq s \leq k$ and $2^s \leq t \leq 2^k$,
			\begin{align*} 
				\Ex & [\dX(f(W_t),f(\tilde{W}_t(t-2^s)))^p] \\
				& = \frac{1}{2^{t-2^s}} \sum_{\vep \in \cube{t-2^s}} \sum_{\ell=1}^{2^s} \frac{1}{2^\ell} \frac{1}{2^{\ell-1}} & \sum_{\vep' \in \cube{\ell-1}}\frac{2}{4^{2^s-\ell}} \sum_{\delta,\delta' \in \cube{2^s-\ell}} \dX(f(\vep,\vep',-1,\delta),f(\vep,\vep',1,\delta'))^p,
			\end{align*}
		\item and \begin{align} \label{eq:Markovfork2}
				\sum_{s=1}^{k-1} \frac{1}{2^k}\sum_{t=2^s}^{2^k} & \dfrac{\Ex[\dX(f(W_t),f(\tilde{W}_t(t-2^s)))^p]}{2^{sp}} \nonumber \\
				& \geq \frac{1}{2}\sum_{s=1}^{k-1}\min_{\vep \in \cube{\leq 2^k-2^s}} \min_{\delta,\delta' \in \cube{2^s-1}} \dfrac{\dX(f(\vep,-1,\delta),f(\vep,1,\delta'))^p}{2^{sp}}.
			        \end{align}
	\end{enumerate}
\end{lemm}

\begin{proof}
The proof of $(i)$ goes by performing consecutive ad-hoc conditionings. It is clear that $W_{t-2^s} = \tilde{W}_{t-2^s}(t-2^s)$ almost surely and both are uniformly distributed over the set $\cube{t-2^s}$, which has cardinality $2^{t-2^s}$. Therefore, by the law of total expectations,
	\begin{equation*}
		\Ex[\dX(f(W_t),f(\tilde{W}_t(t-2^s)))^p] = \frac{1}{2^{t-2^s}} \sum_{\vep \in \cube{t-2^s}} \Ex[\dX(f(W_t),f(\tilde{W}_t(t-2^s)))^p \mid W_{t-2^s} = \vep]
	\end{equation*}
Fix $\vep$ in the sequel.

Next, consider the event, denoted $\cE_\ell$, that $W_{t-2^s}$ and $\tilde{W}_{t-2^s}(t-2^s)$ branch from each other immediately before making the $\ell$th step after $\vep$. Formally, for every $1 \leq \ell \leq 2^s$ and $\vep' \in \cube{\ell-1}$,
	\begin{equation*}
		\cE_\ell \eqd \bigcup_{\vep' \in \cube{\ell-1}} A_\ell(\vep')
	\end{equation*}
where
	\begin{equation*}
		A_\ell(\vep') \eqd \bigcup_{u \in\{-1,1\}} \left\{\substack{W_{t-2^s+\ell-1} = \tilde{W}_{t-2^s+\ell-1}(t-2^s) = (\vep,\vep'), \\ W_{t-2^s+\ell}  = (\vep,\vep', u), \\ \tilde{W}_{t-2^s+\ell}(t-2^s) = (\vep,\vep', -u)}\right\}.
	\end{equation*}
The events $\cE_\ell$, $1\le \ell \le 2^s$, are clearly disjoint, and a simple computation shows that $\cE_\ell$ occurs with probability $2^{-\ell}$. Consequently,
	\begin{align*}
		\Ex[\dX(f(W_t), & f(\tilde{W}_t(t-2^s)))^p  \mid  W_{t-2^s} = \vep] \\
		& = \sum_{\ell=1}^{2^s} \frac{1}{2^{\ell}} \Ex[\dX(f(W_t), f(\tilde{W}_t(t-2^s)))^p  \mid \{W_{t-2^s} = \vep\}, \cE_\ell].
	\end{align*}
For each fixed $\ell$, the events $\{A_\ell(\vep')\}_{\vep' \in \cube{\ell-1}}$ are obviously disjoint and, after conditioning on $\cE_\ell$, each occur with probability $2^{1-\ell}$. Thus,
	\begin{align*}
		\Ex[\dX(f(W_t), & f(\tilde{W}_t(t-2^s)))^p  \mid \{W_{t-2^s} = \vep\}, \cE_\ell] \\
		& =  \frac{1}{2^{\ell-1}}\sum_{\vep' \in \cube{\ell-1}} \Ex[\dX(f(W_t), f(\tilde{W}_t(t-2^s)))^p  \mid \{W_{t-2^s} = \vep\}, \cE_\ell, A_\ell(\vep')].
	\end{align*}
Finally, recalling the definitions of the events we have conditioned on, the inequality below clearly holds
\begin{align*}
\Ex[\dX(f(W_t),f(\tilde{W}_t(t-2^s)))^p & \mid \{W_{t-2^s} = \vep\}, \cE_\ell, A_\ell(\vep')] \\
& = \sum_{u \in \{-1,1\}}\frac{1}{4^{2^s-\ell}} \sum_{\delta,\delta' \in \cube{2^s-\ell}} \dX(f(\vep,\vep',u,\delta),f(\vep,\vep',-u,\delta'))^p \\
& = \frac{2}{4^{2^s-\ell}} \sum_{\delta,\delta' \in \cube{2^s-\ell}} \dX(f(\vep,\vep',-1,\delta),f(\vep,\vep',1,\delta'))^p.
\end{align*}
Walking back through the chain of equalities we have the desired equality.

We now use $(i)$ to show $(ii)$.  Observe first that the inequality
	\begin{align*}\label{lem:Markovfork2}
\sum_{s=1}^{k-1}\frac{1}{2^k-2^s+1}\sum_{t=0}^{2^k-2^s}\frac{1}{2^t} \sum_{\vep \in \cube{t}} \frac{1}{4^{2^s-1}}&\sum_{\delta,\delta' \in \cube{2^s-1}} \dfrac{\dX(f(\vep,-1,\delta),f(\vep,1,\delta'))^p}{2^{sp}} \\
		& \geq \sum_{s=1}^{k-1}\min_{\vep \in \cube{\leq 2^k-2^s}} \min_{\delta,\delta' \in \cube{2^s-1}} 						\dfrac{\dX(f(\vep,-1,\delta),f(\vep,1,\delta'))^p}{2^{sp}}
	\end{align*}
holds trivially, since the top expression involves convex combinations over the sets $\cube{\le 2^k-2^s}$ and $\cube{2^s-1} \times \cube{2^s-1}$, and the bottom expression involves minima over these sets. To prove inequality \eqref{eq:Markovfork2}, observe that $\frac{1}{2^{k-1}} \geq \frac{1}{2^k-2^s}$ when $s \leq k-1$, and hence
	\begin{align*}
		2\sum_{s=1}^{k-1}\frac{1}{2^k} & \sum_{t=2^s}^{2^k}  \dfrac{\Ex[\dX(f(W_t),f(\tilde{W}_t(t-2^s)))^p]}{2^{sp}} \geq  \sum_{s=1}^{k-1}\frac{1}{2^k-2^s}\sum_{t=2^s}^{2^k}\dfrac{\Ex[\dX(f(W_t),f(\tilde{W}_t(t-2^s)))^p]}{2^{sp}} \\
		& \stackrel{(i)}{\geq} \sum_{s=1}^{k-1}\frac{1}{2^k-2^s}\sum_{t=2^s}^{2^k}\frac{1}{2^{t-2^s}} \sum_{\vep \in \cube{t-2^s}} \frac{1}{2}\frac{2}{4^{2^s-1}} \sum_{\delta,\delta' \in \cube{2^s-1}} \dfrac{\dX(f(\vep,-1,\delta),f(\vep,1,\delta'))^p}{2^{sp}} \\
		& = \sum_{s=1}^{k-1}\frac{1}{2^k-2^s}\sum_{t=0}^{2^k-2^s}\frac{1}{2^t} \sum_{\vep \in \cube{t}} \frac{1}{4^{2^s-1}} \sum_{\delta,\delta' \in \cube{2^s-1}} \dfrac{\dX(f(\vep,-1,\delta),f(\vep,1,\delta'))^p}{2^{sp}},
	\end{align*}
where in the application of $(i)$, we discarded all the terms with $\ell>1$. 
\end{proof}

\begin{prop}\label{prop:Markov-fork}
Every Markov $p$-convex metric space $\metXd$ is infrasup-fork $p$-convex. Moreover, $\Pi^{isf}_p(\metX) \le 2^{1/p}\Pi_p^{M}(\met X)$.
\end{prop}

\begin{proof}
Let $\metXd$ be a Markov $p$-convex metric space and $k,\{W_t\}_{t \in \bZ},f$ as in the statement of Lemma \ref{lem:Markovfork1}. Then
	\begin{align*}
		\sum_{s=1}^{k-1} & \min_{\vep \in \cube{\leq 2^k-2^s}} \min_{\delta,\delta' \in \cube{2^s-1}} \dfrac{\dX(f(\vep,-1,\delta),f(\vep,1,\delta'))^p}{2^{sp}} \\
		& \overset{\eqref{eq:Markovfork2}}{\leq}  2\sum_{s=1}^{k-1}\frac{1}{2^k}\sum_{t=2^s}^{2^k}\dfrac{\Ex[\dX(f(W_t),f(\tilde{W}_t(t-2^s)))^p]}{2^{sp}} \\
		& \leq  2\Pi^M_p(\metX)^p\frac{1}{2^k}\sum_{t=1}^{\infty}\Ex[\dX(f(W_t),f(W_{t-1}))^p] \\
		& \leq  2\Pi^M_p(\metX)^p\lip(f)^p\frac{1}{2^k}\sum_{t=1}^{2^k}\Ex[\dX(W_t,W_{t-1})^p] \\
		& =  2\Pi^M_p(\metX)^p\lip(f)^p.
	\end{align*}
\end{proof}

\begin{rema}\label{rem:Tessera}
We can show that Tesserra's $p$-inequality \eqref{eq:q-tess} is implied by the Markov $p$-convexity inequality \eqref{eq:p-Markov} using arguments along the lines of those in the proofs of Lemma \ref{lem:Markovfork1} and Proposition \ref{prop:Markov-fork}. 
\end{rema}

We record local analogues of the asymptotic results found in Section \ref{sec:distortion}. These local analogues are extensions of results that are known to be valid for spaces that satisfy the Markov $p$-convexity inequality or Tessera's $p$-inequality. The proofs of these local results are nearly identical to their asymptotic counterparts and can be safely omitted.
The first result deals with distortion lower bounds.

\begin{prop}
\label{prop:bin-tree-dist}
For all $p\in(0,\infty)$, $\Pi^{usf}_p(\bin{2^k})\ge 2(k-1)^{1/p} $ and hence 
\[\cdist{\metY}(\bin{k})=\Omega((\log k)^{\frac{1}{p}}),\]
for every metric space $\metYd$ that is infrasup-fork $p$-convex.
\end{prop}

The second result provides compression lower bounds.

\begin{theo} 
\label{thm:bintree-compression}
Let $p\in(0,\infty)$. Assume that there are non-decreasing maps $\rho,\omega\colon [0,\infty)\to [0,\infty)$ and for all $k\ge 1$ a map $f_k\colon \bin{2^k}\to \metY$ such that for all $x,y\in \bin{2^k}$ 
\begin{equation*}
\rho(\sd_\tree(x,y))\le \dY(f_k(x),f_k(y))\le \omega(\sd_\tree(x,y)).
\end{equation*}
Then,
\begin{equation*}
\int_{1}^\infty \Big(\frac{\rho(t)}{t}\Big)^p\frac{dt}{t} \le \frac{2^p-1}{p}\Pi^{isf}_p(\metY)^p \omega(1)^p.
\end{equation*}
In particular, the compression rate of any equi-coarse embedding of $\{\bin{k}\}_{k\ge 1}$ into a metric space that is infrasup-fork $p$-convex satisfies  
\begin{equation}\label{eq:bin-comp-bound}
\int_{1}^\infty \Big(\frac{\rho(t)}{t}\Big)^p\frac{dt}{t}<\infty.
\end{equation}
\end{theo}

Equipped with Proposition \ref{prop:bin-tree-dist}, we can show that two results from \cite{LNP09} about Markov convexity actually holds for the much weaker notion of infrasup-fork convexity. The proofs are the same as in \cite{LNP09} where the full power of Markov convexity was not needed (these partial results were greatly strengthened in \cite{MendelNaor13} where the proof of Theorem \ref{thm:MN-LNP} was completed). We recall the short arguments for the convenience of the reader. 

\begin{coro}\ 
\begin{enumerate}
\item Let $\banX$ be a Banach space. If $\banX$ is infrasup-fork $p$-convex for some $p \ge 2$, then $\banX$ is super-reflexive and has Rademacher cotype $p+\vep$ for every $\vep>0$. 
\item If a Banach lattice $\banX$ that is infrasup-fork $p$-convex for some $p\ge 2$, then for every $\vep>0$, $\banX$ admits an equivalent norm that $(p+\vep)$-uniformly convex. 
\end{enumerate}
\end{coro}

\begin{proof}
For the first assertion, Proposition \ref{prop:bin-tree-dist} together with Bourgain's super-reflexivity characterization implies that $
\banX$ is super-reflexive. The second part follows from the fact that by Maurey-Pisier theorem \cite{MaureyPisier76} $\banX$ contains the $\ell_{q_\banX}^n$'s where $q_\banX=\inf\{q\colon \banX \textrm{ has cotype } q\}$. By Bourgain's tree embedding we have that $\cdist{\banX}(B_k)=O((\log k)^{1/q_\banX})$ and it follows from Proposition \ref{prop:bin-tree-dist} that $q_\banX\le p$.
The second assertion follows from the first and a renorming result of Figiel \cite{Figiel76} which says that every super-reflexive Banach lattice with cotype $q$ admits an equivalent norm that is $(q+\vep)$-uniformly convex for every $\vep>0$.
\end{proof}

A very interesting dichotomy is contained in \cite{LNP09} where it was proved that for an infinite metric tree $\tree$, $\sup_{k\in \bN}\cdist{\tree}(\bin{k})<\infty$ if and only if $\cdist{\ell_2}(\tree)=\infty$. The following corollary can be found in \cite{LNP09} and the additional assertion $(4')$ follows from the observations of this section.

\begin{coro}
\label{cor:dicho}
Let $\tree$ be an infinite metric tree. The following assertions are equivalent.
\begin{enumerate}
\item $\sup_{k\in \bN}\cdist{\tree}(\bin{k})=1$.
\item $\sup_{k\in \bN}\cdist{\tree}(\bin{k})<\infty$. 
\item $\tree$ is not Markov $p$-convex for any $p\in(1,\infty)$.
\item $\tree$ is not Markov $p$-convex for some $p\in(1,\infty)$.
\item[(4')] $\tree$ does not have non-trivial infrasup-fork convexity.
\end{enumerate}
\end{coro}

The proof of the non-trivial implication $(4)\implies (1)$ in \cite{LNP09} is based on a delicate analysis of certain edge-colorings of trees and their relation to the $\ell_p$-distortion of trees. In a nutshell, if $(4)$ holds then $\cdist{\ell_p}(\tree)$ is unbounded, which in turn forces a certain coloring parameter to vanish. The vanishing of the coloring parameter is then utilized to show the presence of binary trees in $\tree$ with arbitrarily good distortion. Since it is sufficient to assume $(4')$ to guarantee that $\cdist{\ell_p}(\tree)$ is unbounded (e.g. via Proposition \ref{prop:bin-tree-dist}), the implication $(4')\implies (1)$ follows from the same edge-coloring based argument. 

\begin{rema}
It was shown in \cite{MendelNaor13} that the equivalence between $(1)$ and $(2)$ in Corollary \ref{cor:dicho} does not hold if the target space is an arbitrary metric space. Also, one can prove the analogue of Proposition \ref{prop:saturation} for binary trees using infrasup-fork convexity.
\end{rema}


The relaxations of the $p$-fork inequality that we considered in this section are formally significantly weaker, and it would be interesting to identify more examples of metric spaces satisfying these seemingly very weak inequalities. In fact, these examples must be found outside the realm of tree metrics and Banach spaces. 

Note that if an infinite metric tree $\tree$ admits an equivalent metric that satisfies the infrasup $p$-fork inequality \eqref{eq:super-relaxed-q-fork}, then Theorem \ref{thm:fork} says that $\tree$ has non-trivial infrasup-fork convexity, and by Proposition \ref{prop:bin-tree-dist} we have $\sup_{k\in \bN}\cdist{\tree}(\bin{k})=\infty$. Then it follows from the dichotomy in \cite{LNP09} that $\cdist{L_p}(\tree)\le \cdist{L_2}(\tree)<\infty$, and thus $\tree$ admits an equivalent metric that satisfies the $r$-fork inequality where $r=\max\{2,p\}$. Therefore, when $p\in[2,\infty)$ an infinite metric tree admits an equivalent metric that satisfies the infrasup $p$-fork inequality if and only if it admits an equivalent metric that satisfies the $p$-fork inequality. 

In the Banach space setting, we consider an alternative definition of uniform convexity via the following modulus which is a local analogue of the asymptotic modulus $\bar{\beta}$ naturally linked to property $(\beta)$:
\begin{equation}
\beta_\banX(\vep)\eqd \inf \Big\{\max_{i\in \{1,2\}} \Big\{1-\bnorm{\frac{z-x_i}{2}}_\banX \Big\} \colon \norm{z}_\banX, \norm{x_1}_\banX, \norm{x_2}_\banX \le 1, \norm{x_1-x_2}_\banX\ge \vep \Big\}.
\end{equation}
It is easily verified that for all $\vep\in(0,2)$, $\beta_\banX(\vep)>0$ if and only if there exists $\delta>0$ such that for all $z,x,y\in B_\banX$, if $\norm{x-y}_\banX\ge \vep$ then $\min\{\norm{\frac{z-x}{2}}_\banX, \norm{\frac{z-y}{2}}_\banX\}\le 1- \delta$.

The modulus $\beta_\banX$ is a ``fork variant", inspired by property $(\beta)$, of the classical $2$-point modulus of uniform convexity $\delta_\banX$:
\begin{equation}
\delta_\banX(\vep)\eqd \inf \Big\{1-\bnorm{\frac{x+y}{2}}_\banX \colon \norm{x}_\banX, \norm{y}_\banX \le 1, \norm{x-y}_\banX\ge \vep \Big\}.
\end{equation}
\begin{lemm}
\label{lem:2pt-3pt}
For all $\vep\in (0,2)$, 
\begin{equation}
\delta_\banX\Big(\frac{\vep}{2}\Big)\le \beta_\banX(\vep) \le 2 \delta_\banX(\vep).
\end{equation}
\end{lemm}

\begin{proof}
Let $z,x_1,x_2\in B_\banX$ and $\norm{x_1-x_2}_\banX\ge \vep$. If $\norm{x_1+z}_\banX\ge \frac{\vep}{2}$ then $\norm{\frac{x_1 - z}{2}}_\banX\le 1-\delta_\banX(\frac{\vep}{2})$. Otherwise, $\norm{x_1+z}_\banX<\frac{\vep}{2}$ and this implies that $\norm{x_2+z}_\banX \ge \norm{x_2-x_1}-\norm{x_1+z}> \vep-\frac{\vep}{2}=\frac{\vep}{2}$.
Therefore, $\norm{\frac{x_2 - z}{2}}_\banX\le 1-\delta_\banX(\frac{\vep}{2})$. In any case, $\max_{i\in \{1,2\}} \Big\{1-\bnorm{\frac{z-x_i}{2}}_\banX\Big\}\ge \delta_\banX(\frac{\vep}{2})$, and the left-most inequality is proved.

For the right-most inequality, let $x,y\in B_\banX$ and $\norm{x-y}_\banX\ge \vep$. Take $z=-x$. By definition of $\beta_\banX$, either $\norm{\frac{x+y}{2}}_\banX\le 1-\beta_\banX(\vep)$ or $\norm{x}_\banX \le 1- \beta_\banX(\vep)$. In the former case, there is nothing to do. In the latter case, $\norm{\frac{x+y}{2}}_\banX \le \frac12 (1-\beta_\banX(\vep))+\frac12 = 1-\frac12 \beta_\banX(\vep)$, and the conclusion follows. 
\end{proof}

It follows immediately from Lemma \ref{lem:2pt-3pt} that a Banach space $\banX$ is uniformly convex if and only if $\beta_\banX(\vep)>0$ for all $\vep>0$. Note in passing that this provides a rather direct proof that uniformly convex spaces have property $(\beta)$. Quantitatively, $\banX$ is uniformly convex with power type $p$ if and only if $\beta_\banX(\vep)\gtrsim \vep^p$. It is also easy to see that if $\banX$ supports the infrasup $p$-fork inequality \eqref{eq:super-relaxed-q-fork}, then $\beta_\banX(\vep)\gtrsim \vep^p$, and thus $\banX$ is uniformly convex with power type $p$ by Lemma \ref{lem:2pt-3pt}. Consequently, by \cite{BCL94} $\banX$ is $p$-uniformly convex, and by \cite[Lemma 2.3]{MendelNaor13} it satisfies the $p$-fork inequality. Thus for Banach spaces, the $p$-fork inequality and the infrasup $p$-fork inequality are equivalent up to the value of the constants involved. 

For the sake of completeness, we provide a more direct proof of the fact above which uses neither \cite{BCL94} nor \cite{MendelNaor13} and for which it is easier to keep track of the value of the constant.

\begin{lemm}
Let $\banX$ be a Banach space. If $\beta_\banX(t)\ge \frac{1}{c}t^p$ then the infrasup $p$-fork inequality \eqref{eq:super-relaxed-q-fork} holds in $\banX$ with constant $c^{1/p}4^{-1}$.
\end{lemm}

\begin{proof}
Assume that $\beta_\banX(t)\ge \frac{1}{c}t^p$ and let $w, z, x, y \in \banX$. Since the distance in $\banX$ is translation invariant, we may assume $z=0$. Also, by scale invariance of \eqref{eq:super-relaxed-q-fork} we can assume that $w, x, y \in B_\banX$. Thus \eqref{eq:super-relaxed-q-fork} reduces to
\begin{equation}
\label{eq:aux10}
\frac{1}{2^p}\min\{\norm{w-x}^p,\norm{w-y}^p\}+\frac{\norm{x-y}^p}{4^{p}K^p}\le 1.
\end{equation} 

Now observe that $\min\{\frac{\norm{w-x}^p}{2^p}; \frac{\norm{w-x}^p}{2^p}\} \le \min\{\frac{\norm{w-x}}{2}; \frac{\norm{w-y}}{2}\}$ whenever $w,x, y \in B_\banX$.
Therefore, \eqref{eq:aux10} follows from the fact that by definition of $\tilde{\delta}_\banX$ it holds $\min\{\frac{\norm{w-x}}{2}; \frac{\norm{w-y}}{2}\}\le 1-\frac{1}{c}\norm{x-y}^p$, since without loss of generality we may assume that $\norm{x-y}>0$.
\end{proof}


\section{A characterization of non-negative curvature}\label{sec:non-negative}

Recall that a geodesic metric space has \emph{non-negative curvature} if for all $x,y,z\in \met X$ and $m_{xy}$ a midpoint of $x$ and $y$,
\begin{equation}
2\dX(z,m_{xy})^2+\frac{\dX(x,y)^2}{2}\ge \dX(z,x)^2+\dX(z,y)^2
\end{equation}
Austin and Naor \cite{AustinNaor} showed that a geodesic metric space $\metXd$ with non-negative curvature satisfies the $2$-fork inequality with constant $K=1$. In this section we prove the missing implication in Theorem \ref{thm:D} adapting an argument of Lebedeva and Petrunin \cite{LebedevaPetrunin10} which is used to characterize non-negative curvature in terms of a certain fork inequality.

\begin{prop}
If a geodesic metric space $\metXd$ satisfies the $2$-fork inequality with constant $K=1$, then $\met X$ has non-negative curvature.
\end{prop}

\begin{proof}
Let $x,y,z\in \met X$ and let $m_{xy}$ be a midpoint of $x$ and $y$. Since $\met X$ is geodesic there exists a geodesic connecting $m_{xy}$ and $z$, and for each $n \ge 1$, a point $z_n$ on this geodesic such that $\dX(m_{xy},z_n)=\frac{\dX(m_{xy},z)}{2^n}$. Set $z_0=z$ and  $\alpha_n$ to be such that 
\begin{equation}
\alpha_n\dX(z_n,m_{xy})^2=\dX(z_n,x)^2+\dX(z_n,y)^2-\frac{\dX(x,y)^2}{2}
\end{equation}

Note that it is sufficient to show that $\alpha_0\le 2$ in order to show that $\metXd$ has non-negative curvature. Observe first that the $2$-fork inequality with constant $K=1$ applied to $z_n,x,y,m_{xy}$ gives that for all $n\ge 1$,
\begin{equation}
\dX(z_n,x)^2+\dX(z_n,y)^2+\frac{\dX(x,y)^2}{2}\le 4 \dX(m_{xy},z)^2+2\dX(m_{xy},x)^2+2\dX(m_{xy},y)^2
\end{equation}
and thus 
\begin{align*}
\dX(z_n,x)^2+\dX(z_n,y)^2-\frac{\dX(x,y)^2}{2} & \le 4 \dX(m_{xy},z_n)^2+2\dX(m_{xy},x)^2+2\dX(m_{xy},y)^2-\dX(x,y)^2\\
									& =  4 \dX(m_{xy},z_n)^2 + 2\frac{\dX(x,y)^2}{4}+2\frac{\dX(x,y)^2}{4}-\dX(x,y)^2\\
									& =  4 \dX(m_{xy},z_n)^2,
\end{align*}
which means that $\alpha_n\le 4$ for all $n\ge 1$.

Now if we subtract $\frac{\dX(x,y)^2}{2}$ to the $2$-fork inequality with constant $1$ applied to $z_{n+1},x,y,z_{n}$, we have 
\begin{equation*}
\dX(z_{n+1},x)^2+\dX(z_{n+1},y)^2+2\dX(z_{n+1},z_{n})^2-\frac{\dX(x,y)^2}{2}\ge \frac{\dX(z_n,x)^2}{2}+\frac{\dX(z_n,y)^2}{2}-\frac{\dX(x,y)^2}{4}.
\end{equation*}
Hence,
\begin{equation*}
\alpha_{n+1}\dX(z_{n+1},m_{xy})^2\ge \frac{\alpha_n}{2}\dX(z_n,m_{xy})^2-2\dX(z_{n+1},z_{n})^2
\end{equation*}
which ultimately gives
\begin{equation*}
\alpha_{n+1}\frac{\dX(z,m_{xy})^2}{2^{2(n+1)}}\ge \frac{\alpha_n}{2}\frac{\dX(z,m_{xy})^2}{2^{2n}}-2\dX(z_{n+1},z_{n})^2.
\end{equation*}
Observe now that since the  $z_n$'s are on same geodesic
$$\dX(z_{n+1},z_n)=\frac{\dX(m_{xy},z)}{2^n}-\frac{\dX(m_{xy},z)}{2^{n+1}}=\frac{\dX(m_{xy},z)}{2^{n+1}}.$$

Then,

$$\alpha_{n+1}\ge \frac{2^{2n+2}}{2^{2n+1}}\alpha_n-\frac{2\cdot 2^{2n+2}}{2^{2n+2}}\ge 2\alpha_n-2.$$

Assume that $\alpha_0>2$. Then a simple induction gives that $\alpha_n \geq 2^n(\alpha_0-2)+2$ and hence $\lim_n \alpha_n=\infty$, contradicting the fact that $\alpha_n\le 4$. Therefore $\alpha_0\le 2$ and the conclusion follows.
\end{proof}

\section{Concluding remarks and open problems}
\label{sec:conclusion}
Our work raises a myriad of natural questions and problems. We will highlight a few of them we feel are particularly important and most likely challenging.  

It follows from Corollary \ref{cor:quotient} that umbel $p$-convexity is stable under uniform homeomorphisms between Banach spaces. Because of this fact, umbel convexity cannot settle the metric characterization of Banach spaces with property $(\beta_p)$. Indeed, Kalton showed  \cite{Kalton13} that given a sequence $\{\ban F_n\}_{\ge 1}$ that is dense in the Banach-Mazur compactum of finite-dimensional spaces, the Banach space $\ban C_p\eqd (\sum_{n=1}^\infty \ban F_n)_{\ell_p}$ is uniformly homeomorphic to $\ban K_p\eqd (\sum_{n=1}^\infty \ban F_n)_{\ban T_p}\oplus (\sum_{n=1}^\infty \ban F_n)_{\ban T_p}$ where $\ban T_p$ is the $p$-convexification of Tsirelson space $\ban T$. Therefore $\ban K_p$ is umbel $p$-convex since $\ban C_p$ has property $(\beta_p)$, but Kalton observed that $\ban K_p$ does not admit an equivalent norm that is asymptotically uniformly convex with power type $p$, and by \cite{DKR16} does not admit an equivalent norm with property $(\beta_p)$. The space $\ban K_p$ is thus an example of a Banach space that is umbel $p$-convex and that does not admit an equivalent norm with property $(\beta)$ with power type $p$. This is in stark contrast with the renorming Theorem \ref{thm:MN-LNP} for Markov convexity. The stochastic apparatus of Markov convexity is a powerful tool that is dearly missed in the asymptotic setting, and a new idea is needed to solve the following problem.

\begin{prob}
\label{pb:betap}
For a given $p\in(1,\infty)$, find a metric characterization of the class of Banach spaces admitting an equivalent norm with property $(\beta_p)$.
\end{prob}

Interestingly, it was shown in \cite{DKR16} that $\ban K_p$ admits for every $\vep>0$ an equivalent norm with property $(\beta_{p+\vep})$ and the next problem arises naturally.
 
\begin{prob}
\label{pb:renorming}
If Banach space is umbel $p$-convex for some $p\in(1,\infty)$, does it admit for all $q>p$ an equivalent norm with property $(\beta_q)$?
\end{prob}

Our work shows that if a Banach space is umbel $p$-convex for some $p\in(1,\infty)$, then it admits an equivalent norm with property $(\beta_q)$ for some $q>1$. The difficulty in solving Problem \ref{pb:betap} and Problem \ref{pb:renorming} stems from the fact that the renorming theory for spaces with property $(\beta)$ is not fully grasped yet as it currently goes through the much better understood asymptotic uniformly convex/smooth renorming theories.

A tentatively more tractable, and somewhat related problem, is a local analogue of Problem \ref{pb:renorming}.

\begin{prob}
\label{pb:renorming-local}
If a Banach space is fork $p$-convex for some $p\in [2,\infty)$, does it admit for $q=p$ (or more modestly for all $q>p$) an equivalent norm which is $q$-uniformly convex?
\end{prob}

In Section \ref{sec:examples} we showed that $(\bH(\ell_2),\sd_{cc})$ is infrasup-umbel $2$-convex. In particular, this implies that $\cdist{\bH(\ell_2)}(\tree^\omega_k)=\Omega \big( \sqrt{\log k} \big)$, and this is optimal by Bourgain's tree embedding (see Proposition \ref{prop:am-dist}). This has to be contrasted with the fact that $(\bH(\ell_2),\sd_{cc})$ is only Markov $4$-convex, and a Markov convexity-based argument gives $\cdist{\bH(\ell_2)} \big( \bin{k} \big)=\Omega \big( (\log k)^{1/4} \big)$. This lower bound is suboptimal since S. Li \cite{Li16} proved, using a refinement of an argument of Matousek \cite{Matousek99}, that $\cdist{\bH(\ell_2)}\big( \bin{k} \big)=\Omega \big( \sqrt{\log k} \big)$, and this latter bound is optimal by Bourgain's tree embedding.
By Theorem~\ref{thm:Sean} and Proposition~\ref{prop:Markov-fork} $(\bH(\ell_2),\sd_{cc})$ is infrasup-fork $p$-convex for all $p \geq 4$, and it would be interesting to compute its exact infrasup-fork convexity.

\begin{prob}
\label{pb:forkcotypeH}
Is $(\bH(\omega_\banX),\sd_{cc})$ infrasup-fork $p$-convex whenever $\banX$ is $p$-uniformly convex?
\end{prob}

If Problem \ref{pb:forkcotypeH} has a positive answer, then the notion of infrasup-fork convexity would be a metric invariant that could detect the right order of magnitude for the distortion required to embed binary trees into the infinite Heisenberg group, something that Markov convexity is unable to achieve.

In the proof of Theorem \ref{thm:Heisenberg}, we showed that $(\bH(\omega_\banX),\sd_{cc})$ admits an equivalent quasi-metric satisfying the $2p$-fork inequality \eqref{eq:qfork} whenever $\banX$ is $p$-uniformly convex. The following asymptotic problem remains open.

\begin{prob} \label{prob:Heisenbergumbel}
Does $(\bH(\omega_\banX),\sd_{cc})$ admit an equivalent quasi-metric satisfying the $p$-umbel inequality whenever $\banX$ has property $(\beta_p)$? More generally, is $(\bH(\omega_\banX),\sd_{cc})$ umbel $p$-convex whenever $\banX$ has property $(\beta_p)$?
\end{prob}

The scale-invariant parallelogram convexity inequality \eqref{eq:UCmod} defining $p$-uniform convexity in Banach spaces has a natural analogue in Heisenberg groups, and the proof of Theorem \ref{thm:Heisenberg} goes through establishing this inequality. The difficulty in adapting the proof to solve Problem \ref{prob:Heisenbergumbel} exactly lies in the fact that no scale-invariant ``parallelogram" inequality exists for property $(\beta_p)$. 

The reason why $\bH(\ell_2)$ cannot be Markov $p$-convex for any $p<4$ comes from the fact that certain Laakso graphs, which are known not to have non-trivial Markov convexity, can be embedded well enough in $\bH(\bR)$, and hence in $\bH(\ell_2)$. It seems possible that Laakso or diamond graph constructions could have non-trivial infrasup-fork convexity and thus infrasup-fork convexity would be a metric invariant capable of preventing bi-Lipschitz embeddings of trees into diamond like structures. It is worth pointing out that it was proved by Ostrovskii \cite{Ostrovskii14} (see also \cite{LNOO18}) that binary trees do not embed equi-bi-Lipschitzly into diamond graphs. Note also that diamond convexity is a metric invariant that prevents bi-Lipschitz embeddings of diamond or Laakso graphs into trees, since it was proved in \cite{EMN} that trees are diamond $2$-convex.

\begin{prob} Let $\gra G_k$ be one of the following graphs: the diamond graph $\dia_k$, the Laakso graph $\sL_k$, or their countably branching versions $\dia^\omega_k$ and $\sL^\omega_k$, respectivelly. 
Are the parameters $\sup_{k\in \bN} \Pi^{isf}_p(\gra G_k)$, $\sup_{k\in \bN} \Pi^{isu}_p(\gra G_k)$, or $\sup_{k\in \bN} \Pi^u_p(\gra G_k)$ finite for some $p<\infty$?
\end{prob}

It would be very interesting to exhibit examples of metric spaces that admit an equivalent metric satisfying the infrasup $p$-fork inequality but with no equivalent metric satisfying the $p$-fork inequality. In light of Theorem \ref{thm:Heisenberg}, Proposition 2.3 in \cite{Li16} (or the proof of Theorem \ref{thm:Sean}), and the discussion above, a natural candidate for $p=2$ is the infinite Heisenberg group.

\begin{prob}\label{pb:heisenberg-is2fork}
Does $(\bH(\ell_2),\sd_{cc})$ admit an equivalent (quasi)-metric satisfying the infrasup $2$-fork inequality?
\end{prob}

Finally, we do not know whether Markov $p$-convexity implies fork $p$-convexity. Loosely speaking, the issue is that the left-hand side of the Markov $p$-convexity inequality involves an average over all levels of the binary tree, while the left-hand side of the fork $p$-convexity inequality involves an average over dyadic levels\footnote{For similar reasons, we do not know the relationship between the previously mentioned Poincar{\'e} inequality on binary trees \cite[page 382]{LMN02} and Markov convexity or fork convexity.}.

\begin{prob}\label{pb:markov->fork}
Does Markov $p$-convexity imply fork $p$-convexity?
\end{prob}

\medskip\noindent
{\bf Acknowledgements.}  We would like to thank the anonymous referee for many meaningful comments and for suggesting that we adopt a more intuitive terminology for some of the inequalities and metric invariants considered.

\appendix

\section{Table of Inequalities}
\label{app:table}
For the convenience of the reader, we summarize in the following table the main inequalities introduced or recalled in the paper. The table is organized so that the following three facts hold:
\begin{itemize}
    \item An inequality in row $i$ column $j$ implies the inequality in row $k$ column $j$ for $k > i$ (with the exception of Markov $p$-convexity implying fork $p$-convexity, see Problem~\ref{pb:markov->fork}).
    \item A point-inequality in row $i$ column $j$ implies the Poincar{\'e} inequality in row $i$ column $j+1$.
    \item A local inequality in row $i$ column $j$ implies the asymptotic inequality in row $i$ column $j+2$.
\end{itemize}

\begin{table}[h]
\resizebox{\textwidth}{!}{%
\begin{tabular}{|c|c|c|c|} \hline
\multicolumn{2}{|c|}{\bf Local} & \multicolumn{2}{|c|}{\bf Asymptotic} \\ \hline
\bf $4$-point inequality & \bf Poincar{\'e} inequality & \bf $\omega$-point inequality & \bf Poincar{\'e} inequality \\ \hline\hline
$p$-fork inequality \eqref{eq:qfork} & Markov $p$-convexity \eqref{eq:p-Markov}  && \\ \hline
relaxed $p$-fork inequality \eqref{eq:relaxed-q-fork} & fork $p$-convexity \eqref{eq:forkqconvex} & $p$-umbel inequality \eqref{eq:pumbelmetric} & umbel $p$-convexity \eqref{eq:umbel-p-convex} \\ \hline
&& sup $p$-umbel inequality \eqref{eq:relaxed-p-umbel} & sup-umbel $p$-convexity \eqref{eq:relaxed-umbel-p} \\ \hline
infrasup $p$-fork inequality \eqref{eq:super-relaxed-q-fork} & infrasup-fork $p$-convexity \eqref{eq:forkcotypeq} & infrasup $p$-umbel inequality \eqref{eq:superrelaxed-p-umbel} & infrasup-umbel $p$-convexity \eqref{eq:umbelcotypep} \\ \hline
\end{tabular}}
\end{table}

\bibliographystyle{alpha}


\end{document}